\documentclass[preprint,3p]{elsarticle}

\usepackage{lineno,hyperref}
\modulolinenumbers[5]

 \journal{Applied Numerical Mathematics}

\usepackage{graphicx}
\usepackage{amsmath}
\usepackage{amssymb}
\usepackage{float}
\usepackage{rotating}
\usepackage{multirow}
\usepackage{color}

\newtheorem{theorem}{Theorem}
\newtheorem{lemma}{Lemma}

\newtheorem{remark}{Remark}%

\newtheorem{definition}{Definition}%

\newtheorem{hypothesis}{Hypothesis}
\newtheorem{strategy}{Adaptive scheme}
\newtheorem{sampling}{Sampling method}

\newproof{proof}{Proof}

\begin{document}


\begin{frontmatter}

\title{Weak variable step-size schemes for stochastic differential equations 
based on controlling conditional moments
\tnoteref{fnSup}
}

\tnotetext[fnSup]{Supported by Universidad de Concepci\'on project VRID-Enlace 218.013.043-1.0}

\author[UdeC]{Carlos M. Mora}
\ead{cmora@ing-mat.udec.cl}

\author[ICIMAF]{Juan Carlos Jimenez}
\ead{jcarlos@icimaf.cu}

\author[UdeC]{Monica Selva}
\ead{mselva@ing-mat.udec.cl}

\address[UdeC]{Departamento de Ingenier\'{\i}a Matem\'{a}tica, 
Facultad de Ciencias F\'{\i}sicas y Matem\'{a}ticas, Universidad de Concepci\'{o}n,
Casilla 160 C, Concepci\'{o}n, Chile.}

\address[ICIMAF]{Departamento de Matem\'atica Interdisciplinaria,
Instituto de Cibern\'{e}tica, Matem\'{a}tica y F\'{i}sica, 
La Habana, Cuba
}

\begin{abstract}
We address the weak numerical solution of stochastic differential equations driven by independent Brownian motions
(SDEs for short).
This paper develops a new methodology to design adaptive strategies for determining automatically the step-sizes of
the numerical schemes that compute the mean values of smooth functions of the solutions of SDEs.
First, 
we introduce a general method for constructing variable step-size weak schemes for SDEs, 
which is based on controlling the match between the first conditional moments  of the increments of the numerical integrator 
and the ones corresponding to an additional weak approximation.
To this end,
we use certain local discrepancy functions that do not involve sampling random variables. 
Precise directions for designing suitable discrepancy functions and for selecting starting step-sizes are given.
Second, 
we introduce a variable step-size Euler scheme,
together with a variable step-size second order weak scheme
via extrapolation. 
Finally,  
numerical simulations are presented to show the potential of the introduced variable step-size strategy 
and the adaptive scheme to overcome  known instability problems of the conventional fixed step-size schemes in the computation of diffusion functional expectations.
\end{abstract}

\begin{keyword}

Adaptive time-stepping \sep stochastic differential equation \sep numerical solution
\sep  weak error \sep Monte-Carlo method \sep  Euler scheme

\MSC[2010] 65C30 \sep 65C05 \sep 60H35 \sep 60H10

\end{keyword}

\end{frontmatter}



\section{Introduction}

In this paper,
we introduce a new methodology to design adaptive strategies for 
the computation of expected values of functionals of It\^o stochastic differential equations (SDEs for short).
We focus on the automatic calculation of  the mean value of $ \varphi \left(X_{T}\right) $
with $\varphi : \mathbb{R}^{d} \rightarrow \mathbb{R}$  smooth, $T > 0$,
and 
\begin{equation}
\label{eq:SDE}
X_{t}
=
X_{0} 
+ \int_{0}^{t}b\left( s, X_{s}\right)  ds+ \sum_{k=1}^{m}\int_{0}^{t}\sigma_k \left( s, X_{s}\right)  dW^{k}_{s} 
\hspace{1cm}
\text{for all }t \in \left[ 0, T \right] .
\end{equation}
Here,
$b, \sigma_1, \ldots, \sigma_m :  \left[ 0, T  \right] \times \mathbb{R}^{d}\rightarrow\mathbb{R}^{d}$  are 
locally Lipschitz smooth functions,
$W^1,\ldots,W^m$ are independent Brownian motions on 
a filtered complete probability space 
$\left( \Omega ,\mathfrak{F}, \left(\mathfrak{F}_{t}\right) _{t\geq 0},\mathbb{P}\right) $,
and 
the unknown $X_t $ is  an adapted $\mathbb{R}^{d}$-valued stochastic process with continuous trajectories.
Let $ \left( Y_n \right)_{n\geq 0}$ be a one-step numerical scheme solving \eqref{eq:SDE} at nodes $\left( \tau_n \right)_{n\geq 0} $,
where $\left( \tau_n \right)_{n\geq 0} $ is a random discretization of $\left[ 0, T \right] $
(see Section \ref{subsec:BasicAssumptions} for details).
This article addresses the automatic selection  of the step-sizes $\tau_{n+1} - \tau_{n}$ for which,  roughly speaking, 
$Y_n$ is an appropriate weak approximation of $X_{ \tau_{n}}$.

It is known that 
many complex initial value problems for ordinary differential equations (ODEs) are solved efficiently 
by controlling effectively the local discretization  errors 
--the error committed in one step of the numerical integration--
(see, e.g., \cite{Shampine2005,Shampine2003}).
In contrast, 
variable time-stepping schemes that control  the global error are generally considered computationally expensive
(see, e.g., \cite{Shampine2005} for a deeper discussion). 
In the numerical solution of SDEs,
many variable step-size strategies for schemes that approximate the trajectories of the solution of  \eqref{eq:SDE}
(strong approximations) have been proposed  by extending the local error trajectory-based  approaches developed for integrating ODEs. 
For instance, there is
the strategy of halving or doubling the current step size (see, e.g.,  \cite{GainesLyons1997}), 
and other strategies based on embedded methods (see, e.g.,  \cite{BurrageBurrage2002}), 
and on predictive-integral (PI) controllers (see, e.g., \cite{BurrageHerdianaBurrage2004,IlieJacksonEnright2015}).
On the other hand,
strong integrators that control adaptively the  numerical stability, by using the drift component of the SDE,
have been developed by, e.g., \cite{LambaMattinglyStuart2007,FangGiles2020,KellyLord2016,KellyLord2022}.
The coupling between the standard Multilevel Monte Carlo method (see, e.g., \cite{Giles2015,Giles2018} for reviews)
and  the schemes controlling adaptively the  numerical stability has been developed in, e.g., \cite{FangGiles2020,KellyLord2016}.

For weak numerical integration of SDEs, 
Szepessy, Tempone and Zouraris \cite{Szepessy2001} introduced 
two adaptive time-stepping strategies 
for the Euler-Maruyama scheme that are based on the computation of leading-order terms of a-posteriori  estimates
of the global weak error via Monte Carlo simulations
(see also, e.g., \cite{MerleProhl2021,MoonSzepessyTemponeZouraris2005,Mordecki2008} for further developments).
The algorithms of \cite{Szepessy2001} start by 
sampling the Euler-Maruyama scheme with an initial time discretization given by the user.
Then,
\cite{Szepessy2001} constructs a recursive sequence of partitions of $\left[ 0, T \right]$
by halving the step size in some nodes of the previous time discretization.
In \cite{Hoel2014} the global variable step-size schemes introduced by \cite{Szepessy2001}
are implemented to be sampled by means of a Multilevel Monte Carlo method.

The adaptive methods based on a-posteriori estimates
use estimations of global errors to select the step-sizes, 
a procedure that has experienced difficulties in dealing with ODEs.
Motivated partially by the fact that the codes commonly used to solve initial value problems for ODEs 
are based on controlling the local error,
this paper develops the design of adaptive algorithms for selecting the step-sizes of weak schemes for SDEs 
that are based on local discrepancy functions. 
In this direction, R\"ossler \cite{Rossler2004} extended straightforwardly 
the conventional step-size control of embedded schemes for ODEs 
to get  a variable deterministic time discretization  $\left( \tau_n \right)_{n=0,\ldots,N} $
(see also, e.g., \cite{KupperLehnRossler2007,Valinejad2010}).
To this end, 
\cite{Rossler2004} combines a pair of embedded stochastic Runge-Kutta schemes
with samples generated by Monte Carlo simulations
to estimate the ``local error"
$
\left\vert 
 \mathbb{E} \left(  \varphi \left(    X_{\tau_{n+1}} \left( \tau_{n} ,  Y_{n}  \right)  \right)   \diagup \mathfrak{F}_{\tau_n}  \right)
 -
\mathbb{E} \left(  \varphi \left(  Y_{n+1}\right)  \diagup \mathfrak{F}_{\tau_n} \right)
\right\vert ,
$
where
$ \left( X_{t} \left( \tau_{n} ,  Y_{n}  \right) \right)_{ t \in  \left[  \tau_{n}, T \right] }$ is  from now on the solution of 
\eqref{eq:SDE} with initial condition $Y_{n} $ at $\tau_{n}$.
In \cite{Rossler2004,KupperLehnRossler2007,Valinejad2010} the starting step-size is given by the user.
In the framework of the continuous-discrete estimation problem of the filtering theory,
\cite{Jimenez2019} develops an adaptive filter of minimum variance
that uses, 
between consecutive observation times,
the weak local linearization scheme given in \cite{JimenezMoraSelva2017}
(see also \cite{Carbonell2006}),
together with an adaptive strategy controlling the predictions for the first two conditional moments of the continuous state equation that does not involve sampling random variables.

In this paper,
we develop  a general methodology for determining automatically  the step-sizes 
$\tau_{n+1} - \tau_n$ of the scheme $ \left( Y_n \right)_{n\geq 0}$  solving \eqref{eq:SDE}
so that some measure of  a good match between the conditional distributions of  $Y_{n+1}$ and $ X_{\tau_{n+1}} \left( \tau_{n} ,  Y_{n}  \right)$, given $Y_{n} $,
is smaller than a tolerance provided by the user.
Namely,
inspired by \cite{Jimenez2019} and by the variable step-size strategies based on embedded schemes for ODEs, 
in Section \ref{sec:GeneralStrategyP} we introduce 
a new method for constructing  variable step-size weak schemes for SDEs.
The heart of the method is to control the matching between the first conditional moments of embedded pairs of weak approximations  by means of discrepancy functions that do not involve sampling random variables.
This allows us to design solvers computationally much faster than those mentioned above with a promising performance in practical problems. 
Roughly speaking, we determine automatically the step-size $\tau_{n+1} - \tau_{n}$ of the one-step numerical scheme $ \left( Y_n \right)_{n\geq 0}$
by keeping a weighted norm of estimates of the conditional expectations 
$
 \mathbb{E} \left(  \hat{Y}_{n+1}   \diagup \mathfrak{F}_{\tau_n} \right)
 -
  \mathbb{E} \left(  Y_{n+1}  \diagup \mathfrak{F}_{\tau_n} \right)
$
and
\[
 \mathbb{E} \left(  \left(   \hat{Y}_{n+1}   -   Y_{n} \right)  \left(   \hat{Y}_{n+1}   -   Y_{n} \right)^{\top}   \diagup \mathfrak{F}_{\tau_n} \right)
 -
  \mathbb{E} \left(  \left(   Y_{n+1}  -   Y_{n} \right)  \left(   Y_{n+1}  -   Y_{n} \right)^{\top}   \diagup \mathfrak{F}_{\tau_n} \right)
\]
within the range of a threshold given by the user,
where $\hat{Y}_{n+1} $ is an auxiliary weak approximation of $X_{ \tau_{n+1} } \left( \tau_{n} ,  Y_{n}  \right)$.
We provide directions for designing suitable discrepancy functions.

In Section \ref{sec:VariableStepSizeEM},
we show how to use the general method introduced in Section \ref{sec:GeneralStrategyP}.
We consider the Euler-Maruyama scheme
\begin{equation}
\label{eq:Euler-Maruyama}
Y_{n+1} 
=
Y_{n} 
+  b \left( \tau_n ,  Y_{n} \right) \left( \tau_{n+1}  - \tau_{n} \right) 
+ \sum_{k=1}^{m} \sigma_k \left(  \tau_n ,  Y_{n}  \right)  \left( W^k_{\tau_{n+1}} - W^k_{\tau_n} \right) ,
\end{equation}
where 
 $\left( \tau_n \right)_{n\geq 0} $ is a random discretization of $\left[ 0, T \right] $
(see Section \ref{subsec:BasicAssumptions} for details).
Then,
in Subsections \ref{sec:BasicLossFunction} and \ref{sec:adaptiveEuler-Maruyama} 
we design an adaptive algorithm for selecting the step-size $\tau_{n+1}-\tau_n$ of \eqref{eq:Euler-Maruyama},
which is in itself  important for the applications.
To this end,
as  the additional approximation $\hat{Y}_{n+1}$ we select a second order  weak It\^o-Taylor approximation
 of $X_{ \tau_{n+1} } \left( \tau_{n} ,  Y_{n}  \right)$,
 i.e., $\hat{Y}_{n+1}$ is the approximation \eqref{eq:SecondOrderSchemeN} given below.
The step-sizes $\tau_{n+1} - \tau_{n}$  are computed automatically  
without sampling the random variables  $ Y_{n+1}$ and $  \hat{Y}_{n+1} $,
and without employing any accept/reject algorithm
(usual  in the adaptive integrators for ODEs).
Subsection \ref{subsec:Extrapolation} provides a variable step-size second order weak scheme.
In the spirit of the local extrapolation procedure used in ODEs (see, e.g., \cite{HairerNorsettWanner1993,Shampine1994}),
we estimate  $\tau_{n+1}- \tau_{n}$ as in Subsection \ref{sec:adaptiveEuler-Maruyama},
but we compute  the numerical solution of \eqref{eq:SDE} from $ \tau_{n}$ to 
$\tau_{n+1}$ with the higher-order numerical method  $\hat{Y}_{n+1}$.

The choice of  the initial step-size $\tau_1$ is a critical stage in 
variable step-size methods for ODEs  
(see, e.g., \cite{GladwellShampineBrankin1987,HairerNorsettWanner1993}).
In Section \ref{sec:Starting},
we introduce a general procedure for the automatic selection of the starting step-size
based on controlling the size of the first two moments of $ Y_{1} - Y_0$.

We illustrate the performance of the new adaptive schemes by means of  numerical experiments with four  benchmark SDEs.
The new adaptive schemes reduce appropriately the step-sizes of the schemes as the tolerances parameters become smaller,
and greatly overcome the accuracy and stability of the Euler and second order Taylor schemes with fixed step-size.
The very good performance of the new adaptive strategy
is brought out in a comparison of the new adaptive strategy 
with the stochastic strategy given by  \cite{Szepessy2001}.


\section{ Preliminaries}


\subsection{Notation}

We use the symbols $  \left\Vert \cdot  \right\Vert_{ \mathbb{R}^d }$ and $\left\Vert \cdot \right\Vert_{\ell^{p}}$ 
to denote weighted  $\ell^{p}$  norms on $\mathbb{R}^d$, with $p   \in   \mathbb{N}  \cup \left\{ + \infty \right\}$.
We write $\left\Vert \cdot \right\Vert $ for the Euclidean norm on $\mathbb{R}^d$.
The $\left(i , j \right)$ component of the matrix $A$ is denoted by $ A^{i,j}$,
and
we represent the elements of $ \mathbb{R}^{d}$ as column vectors.
By $\left\Vert \cdot \right\Vert_{\mathbb{R}^{d \times d}} $ we mean  a norm 
on the space of all real matrices of order $d \times d$
that satisfies 
$
\left\Vert  \left(  A^{i,j}  \right)_{i,j = 1 , \ldots, d} \right\Vert_{\mathbb{R}^{d \times d}} 
=
\left\Vert \left( \left\vert A^{i,j} \right\vert \right)_{i,j = 1 , \ldots, d} \right\Vert_{\mathbb{R}^{d \times d}} 
$
for any $ \left(  A^{i,j}  \right)_{i,j = 1 , \ldots, d} \in \mathbb{R}^{d \times d}$, and  
\begin{equation}
\label{eq:4.4}
 \left\Vert x \, y ^{\top}  \right\Vert_{\mathbb{R}^{d \times d}} 
 \leq
 \left\Vert x   \right\Vert_{ \mathbb{R}^d }  \left\Vert y   \right\Vert_{ \mathbb{R}^d }
 \qquad  \qquad 
 \text{for all } x, y \in \mathbb{R}^{d}  .
\end{equation}
Examples of $\left\Vert \cdot \right\Vert_{\mathbb{R}^{d \times d}} $,
where the inequality in \eqref{eq:4.4}  becomes an equality,
are the Frobenius norm,  
the element-wise max matrix norm 
$
\left\Vert  \left(  A^{i,j}  \right)_{i,j = 1 , \ldots, d} \right\Vert_{\infty} 
= 
\max_{i,j = 1 , \ldots, d}  \left\vert A^{i,j} \right\vert
$,
and the $\ell_{1,1}$ entry-wise matrix norm 
$
\left\Vert  \left(  A^{i,j}  \right)_{i,j = 1 , \ldots, d} \right\Vert_{1,1} 
= 
\sum_{i,j = 1 , \ldots, d}  \left\vert A^{i,j} \right\vert
$,
provided that  $ \left\Vert  \cdot   \right\Vert_{ \mathbb{R}^d }$ is, respectively, 
the Euclidean norm  $ \left\Vert  \cdot   \right\Vert $, the max norm, and the $\ell_1$ norm.

From now on, 
$K$ (resp. $K\left( \cdot\right) $ and $q$) stands for different non-negative real numbers 
(resp. non-negative increasing functions and natural numbers) 
that are independent of the discretizations of $\left[ 0, T \right]$. 
We use the standard multi-index notation.
In particular, for any multi-index 
$
\alpha = \left( \alpha_1, \ldots, \alpha_d \right) \in \left(\mathbb{Z}_+ \right)^d
$
we set
$
\left\vert \alpha \right\vert = \alpha_1 + \cdots + \alpha_d 
$,
$
\alpha  ! = \alpha_1 !  \cdots \alpha_d !
$,
$
x^{\alpha} = \left( x^1\right)^{\alpha_1} \cdots  \left( x^d \right)^{\alpha_d}
$,
and
$
\partial_x^{ \alpha} = \left(\frac{\partial}{\partial x^1}\right)^{ \alpha_1} \ldots \left(\frac{\partial}{\partial x^d}\right)^{ \alpha_d} 
$.
The space $\mathcal{C}_{P}^{L}\left( \left[ 0, T \right] \times \mathbb{R}^{d},\mathbb{R}\right) $ 
is the set of all 
$
f  :
\left[ 0, T \right] \times \mathbb{R}^{d}\rightarrow \mathbb{R}
$ 
such that 
$ \partial_x^{ \alpha} f $ is continuous and 
$
\left\vert \partial_{x}^{ \alpha } f \left( t,x\right)
\right\vert \leq K  \left( 1+\left\Vert x\right\Vert ^{q}\right) 
$,
for all
$ t \in \left[ 0, T \right]$ and $x\in\mathbb{R}^{d}$,
whenever $\left\vert \alpha \right\vert \leq L$.
The function 
$
f : \mathbb{R}^{d} \rightarrow \mathbb{R}
$
belongs to
$
\mathcal{C}_{P}^{L}\left(\mathbb{R}^{d},\mathbb{R}\right)
$
 if and only if
$
\left( t, x \right) \mapsto f \left( x \right)
$
is in
$
\mathcal{C}_{P}^{L}\left( \left[ 0, T \right] \times \mathbb{R}^{d},\mathbb{R}\right)
$.

\subsection{Basic assumptions on the SDE}

The SDE \eqref{eq:SDE}  has a unique continuous strong solution up to an explosion time,
because 
$b$ and $\sigma_k$ are assumed to be locally Lipschitz functions.
 (see, e.g., \cite{Mao2011,Protter2005}).
We suppose that:

\

\begin{hypothesis}
\label{hyp:SDE}
\
\begin{itemize}
\item[(a)] 
For all $j = 1, \ldots, d$
the functions $b^j, \sigma^j_1, \ldots, \sigma^j_m$ belong to $\mathcal{C}_{P}^{5}\left( \left[ 0, T \right]  \times \mathbb{R}^{d}, \mathbb{R} \right)$
and 
$\partial_t  b^j,  \partial_t \sigma^j_1, \ldots, \partial_t  \sigma^j_m$ 
are in 
$\mathcal{C}_{P}^{1}\left( \left[ 0, T \right]  \times \mathbb{R}^{d},\mathbb{R}\right) $.

\item[(b)] For all $p \in \mathbb{N}$,
$\mathbb{E} \left(  \left\Vert  X_{0} \right\Vert^{p} \right) < + \infty$.

\item[(c)]  
The equation \eqref{eq:SDE}  has a unique continuous strong solution on the interval $\left[ 0, T \right]$.
Furthermore, 
for any $p \in \mathbb{N}$ there exist $q \in \mathbb{N}$ and $K \geq 0$ satisfying
\begin{equation}
 \label{eq:0.1}
\mathbb{E} \left(  \left\Vert   X_t \right\Vert^{p} \right)
 \leq
 K  \left( 1 + \mathbb{E} \left(  \left\Vert   X_{0} \right\Vert^{q } \right) \right) 
 \hspace{2cm}
 \text{for all } t \in \left[ 0, T \right] .
\end{equation}

\end{itemize}
\end{hypothesis}

\begin{remark}
The condition (c) of Hypothesis \ref{hyp:SDE} holds, e.g., in case  (see, e.g., \cite{Mao2011})
\begin{equation}
 \label{eq:MonotoneCondition}
 x^{\top} b \left( t, x \right)  
+
\frac{1}{2} \sum_{k=1}^{m}  \left\Vert \sigma_k \left( t,  x \right)  \right\Vert^2
\leq 
K \left( 1 +   \left\Vert  x \right\Vert^2 \right)  
\qquad 
\text{for all }x \in \mathbb{R}^d \text{ and } t \in \left[ 0, T \right] .
\end{equation}
\end{remark}

Since we are interested in the weak numerical solution of \eqref{eq:SDE},
we consider the backward Kolmogorov equation 
\begin{equation}
\label{eq:Kolmogorov}
\begin{cases}
\partial_t u \left( t ,x\right) 
 =
-\mathcal{L}\left( u \right) \left( t ,x\right)  
&
\text{if } t \in \left[ 0,T \right] \text{ and } x\in \mathbb{R}^{d},
\\
u \left( T,x\right) = \varphi \left( x\right)  
& \text{ if } x \in \mathbb{R}^{d} ,
\end{cases}
\end{equation}
where 
$
\mathcal{L} 
=
\sum_{k=1}^{d}b^k \frac{\partial}{\partial x^k}  
+
\frac{1}{2} \sum_{k, \ell =1}^{d}  \left( \sum_{j =1}^{m} 
\sigma_j^k  \sigma_j^{\ell} \right) \frac{\partial^2}{\partial x^k \partial x^{\ell}} 
$.
The following assumption on the regularity of the solution of \eqref{eq:Kolmogorov} 
is commonly verified in the proofs of  the linear rate of weak convergence of the numerical schemes for \eqref{eq:SDE}
(see, e.g., \cite{GrahamTalay2013,Kloeden1992,Krylov1999,Milstein2004,Mora2017}).

\begin{hypothesis}
\label{hyp:EDP}
The function $\varphi$ belongs to $\mathcal{C}_{P}^{5}\left( \mathbb{R}^{d},\mathbb{R}\right) $.
The partial differential equation  \eqref{eq:Kolmogorov} has a solution 
$u \in \mathcal{C}^{1,4} \left(\left[ 0, T \right] \times \mathbb{R}^{d},\mathbb{R}\right) $
such that  
$
u \in  \mathcal{C}_{P}^{ 5 } \left( \left[ 0, T \right] \times \mathbb{R}^d,\mathbb{R}\right)
$
and 
$
 \frac{\partial }{\partial t}\partial _{x}^{ \alpha } u
=
-\partial_{x}^{ \alpha }\mathcal{L}\left( u \right) 
$
for any multi-index  $\alpha $ with $\left\vert \alpha \right\vert \leq 3$.
\end{hypothesis}

\begin{remark}
\label{re:Hyp2}
Adopt  Hypothesis \ref{hyp:SDE}.
Then, Hypothesis \ref{hyp:EDP} holds  if  
the coefficients of \eqref{eq:SDE} are Lipschitz continuous functions with linear growth
(see, e.g., \cite{GrahamTalay2013,Kloeden1992,Krylov1999,Milstein2004}),
or in case $\mathcal{L}$ is elliptic (see, e.g., \cite{Cerrai2001}).
If $\mathcal{L}$ is hypoelliptic,
then 
any solution of the Kolmogorov equation is smooth, even for locally Lipschitz SDE \eqref{eq:SDE}
(see, e.g., \cite{Hairer2015,Hormander1967}).
In the locally Lipschitz scalar case, 
the regularity of the solution of the backward Kolmogorov equation \eqref{eq:Kolmogorov} on 
$\left[ 0, T \right] \times \left] 0, + \infty \right[$,
with   $d=1$, $m=1$, $\sigma_1 \left( x \right) = x^{\alpha}$ and $ \varphi$ bounded,
is proved by \cite{Bossy2021} 
in case $\alpha > 1$ and  
$b \left( x \right) \leq K_1 x - K_2 x^{2 \, \alpha -1} +b \left( 0 \right) $ for all $x > 0$,
where $K_1, K_2 > 0$.
In general, \eqref{eq:Kolmogorov} does not have a classical solution (see  \cite{Hairer2015}).
\end{remark}

\subsection{Basic assumptions on the numerical scheme}
\label{subsec:BasicAssumptions}

We design adaptive schemes that satisfy the conditions of Definition \ref{def:Discretizacion},
where, throughout this paper, 
$\left( \tau_n \right)_{n \in \mathbb{Z}_+}$ is a time discretization of $\left[ 0 , T \right]$,
$Y_{n}$ is a weak approximation of $X_{\tau_{n}}$,
and
\[ 
\mathcal{N} \left( \omega \right) = \min \left\{ n \in \mathbb{N}: \tau_{n}  \left( \omega \right) \geq T \right\} 
\qquad \qquad \text{for all } \omega \in \Omega.
\]

\begin{definition}
\label{def:Discretizacion}
 A set  of  random variables $ \tau_n$ and $Y_n$,
where $n \in \mathbb{Z}_+$, 
is called an admissible adaptive strategy (in $\left[ 0, T \right]$) if 
$\tau_{n+1}$, $Y_{n}$ are $\mathfrak{F}_{\tau_{n}}$-measurable for any $n \geq 0 $,
$\tau_0 = 0$,
and 
there exists a random variable $\mathcal{N}$ with values in $\mathbb{N}$ such that
$Y_{n} \left( \omega \right) =  Y_{ \mathcal{N}\left( \omega \right)} \left( \omega \right) $
for any  $n \geq \mathcal{N}\left( \omega \right)$,  and
\[
\begin{cases}
  \tau_{n} \left( \omega \right) = T 
  & \text{if } n \geq \mathcal{N}\left( \omega \right)
  \\
  \tau_{n} \left( \omega \right) <  \tau_{n+1} \left( \omega \right)
  & \text{if }
   n < \mathcal{N}\left( \omega \right)
\end{cases}
\qquad \qquad \text{for all } \omega \in \Omega.
\]

\end{definition}

Similarly to the numerical solution of ODEs,
we construct variable step-sizes $\tau_{n+1} \left( \omega \right) -  \tau_{n} \left( \omega \right) $
such that 
\[
\tau_{n+1} \left( \omega \right) -  \tau_{n} \left( \omega \right)  
\geq
\Delta_{min}
\qquad \qquad \text{for any } n < \mathcal{N}\left( \omega \right) ,
\]
where the minimum step-size $\Delta_{min}$ is given by the user.
Thus,
our adaptive schemes satisfy the following assumption. 

\begin{hypothesis}
 \label{hyp:AcotN}
 Let $ \tau_n$, $Y_n$, and $ \mathcal{N}$ be as in Definition \ref{def:Discretizacion}.
 We suppose that $\sup \left\{ \mathcal{N} \left( \omega \right) :  \omega \in \Omega  \right\} < + \infty $.
\end{hypothesis}

Next, we present a straightforward extension of 
the standard conditions on a numerical scheme  for \eqref{eq:SDE}
to prove that its rate of weak convergence is equal to $1$
(see, e.g., \cite{Kloeden1992,Milstein2004,Mora2017}).

\begin{hypothesis}
\label{hyp:WeakOrden1}
 Let  $\mathcal{A}$ be a collection of admissible  adaptive strategies.
We assume that for  all $\left(\tau_{n} , Y_{n} \right)_{n \geq 0 }$  belonging to $\mathcal{A}$ we have:
\begin{itemize}

\item[(a)] 
For any $p \in \mathbb{N}$ there exist $q \in \mathbb{N}$ and $K \geq 0$ such that 
$
\mathbb{E} \left(  \left\Vert   Y_{n} \right\Vert^{p} \right)
 \leq
 K  \left( 1 + \mathbb{E} \left(  \left\Vert   Y_{0} \right\Vert^{q } \right) \right) 
 $
for all $n \geq 0$.

\item[(b)]  
For every $p \in \mathbb{N}$ there exist $q \in \mathbb{N}$ and $K \geq 0$ 
such that 
\[
\mathbb{E} \left(  \left\Vert   Y_{n+1}  -  Y_{n} \right\Vert^{2p}
\diagup  \mathfrak{F}_{\tau_{n}}
 \right)
 \leq
 K   \left( 1 +  \left\Vert   Y_{n} \right\Vert^{q } \right) 
 \left(\tau_{n+1} - \tau_n \right)^p 
 \qquad \text{for all } n \geq 0.
\]

\item[(c)]   
For every multi-index $\alpha $ with $\left\vert \alpha \right\vert \leq 3$,
there exist $q \in \mathbb{N}$ and $K \geq 0$ such that
\[
\left\vert
\mathbb{E} \left(  
\left(   Y_{n+1} -   Y_{n} \right)^{\alpha}
\diagup  \mathfrak{F}_{\tau_{n}} \right)
-
\mathbb{E} \left(
\left(  Z _{n+1}\left(  Y_{n} \right) -   Y_{n} \right) ^{\alpha}
 \diagup  \mathfrak{F}_{\tau_{n}} \right)
\right\vert
 \leq
K \left( 1+\left\Vert  Y_n \right\Vert ^{q}\right) 
\left(\tau_{n+1} - \tau_n \right)^2 
\]
for all $ n \geq 0 $,
where 
\[
Z_{n+1} \left( x \right) 
: = 
x
+ b\left(  \tau_{n} , x  \right) \left( \tau_{n+1} - \tau_{n} \right)
+ \sum_{k=1}^{m} \sigma_k \left(  \tau_{n}, x \right) \left(W^{k}_{\tau_{n+1}} - W^{k}_{\tau_{n}} \right) 
.
\]

\item[(d)] 
For any $ \phi \in C_{P}^{4} \left(  \mathbb{R}^{d},\mathbb{R}\right)  $
there exist $q \in \mathbb{N}$ and $K \geq 0$ such that
\[
\left\vert \mathbb{E} \phi \left(X_{0} \right) 
-  \mathbb{E}   \phi \left(   Y_{0}  \right)  
 \right\vert 
 \leq K\left(  1+ \left\Vert X_{0} \right\Vert ^{q}\right) 
 \sup_{ k \geq 0, \,  \omega \in \Omega} \left\{  \tau_{k+1}   \left( \omega \right) -   \tau_{k}   \left( \omega \right)   \right\} .
 \]
\end{itemize}
Here, the constants  $q \in \mathbb{N}$ and $K \geq 0$ depend on $\mathcal{A}$,
and do not depend on $\left( \tau_{n}  , Y_{n} \right)_{n \geq 0 }$.
\end{hypothesis}

\begin{remark}
\label{re:Hyp4}
 If $b$, $\sigma_k$ are  Lipschitz functions with all their partial derivatives having at most polynomial growth at infinity,
then the Euler scheme \eqref{eq:Euler-Maruyama} satisfies Hypothesis \ref{hyp:WeakOrden1}.
In the case that $b$ is a non-globally Lipschitz autonomous drift satisfying a one-side linear growth condition,
the Euler scheme fulfills the requirement (a) of Hypothesis  \ref{hyp:WeakOrden1}
whenever  
\begin{equation}
 \label{eq:CondHyp4}
 Y_{n}^{\top} b \left( Y_{n} \right)  
+
\frac{1}{2} \left(  \tau_{n+1} - \tau_{n} \right)  \left\Vert  b \left( Y_{n} \right)    \right\Vert^2
\leq 
K \left( 1 +   \left\Vert  Y_{n} \right\Vert^2 \right) ,
\end{equation}
and  $ \tau_{n+1} - \tau_{n} $ only depends on $ Y_{n} $ 
(see, e.g., \cite{FangGiles2020}).
\end{remark}

For completeness, 
we next establish the linear rate of convergence of the weak  error 
$
\left\vert \mathbb{E} \varphi \left(X_{T}\right) -  \mathbb{E} \varphi \left(  Y_{ \mathcal{N} }  \right)   \right\vert 
$
with respect to the maximum step-size
under Hypothesis \ref{hyp:SDE} - \ref{hyp:WeakOrden1}.
To this end, we use the classical weak convergence analysis introduced by Milstein and Talay.

\begin{theorem}
\label{theo:WeakConvergence}
Suppose that Hypotheses \ref{hyp:SDE} and \ref{hyp:EDP} hold.
Assume that $\mathcal{A}$ is a collection of admissible adaptive strategies satisfying 
Hypotheses \ref{hyp:AcotN} and  \ref{hyp:WeakOrden1}.
Then, there exist $q \in \mathbb{N}$ and $K \left( \cdot \right)$ such that
\[
 \left\vert   \mathbb{E} \varphi \left( X_{T} \right) - \mathbb{E} \varphi \left(  Y_{\mathcal{N}}\right) \right\vert 
 \leq 
 K \left( T \right) \left( 1 + \mathbb{E} \left( \left\Vert  X_{0} \right\Vert^q \right) \right) 
 \sup_{ k \geq 0, \,  \omega \in \Omega} \left\{  \tau_{k+1}   \left( \omega \right) -   \tau_{k}   \left( \omega \right)   \right\}
\]
for all $ \left( \tau_{k}  , Y_{k} \right)_{k\geq 0}$ belonging to $\mathcal{A}$.
\end{theorem}

\begin{proof} 
 Deferred to Section \ref{subsec:Proofth:theo:WeakConvergence}.
\end{proof}

\section{General strategy for adjusting the step-size} 
\label{sec:GeneralStrategyP}
This section introduces a new methodology for selecting the step-sizes of a numerical scheme that solves weakly \eqref{eq:SDE}.

\subsection{Choice of the step-size}
\label{sec:GeneralStrategy}

We present a general mechanism to select the step-sizes $ \tau_{k+1} -  \tau_{k} $ 
of the one-step numerical scheme $Y_k$.
Here, 
$  \mathbb{E} \varphi \left( Y_{k} \right)  \approx \mathbb{E} \varphi \left( X_{\tau_{k}} \right) $
whenever  $\varphi : \mathbb{R}^{d} \rightarrow \mathbb{R}$ is smooth,
and
$ \left( \tau_{k} \right)_{n \geq 0 }$ and $\left(Y_k \right)_{n \geq 0 }$ 
is an admissible adaptive strategy as defined in Section \ref{subsec:BasicAssumptions}.

Let  $ \tau_{n} $ and $ Y_{n} $ be known,
where $Y_{n}$ is $\mathfrak{F}_{\tau_{n}}$-measurable.
Next,
we present a new method for adjusting the step-size $\tau_{n+1} - \tau_{n}$,
which is based on controlling the matching of the first moments of $Y_{n+1} - Y_n$ to
those of $X_{t} \left( \tau_{n} ,   Y_{n}  \right) -  Y_{n}$, 
where
\begin{equation}
\label{eq:2.1}
 X_{t} \left( \tau_{n} ,   Y_{n}  \right)
 =
 Y_{n}
+ \int_{\tau_{n} }^{t} b\left( s, X_{s} \left( \tau_{n} ,   Y_{n} \right) \right)  ds
+ \sum_{k=1}^{m} \int_{\tau_{n}}^{t}\sigma_k \left( s, X_{s} \left( \tau_{n} ,   Y_{n}  \right) \right)  dW^{k}_{s} 
\qquad \quad \text{for all }  t \in \left[ \tau_{n}, T \right].
\end{equation} 
Assume that $\Delta$ is a  $\mathfrak{F}_{\tau_n}$-measurable positive random variable
such that $\tau_{n} + \Delta \leq T$.
Then, $\tau_{n} + \Delta$ is a stopping time.
By abuse of notation,
we define $Y_{n} \left( \tau_{n} + \Delta \right) $ to be the approximation of 
$X_{ \tau_{n} + \Delta} \left( \tau_{n} ,   Y_{n}  \right) $ obtained from applying to \eqref{eq:2.1}
one step of $Y$  with step-size $ \Delta $.
For example, 
if $Y_k $ is the Euler scheme \eqref{eq:Euler-Maruyama},
then 
\begin{equation}
 \label{eq:Euler-MaruyamaDelta}
Y_{n} \left( \tau_n + \Delta \right)
=
Y_{n} 
+  b \left( \tau_n ,  Y_{n} \right) \Delta
+ \sum_{k=1}^{m} \sigma_k \left(  \tau_n ,  Y_{n}  \right)  \left( W^k_{\tau_{n}+\Delta} - W^k_{\tau_n} \right)  .
\end{equation}

{\it
First,
consider an additional approximation  $\hat{Y}_{n} \left( \tau_{n} + \Delta \right) $ to $X_{ \tau_{n} + \Delta} \left( \tau_{n} ,   Y_{n}  \right) $
that is  different from the main approximation $Y_{n} \left( \tau_{n} + \Delta \right) $, 
and give 
estimators 
$e_{i , n} \left(  \Delta  \right)$ and $e_{i , j, n} \left(  \Delta  \right)$
of 
\begin{equation}
 \label{eq:3.11}
  \mathbb{E} \left(  \left(   \hat{Y}^i_{n}\left( \tau_{n} + \Delta \right)  -   Y^i_{n} \right)   \diagup \mathfrak{F}_{\tau_n} \right)
 -
  \mathbb{E} \left(  \left(   Y^i_{n} \left( \tau_{n} + \Delta \right)  -   Y^i_{n} \right)  \diagup \mathfrak{F}_{\tau_n} \right)
\end{equation}
and 
\begin{equation}
 \label{eq:3.12}
 \mathbb{E} \left( \prod_{k=i,j}  \left(   \hat{Y}^k_{n}\left( \tau_{n} + \Delta \right)  -   Y^k_{n} \right)   \diagup \mathfrak{F}_{\tau_n} \right)
 -
  \mathbb{E} \left(  \prod_{k=i,j}  \left(   Y^k_{n} \left( \tau_{n} + \Delta \right)  -   Y^k_{n} \right)  \diagup \mathfrak{F}_{\tau_n} \right) ,
\end{equation}
respectively,
such that the computations of $e_{i , n} \left(  \Delta  \right)$ and $e_{i , j, n} \left(  \Delta  \right)$
do not involve  sampling the random variables 
$ Y_{n}  \left( \tau_{n} + \Delta \right)$ and $  \hat{Y}_{n} \left( \tau_{n} + \Delta \right) $, 
where $i, j =1\ldots, d$.}

In order to advance from $Y_n$ to $Y_{n+1}$,
we would like to determine, roughly speaking,
a large enough step-size $\Delta_* > 0 $ such that 
for any $ i,j = 1 , \ldots, d$ the weighted estimators 
$
\mathfrak{p}_{i, n} \cdot e_{i , n} \left(  \Delta_*  \right)
$ 
and
$
\mathfrak{p}_{i,j, n} \cdot e_{i , j,  n} \left(  \Delta_*  \right)
$ 
remain within the range determined by the thresholds $ d_{i , n} \left( \Delta_* \right) $ and $ d_{i,j , n} \left( \Delta_* \right)$, respectively.
For example, one can  take  $\mathfrak{p}_{i , n} = 1$, 
$ \mathfrak{p}_{i , j, n}  = 1/2 $ and  the  thresholds per step
\begin{equation}
\label{eq:3.5}
 \left\{
\begin{aligned}
d_{i , n} 
& =  Atol_{i}+ Rtol_{i} \, \left\vert  Y_{n}^{i} \right\vert
\\ 
d_{i , j, n} 
& =
\left( \sqrt{ Atol_{i}}+ \sqrt{Rtol_{i}} \, \left\vert  Y_{n}^{i} \right\vert \right) \left( \sqrt{Atol_{j}}+ \sqrt{ Rtol_{j}} \, \left\vert  Y_{n}^{j} \right\vert \right)
\end{aligned}
\right.  ,
\end{equation}
where
$Atol_{i}$ and $Rtol_{i}$ are the absolute and relative tolerance parameters given by the user.
In Section  \ref{sec:LossFunctions},
we look more closely at the weights $\mathfrak{p}$ and the thresholds $ d $.

{\it Second,
find a  $\mathfrak{F}_{\tau_n}$-measurable positive random variable $\Delta_*$, as large as possible,
such that 
\begin{equation} 
\label{eq:GeneralAdaptive}
\emph{L}_n(\Delta_*) \leq  1 ,
\end{equation} 
where for all $\Delta > 0$ we set
\begin{equation} 
\label{eq:DefLossFunction}
\emph{L}_n \left( \Delta \right) =
\left\Vert \left( 
\left\Vert \left(  \frac{ \mathfrak{p}_{i , n} }{d_{i , n}  \left( \Delta \right) }  e_{i , n} \left(  \Delta  \right)  \right)_{i}  \right\Vert_{\mathbb{R}^d} ,
\left\Vert \left(  \frac{\mathfrak{p}_{i , j , n} }{d_{i , j, n}  \left( \Delta \right) }  e_{i , j,  n} \left(  \Delta  \right) \right)_{i , j}  \right\Vert_{\mathbb{R}^{d \times d}}
\right)
\right\Vert_{\mathbb{R}^2}  .
\end{equation} 
Here, 
$\mathfrak{p}_{i , n} $, $\mathfrak{p}_{i , j, n} $, $d_{i , n}  \left( \Delta \right)  $ and $d_{i , j,  n}  \left( \Delta \right)  $ 
are positive $ \mathfrak{F}_{\tau_n} $-measurable random variables, for all $\Delta > 0$. }
We use the function $ \Delta \mapsto  \emph{L}_n \left( \Delta \right)$  to measure in a practical way the discrepancy  between
$Y_{n+1}$ and $\hat{Y}_{n+1}$,
which approximate weakly $X_{  \tau_{n+1} } \left( \tau_{n} ,   Y_{n}  \right)$.

{\it Third,
in case $n \geq 1$, 
set  the next integration time
\begin{equation}
 \label{eq:3.3}
 \tau_{n+1} 
=
\tau_{n} 
+
\max \left\{  \Delta_{min}, 
\min \left\{ \Delta_{max}, \Delta_*,  \mathfrak{fac}_{max} \cdot  \left( \tau_{n} - \tau_{n-1} \right) \right\}
\right\}, 
\end{equation}
where 
--likewise in the numerical integration of ODEs (see, e.g., \cite{HairerNorsettWanner1993})-- 
the constant $\mathfrak{fac}_{max} > 1$ prevents the code from too large step-size increments, 
and 
$\Delta_{min}$ and $\Delta_{max}$ denote the minimum and maximum step-sizes  that are allowed by the user.
If $\tau_{n+1} \left( \omega \right) > T$, then we take  $\tau_{n+1} \left( \omega \right) = T$.}
In the numerical experiment we set $\Delta_{min}$ to be two times the distance from $1$ to the next larger double precision number.
Alternatively, we can choose  $\Delta_{max}/ \Delta_{min}$ equal to a parameter given by the user.
(see, e.g., \cite{KellyLord2016}).

\begin{remark}
\label{re:3.1}
In \eqref{eq:DefLossFunction}  we can also choose $e_{i , j, n} \left(  \Delta  \right)$
to be an estimator of  
\[
 \mathbb{E} \left( \hat{Y}^i_{n}\left(  \tau_n + \Delta \right)   \hat{Y}^j_{n}\left(  \tau_n + \Delta \right)   \diagup \mathfrak{F}_{\tau_n} \right)
 -
  \mathbb{E} \left( Y^i_{n} \left(  \tau_n + \Delta \right)   Y^j_{n} \left(  \tau_n + \Delta \right)    \diagup \mathfrak{F}_{\tau_n} \right)
.
\]
Then,
we  can use thresholds $d_{i , j, n} \left( \Delta \right)$ like   
$
Atol_{i,j} + Rtol_{i,j}  \left\vert   Y^i_{n}  \, Y^j_{n}  \right\vert 
$,
where $Atol_{i,j}$, $Rtol_{i,j}$ are the absolute and relative tolerance parameters for  $e_{i , j, n} \left(  \Delta  \right)$.
\end{remark}

\subsection{Design of local discrepancy functions}
\label{sec:LossFunctions}

Using heuristic arguments we now show how to design suitable discrepancy functions for  the step-size selection mechanism 
presented in Section \ref{sec:GeneralStrategy}.
As in Section \ref{sec:GeneralStrategy},
we consider a one-step numerical scheme $ \left( Y_k \right)_{k=0,\ldots, N}$
that approximates weakly the solution of \eqref{eq:SDE} at the mesh points $\tau_{k}$.
First,
we decompose the mean values of $ \varphi \left( Y_{N}\right) $ and $\varphi \left( X_{{\tau_{N}} } \right)$
in terms of, respectively,  
the increments $Y_{n+1} -   Y_{n}$ 
and 
$X_{ \tau_{n+1} } \left( \tau_{n} ,   Y_{n}  \right)  - Y_{n}$,
where 
$X_{ t } \left( \tau_{n} ,   Y_{n}  \right) $ is the solution of \eqref{eq:2.1}.
In both cases we apply Taylor's theorem to $u$,
together with techniques  from the weak convergence theory of numerical schemes for SDEs
(see, e.g., \cite{GrahamTalay2013,Kloeden1992,Milstein2004}).

\begin{theorem}
\label{theo:LocalExpansion}
Let Hypotheses \ref{hyp:SDE} and \ref{hyp:EDP} hold. 
Consider a set  $\mathcal{A}$ of admissible adaptive strategies
satisfying Hypotheses \ref{hyp:AcotN} and  \ref{hyp:WeakOrden1}.
Then, for any $ \left( \tau_{k}  , Y_{k} \right)_{k\geq 0}$ belonging to $\mathcal{A}$ 
we have:
\begin{equation}
 \label{eq:LD_Scheme}
 \mathbb{E} \varphi \left(  Y_{\mathcal{N}}\right) 
 =
 \mathbb{E} u \left(  0 ,  Y_{0}  \right)
 + 
  \mathbb{E} \sum _{n=0}^{  \mathcal{N} -1}   \left(  
\mathcal{T}_n \left(  Y_{n+1} \right) + \mathbb{E} \left(  R^Y_{n+1}  \diagup  \mathfrak{F}_{\tau_n} \right)
\right) 
 \end{equation}
and
\begin{equation}
 \label{eq:LD}
  \mathbb{E} \varphi \left( X_{T} \right) 
 =
 \mathbb{E} u \left(  0 , X_{0}  \right)
 + 
  \mathbb{E} \sum_{n=0}^{ \mathcal{N}-1}  \left(
\mathcal{T}_n \left(  X_{\tau_{n+1}} \left( \tau_{n} ,  Y_{n}  \right)  \right) +  \mathbb{E} \left(  R^X_{n+1}  \diagup  \mathfrak{F}_{\tau_n} \right) 
\right)
\end{equation}
where 
$X_{ t } \left( \tau_{n} ,   Y_{n}  \right) $ is the solution of \eqref{eq:2.1},
\[
\mathcal{T}_n \left( \xi \right)
=
 \partial_t u \left(  \tau_n ,  Y_{n}  \right)  \left( \tau_{n+1} - \tau_{n} \right)
 +
 \sum_{ \left\vert \alpha \right\vert = 1, 2}    \frac{1}{\alpha !} \partial_x^{ \alpha} u \left(  \tau_n ,  Y_{n}  \right)
 \mathbb{E} \left(  \left(  \xi -  Y_{n} \right)^{\alpha}  \diagup \mathfrak{F}_{\tau_n} \right) ,
\]
and 
$ R^Y_{n+1}$, $ R^X_{n+1} $ are $\mathfrak{F}_{\tau_{n+1}}$-measurable random variables 
satisfying 
\[
\max \left\{
 \left\vert  \mathbb{E} \left( R^Y_{n+1}  \diagup \mathfrak{F}_{\tau_{n}} \right) \right\vert ,
 \left\vert   \mathbb{E} \left( R^X_{n+1}  \diagup \mathfrak{F}_{\tau_{n}} \right) \right\vert
 \right\}
\leq 
 K  \left( 1 +  \left\Vert  Y_n \right\Vert^{q }  \right) \left( \tau_{n+1} - \tau_{n} \right)^2
\]
for all $n=0,\ldots, N-1$.
Here, 
$K > 0$ and $q \in \mathbb{N}$ do not  depend of the random discretization
$\left( \tau_{k}  , Y_{k} \right)_{k \geq 0 }$.
\end{theorem}

\begin{proof} 
 Deferred to Section \ref{subsec:Proofth:theo:LocalExpansion}.
\end{proof}

At the $\left(n+1 \right)$-integration step the values of $\tau_{n}$ and $  Y_{n}$ are known.
As a consequence of  Theorem \ref{theo:LocalExpansion},
we will select a  large enough step size $\tau_{n+1} - \tau_{n}$ such that 
$ \mathcal{T}_n \left(  Y_{n+1} \right) + \mathbb{E} \left(  R^Y_{n+1}  \diagup  \mathfrak{F}_{\tau_n} \right) $
is close to its desired value
$ 
\mathcal{T}_n \left(  X_{\tau_{n+1}} \left( \tau_{n} ,  Y_{n}  \right)  \right) 
+  \mathbb{E} \left(  R^X_{n+1}  \diagup  \mathfrak{F}_{\tau_n} \right) 
$.
Focusing on the difference between the terms of order  $O \left( \tau_{n+1} - \tau_{n} \right)$
we characterize the loss of accuracy of $Y_{n+1}$ by
\begin{equation}
\label{eq:LossFLocal}
\left\vert
 \sum_{ \left\vert \alpha \right\vert = 1, 2}  \frac{1}{ \alpha !}  \partial_x^{ \alpha} u \left(  \tau_n ,  Y_{n}  \right)
 \Big(
  \mathbb{E} \left(  \left(   Y_{n+1} -   Y_{n} \right)^{\alpha} \diagup \mathfrak{F}_{\tau_n} \right)
 -
  \mathbb{E} \left(  \left(  X_{\tau_{n+1}} \left( \tau_{n} ,  Y_{n}  \right) -  Y_{n} \right)^{\alpha} \diagup \mathfrak{F}_{\tau_n} \right)
\Big)
\right\vert ,
\end{equation}
which involves only the first two moments of 
$  Y_{n+1} -   Y_{n}$ and $ X_{\tau_{n+1}} \left( \tau_{n} ,  Y_{n}  \right) -  Y_{n} $.

In the spirit of the step-size selection strategies for ODEs based on embedded methods 
(see, e.g., \cite{Butcher2008,HairerNorsettWanner1993}),
we consider an additional local approximation 
$  \hat{Y}_{n+1} $ of  $X_{\tau_{n+1}} \left( \tau_{n} ,  Y_{n}  \right)$.
Replacing $  X_{\tau_{n+1}} \left( \tau_{n} ,  Y_{n}  \right) -  Y_{n} $ by $\hat{Y}_{n+1}   -  Y_{n}$ 
in \eqref{eq:LossFLocal}  we get 
\begin{equation}
\label{eq:LossFLocalAp}
 \left\vert
 \sum_{ \left\vert \alpha \right\vert = 1, 2}  \frac{1}{ \alpha !} \partial_x^{ \alpha} u \left(  \tau_n ,  Y_{n}  \right)
 \left(
 \mathbb{E} \left(  \left(   Y_{n+1} -   Y_{n} \right)^{\alpha} \diagup \mathfrak{F}_{\tau_n} \right)
 -
 \mathbb{E} \left(  \left(  \hat{Y}_{n+1}   -   Y_{n} \right)^{\alpha} \diagup \mathfrak{F}_{\tau_n} \right)
\right)
\right\vert ,
\end{equation}
which is our  fundamental  local discrepancy function depending on the first two conditional moments of an embedded pair of weak approximations.
We have that \eqref{eq:LossFLocalAp} is an approximation of \eqref{eq:LossFLocal}.
To see this, 
consider  $\hat{\mu} > 0$ 
such that for any multi-index $\alpha $ with $\left\vert \alpha \right\vert = 1, 2$,
\begin{equation}
\label{eq:LocalWeakConvOrderA}
 \left\vert
 \mathbb{E} \left(  
 \left(   \hat{Y}_{n+1}  -  Y_{n} \right)^{\alpha}  \hspace{-1pt}  \diagup \mathfrak{F}_{\tau_n} \right)
-
 \mathbb{E} \left( \left( X_{\tau_{n+1} } \left( \tau_{n} ,   Y_{n}  \right)   -  Y_{n} \right) ^{\alpha} \hspace{-1pt}  \diagup  \mathfrak{F}_{\tau_n} \right)
\right\vert
 \leq
 K \left(  Y_n \right) \left(  \tau_{n+1} - \tau_{n} \right)^{ \hat{\mu} }
\end{equation}
for all  $\mathfrak{F}_{\tau_n}$-measurable positive random variable $  \tau_{n+1} - \tau_{n}$.
Then, 
the absolute value of the difference between \eqref{eq:LossFLocal} and \eqref{eq:LossFLocalAp}
is  $O \left(  \left(   \tau_{n+1} - \tau_{n} \right) ^{ \hat{\mu} } \right) $
when $ \tau_{n+1} - \tau_{n} \rightarrow 0 +$.
Moreover, 
consider the largest $\mu > 0$ satisfying
\begin{equation}
\label{eq:LocalWeakConvOrderS}
 \left\vert
 \mathbb{E} \left(  
 \left(   Y_{n+1}  -  Y_{n} \right)^{\alpha} \diagup \mathfrak{F}_{\tau_n} \right)
-
 \mathbb{E} \left( \left( X_{\tau_{n+1} } \left( \tau_{n} ,   Y_{n}  \right)   -  Y_{n} \right) ^{\alpha} \diagup \mathfrak{F}_{\tau_n} \right)
\right\vert
 \leq
 K \left(  Y_n \right) \left(  \tau_{n+1} - \tau_{n} \right)^{ \mu }
\end{equation}
for all  $\mathfrak{F}_{\tau_n}$-measurable positive random variable  $  \tau_{n+1} - \tau_{n} \in \left] 0, 1 \right[$
and $\left\vert \alpha \right\vert = 1, 2$.
Hence,  \eqref{eq:LossFLocal}  is equal to $O \left(  \left(   \tau_{n+1} - \tau_{n} \right) ^{ \mu} \right) $
as $ \tau_{n+1} - \tau_{n} \rightarrow 0 +$.
Suppose that $\mu <  \hat{\mu}$,
i.e., the rate of convergence of the  first two conditional moments of $ Y_{n+1}  -  Y_{n}$ 
is smaller than the one of  $\hat{Y}_{n+1}  -  Y_{n} $.
Then, 
\[ 
\frac{\text{expression  \eqref{eq:LossFLocal}  }}{  \left(  \tau_{n+1} - \tau_{n} \right)^{ \mu } }
-
\frac{\text{expression  \eqref{eq:LossFLocalAp}  }}{  \left(  \tau_{n+1} - \tau_{n} \right)^{ \mu } }
\longrightarrow
0 
\qquad \text{ as }  \quad \tau_{n+1} - \tau_{n} \rightarrow 0 +
. 
\]
If the limit as $  \tau_{n+1} - \tau_{n} \rightarrow 0+$  of \eqref{eq:LossFLocalAp} (or  \eqref{eq:LossFLocal}) divided by $ \left(  \tau_{n+1} - \tau_{n} \right)^{ \mu }$ is greater than $0$ (see Lemma \ref{le:Momentos} below for an example),
then
\eqref{eq:LossFLocal} and \eqref{eq:LossFLocalAp} are asymptotically equivalent
as $ \tau_{n+1} - \tau_{n} \rightarrow 0 +$.

We propose to construct computable local  discrepancy functions of the form \eqref{eq:DefLossFunction}
such that the  fundamental  local discrepancy function \eqref{eq:LossFLocalAp} is small enough  
whenever the condition \eqref{eq:GeneralAdaptive} holds.
Then,
we design adaptive strategies 
based on selecting the step-size $ \Delta_* = \tau_{n+1} - \tau_{n}  $ as large as possible that satisfies \eqref{eq:GeneralAdaptive}.
A key point here is that for any $\left\vert \alpha \right\vert \leq 2$  the terms 
$
 \mathbb{E} \left(  \left(   Y_{n+1} -   Y_{n} \right)^{\alpha} \diagup \mathfrak{F}_{\tau_n} \right)
 -
 \mathbb{E} \left(  \left(  \hat{Y}_{n+1}   -   Y_{n} \right)^{\alpha} \diagup \mathfrak{F}_{\tau_n} \right) 
$
can be computed exactly or adequately approximated
without sampling the random variables $ Y_{n+1}$ and $  \hat{Y}_{n+1} $, an issue that 
will be addressed in Section \ref{sec:VariableStepSizeEM}.
However, 
the direct evaluation of \eqref{eq:LossFLocalAp} 
involves the computation of  $ \partial_x^{ \alpha} u \left(  \tau_n ,  Y_{n}  \right)$,
which arises from  the backward Kolmogorov equation \eqref{eq:Kolmogorov}.

We can deal with the term $ \partial_x^{ \alpha} u \left(  \tau_n ,  Y_{n}  \right)$
by using upper-bound estimates.
For example,
according to 
$
u \in  \mathcal{C}_{P}^{ 4 } \left( \left[ 0, T \right] \times \mathbb{R}^d,\mathbb{R}\right)
$
we have that  \eqref{eq:LossFLocalAp} is bounded from above by  
\[
K 
\sum_{\left\vert \alpha \right\vert = 1, 2}
  \frac{1}{ \alpha !} \left( 1+\left\Vert  Y_n \right\Vert ^{q_{\alpha}}\right) 
\,  \left\vert \,
 \mathbb{E} \left(  \left(   Y_{n+1} -   Y_{n} \right)^{\alpha} \diagup \mathfrak{F}_{\tau_n} \right)
 -
 \mathbb{E} \left(  \left(  \hat{Y}_{n+1} -   Y_{n} \right)^{\alpha} \diagup \mathfrak{F}_{\tau_n} \right)
\right\vert  ,
\]
where $q_{\alpha} \in \mathbb{Z}_+$ does not depend on $  \left( \tau_k \right)_k$.
This leads to requiring that 
\begin{equation}
\label{eq:Calpha}
\left\vert
 \mathbb{E} \left(  \left(   Y_{n+1} -   Y_{n} \right)^{\alpha} \diagup \mathfrak{F}_{\tau_n} \right)
 -
 \mathbb{E} \left(  \left(  \hat{Y}_{n+1}  -   Y_{n} \right)^{\alpha} \diagup \mathfrak{F}_{\tau_n} \right)
\right\vert 
\end{equation}
does not exceed a suitable threshold for any $\left\vert \alpha \right\vert \leq 2$.
Hence, we obtain the unweighted local discrepancy function
\begin{equation}
\label{eq:LossFunction1}
\emph{L}_n \left( \Delta \right) 
:=
 \left\Vert \left( 
\left\Vert \left(   \frac{ e_{i , n} \left(  \Delta  \right) }{ d_{i , n}  \left( \Delta \right)}   \right)_{i}  \right\Vert_{\mathbb{R}^d} ,
\left\Vert \left(  \frac{ e_{i , j, n} \left(  \Delta  \right) }{2 \, d_{i , j, n}   \left( \Delta \right)}  \right)_{i , j}  \right\Vert_{\mathbb{R}^{d \times d}}
 \right)
 \right\Vert_{\mathbb{R}^2} ,
\end{equation}
where 
\begin{equation}
\label{eq:e_i}
e_{i , n} \left(  \Delta  \right)
=
\mathbb{E} \left(  \hat{Y}^i_{n}\left(  \tau_n + \Delta \right)   \diagup \mathfrak{F}_{\tau_n} \right)
 -
 \mathbb{E} \left( Y^i_{n} \left(  \tau_n + \Delta \right)   \diagup \mathfrak{F}_{\tau_n} \right) ,
\end{equation}
\begin{equation}
\label{eq:e_ij}
e_{i , j, n} \left(  \Delta  \right)
 =
\mathbb{E} \left( \prod_{k=i,j}  \hspace{-3pt} \left(   \hat{Y}^k_{n}\left( \tau_{n} + \Delta \right)  -   Y^k_{n} \right)  \hspace{-2pt} \diagup \mathfrak{F}_{\tau_n}   \hspace{-2pt} \right)
 -
  \mathbb{E} \left(   \prod_{k=i,j} \hspace{-3pt} \left(   Y^k_{n} \left( \tau_{n} + \Delta \right)  -   Y^k_{n} \right)  \hspace{-2pt} \diagup \mathfrak{F}_{\tau_n}  \hspace{-2pt} \right)  \hspace{-3pt} ,
\end{equation}
and the thresholds $ d_{i , n}  \left( \Delta \right)$, $d_{i , j, n}  \left( \Delta \right)$ are positive $ \mathfrak{F}_{\tau_n} $-measurable random variables like \eqref{eq:3.5} (see Remark \ref{re:Threshold}).
As in Section \ref{sec:GeneralStrategy},
we use the notation
$ Y^i_{n} \left( \tau_n + \Delta \right)  $ (resp. $ \hat{Y}^i_{n}\left( \tau_n + \Delta \right)  $) 
to make explicit the dependence of $ Y^i_{n+1} $ (resp. $ \hat{Y}^i_{n+1}  $)
on the step-size $\Delta$.

An alternative way to treat  the term $ \partial_x^{ \alpha} u \left(  \tau_n ,  Y_{n}  \right)$,
which we will not develop in this paper,  
is to replace $\partial_x^{ \alpha} u \left(  \tau_n ,  Y_{n}  \right)$ 
by an approximation. 
For example,
since $u \left( T, \cdot \right) = \varphi $,
in \eqref{eq:LossFLocalAp} we substitute  $\partial_x^{ \alpha} u \left(  \tau_n ,  Y_{n}  \right)$ by its rough estimate 
$  \partial_x^{ \alpha} u \left(  T ,  Y_{n}  \right) =    \partial_x^{ \alpha}  \varphi  \left(  Y_{n}  \right)$.
Thus, \eqref{eq:LossFLocalAp} becomes 
\begin{equation*}
  \left\vert
 \sum_{ \left\vert \alpha \right\vert = 1, 2}  \frac{ \partial_x^{ \alpha} \varphi  \left(  Y_{n}  \right)}{ \alpha !} 
 \left(
 \mathbb{E} \left(  \left(  \hat{Y}_{n+1}   -   Y_{n} \right)^{\alpha} \diagup \mathfrak{F}_{\tau_n} \right)
  -
 \mathbb{E} \left(  \left(   Y_{n+1} -   Y_{n} \right)^{\alpha} \diagup \mathfrak{F}_{\tau_n} \right)
\right)
\right\vert 
\end{equation*}
This leads to the local discrepancy function
\begin{equation}
 \label{eq:3.6m}
 \sum_{ \left\vert \alpha \right\vert = 1, 2} 
 \frac{ 1}{ \alpha !} 
 \max \left\{ \left\vert \partial_x^{ \alpha} \varphi  \left(  Y_{n}  \right)  \right\vert, u_{fac} \right\} 
  \left\vert
 \mathbb{E} \left(  \left(  \hat{Y}_{n+1}   -   Y_{n} \right)^{\alpha} \hspace{-3pt}  \diagup \mathfrak{F}_{\tau_n} \right)
  -
 \mathbb{E} \left(  \left(   Y_{n+1} -   Y_{n} \right)^{\alpha} \hspace{-3pt}  \diagup \mathfrak{F}_{\tau_n} \right)
\right\vert ,
\end{equation}
where $u_{fac} \geq 0$ is a safety lower bound. 
Similar to \eqref{eq:LossFunction1},
looking for the terms of \eqref{eq:3.6m} to be in  a range of predetermined thresholds 
we obtain the weighted local discrepancy function 
\begin{equation*}
\emph{L}_n \left( \Delta \right) 
:=
 \left\Vert \left( 
\left\Vert \left( \wp_{i , n}  \frac{   e_{i , n} \left(  \Delta  \right)  }{ d_{i , n}  \left( \Delta \right)}  \right)_{i}  \right\Vert_{\mathbb{R}^d} ,
\left\Vert \left( \wp_{i , j, n}  \frac{   e_{i , j, n} \left(  \Delta  \right) }{ d_{i , j, n}  \left( \Delta \right) }  \right)_{i , j}  \right\Vert_{\mathbb{R}^{d \times d}}
 \right)
 \right\Vert_{\mathbb{R}^2} ,
\end{equation*}
where 
$ \wp_{i , n}  = \max \left\{ \left\vert \frac{\partial \varphi}{\partial x^i}    \left(  Y_{n}  \right) \right\vert , u_{fac} \right\} $,
$ \wp_{i , j, n}  = \frac{1}{2} \max \left\{  \left\vert  \frac{\partial^2 \varphi }{\partial x^i \partial x^j}  \left(  Y_{n}  \right) \right\vert , u_{fac} \right\} $,
and
$d_{i , n} $, $d_{i , j, n} $, $e_{i , n}$, $ e_{i , j, n}$ are as in \eqref{eq:LossFunction1}.

\begin{remark}
\label{re:Threshold}
We keep the absolute and relative error between  the conditional means 
$
\mathbb{E} \left(  \hat{Y}^i_{n}\left(  \tau_n + \Delta \right)   \diagup \mathfrak{F}_{\tau_n} \right)
$
and
$
 \mathbb{E} \left( Y^i_{n} \left(  \tau_n + \Delta \right)   \diagup \mathfrak{F}_{\tau_n} \right) 
$
less than the tolerances $Atol_{i}$ and $Rtol_{i}$ given by the user.
That is, we wish that
$ e_{i , n} \left(  \Delta  \right) \lesssim  Atol_{i} $
and 
$
e_{i , n} \left(  \Delta  \right) / 
\left(
\text{size of  } \mathbb{E} \left( Y^i_{n} \left(  \tau_n + \Delta \right)   \diagup \mathfrak{F}_{\tau_n} \right) 
\right)
\lesssim  Rtol_{i} 
$.
Similar to ODEs (see, e.g., \cite{Shampine2003,SoderlindWang2006}),
combining the absolute and relative tolerances  yields, for example, the threshold
$
d_{i , n} =  Atol_{i}+ Rtol_{i} \, \left\vert  Y_{n}^{i} \right\vert 
$
if
$
 \left\vert \mathbb{E} \left(    Y^i_{n} \left( \tau_n + \Delta \right)  \diagup \mathfrak{F}_{\tau_n} \right)  \right\vert 
$
is estimated by
$ \left\vert  Y_{n}^{i} \right\vert $.
Alternatively,
we can also take $d_{i , n} = \max \left\{ Atol_{i} , Rtol_{i} \, \left\vert  Y_{n}^{i} \right\vert \right\}$.
Moreover, we ask for 
 $ e_{i , j, n} \left(  \Delta  \right) \lesssim  Atol_{i,j} $
 and 
$
e_{i , j, n} \left(  \Delta  \right) /
\left(
\text{size of  } Y_{n}^{i} Y_{n}^{j}
\right)
\lesssim  Rtol_{i,j} 
$,
which leads us to  threshold
$
d_{i , j, n} 
=  
\left( \widetilde{Atol}_{i}+  \widetilde{ Rtol}_{i}  \, \left\vert  Y_{n}^{i} \right\vert  \right) \left( \widetilde{ Atol }_{j}+ \widetilde{ Rtol }_{j} \, \left\vert  Y_{n}^{j} \right\vert  \right) 
$,
where  $d_{i , j, n}$ has been divided into the coordinate  thresholds
$\widetilde{Atol}_{k}+  \widetilde{ Rtol}_{k}  \, \left\vert  Y_{n}^{k} \right\vert$
for computational efficiency.
In case the values $ Atol_{i,j} $ and $ Rtol_{i,j} $ be requested to be of the same order of magnitude as $Atol_{i}$ and $Rtol_{i}$,
we  set $\widetilde{Atol}_{i} = \sqrt{Atol_{i}}$ and $ \widetilde{ Rtol}_{i}  =  \sqrt{Rtol_{i}}$.
 This gives  \eqref{eq:3.5}.
 As in ODEs (see, e.g., \cite{BrenanCampbellPetzold1996,Shampine2003}),
 $Atol_{i}$ and $Rtol_{i} $ can provide information about  the scales involving in \eqref{eq:SDE},
in addition to accuracy criteria.
A variant of \eqref{eq:3.5} is, e.g., 
$
d_{i , n} \left( \Delta \right) 
  = 
 Atol_{i} + Rtol_{i}  \left\vert \mathbb{E} \left(    Y^i_{n} \left( \tau_n + \Delta \right)  \diagup \mathfrak{F}_{\tau_n} \right)  \right\vert $
and
$
d_{i , j, n} \left( \Delta \right)
  = 
\prod_{k=i,j} 
\left( \sqrt{Atol_{k}} + \sqrt{Rtol_{k} } \left\vert \mathbb{E} \left(    Y^k_{n} \left(  \tau_n + \Delta \right)  \diagup \mathfrak{F}_{\tau_n} \right)  \right\vert \right)
$.
 \end{remark}

Finally, 
in the local discrepancy function \eqref{eq:LossFunction1} any term 
\begin{equation}
 \label{eq:3.10}
  \mathbb{E} \left(  \left(  \hat{Y}_{n+1}   -   Y_{n} \right)^{\alpha} \diagup \mathfrak{F}_{\tau_n} \right) 
 -
  \mathbb{E} \left(  \left(   Y_{n+1}  -   Y_{n} \right)^{\alpha} \diagup \mathfrak{F}_{\tau_n} \right)
\end{equation}
can be replaced by an approximation, 
or a proper upper bound,
which is necessary when exact expressions for this conditional moment difference are not  known.
For instance,
pursuing computational efficiency, 
we can approximate \eqref{eq:3.10} by the leading-order term of the expansion of \eqref{eq:3.10} in powers of $ \tau_{n+1} - \tau_n$. 
This yields the local discrepancy function \eqref{eq:LossFunction1}
but with 
$e_{i , n} \left(  \Delta  \right)$ and $e_{i , j, n} \left(  \Delta  \right)$ 
being estimators of \eqref{eq:3.11} and \eqref{eq:3.12}, respectively.
Thus, 
for each different approximation 
we obtain a particular example of the function \eqref{eq:DefLossFunction}
that measures the discrepancy between $ Y_{n+1}$ and $\hat{Y}_{n+1} $.

\begin{remark}
Since the local discrepancy function \eqref{eq:LossFunction1}
does not depend on the solution of the backward Kolmogorov equation \eqref{eq:Kolmogorov},
we extend the use of  \eqref{eq:LossFunction1}  to design adaptive strategies 
in cases Hypothesis \ref{hyp:EDP} is not fulfilled.
A motivation comes from the numerical experiments of Subsection \ref{sub:Stochastic_Landau} 
that illustrate the good performance of variable step size schemes based on \eqref{eq:LossFunction1} 
in the numerical solution of a nonhypoelliptic SDEs whose drift coefficients have cubic growth.
\end{remark}

\begin{remark}
We split the fundamental  local discrepancy function \eqref{eq:LossFLocalAp}
into  
$e_{i , n} $, $e_{i , j, n} $ --defined by \eqref{eq:e_i} and \eqref{eq:e_ij}--
and 
the information provided by  the solution of the Kolmogorov equation  \eqref{eq:Kolmogorov}.
Since we can estimate  $e_{i , n} $ and $e_{i , j, n} $
with good precision and relatively low computation cost,
in this paper we introduce adaptive strategies focused  essentially on information given by $e_{i , n} $ and $e_{i , j, n} $.
In contrast,
the expansions of the weak error leading to the time-stepping strategies of \cite{Szepessy2001}
feature approximations of $ \partial_x^{ \alpha} u \left(  \tau_n ,  Y_{n}  \right) $ by using dual functions computed a-posteriori.
Moreover,
they do not include explicitly the moments 
$
 \mathbb{E} \left(  \left(   Y_{n+1} -   Y_{n} \right)^{\alpha} \diagup \mathfrak{F}_{\tau_n} \right)
$
and
$
\mathbb{E} \left(  \left(  X_{\tau_{n+1}} \left( \tau_{n} ,  Y_{n}  \right) -  Y_{n} \right)^{\alpha} \diagup \mathfrak{F}_{\tau_n} \right)
$,
and the $\tau_{n}$'s are not stopping times in the stochastic time stepping algorithm given by  \cite{Szepessy2001}.
\end{remark}

\section{Automatic step-size selection based on comparing first and second order weak approximations}
\label{sec:VariableStepSizeEM}

Using the methodology introduced in Section \ref{sec:GeneralStrategyP},
in Sections  \ref{sec:adaptiveEuler-Maruyama} and \ref{subsec:Extrapolation}  we design 
two variable step-size weak schemes for SDEs
with weak convergence orders $1$ and $2$, respectively.
Section \ref{sec:Starting} provides  a  new mechanism for adjusting the initial step-size.

\subsection{Discrepancy measure between the Euler scheme and a second order weak approximation}

\label{sec:BasicLossFunction}

We address the automatic selection of the step-sizes $\tau_{n+1} - \tau_{n}$
of the Euler scheme \eqref{eq:Euler-Maruyama}
by applying the general approach given in Section \ref{sec:GeneralStrategyP}.
To this end,
as the additional approximation we choose 
\begin{equation}
\label{eq:SecondOrderSchemeN}
\begin{aligned}
\hat{Y}_{n} \left( \tau_n + \Delta \right)
 & =
Y_{n} 
+  b \left( \tau_n ,  Y_{n} \right) \Delta  
+ \frac{1}{2} \mathcal{L}_0   b \, \left( \tau_n ,  Y_{n} \right) \Delta^2 
+ \sum_{k=1}^{m} \sigma_k \left(  \tau_n ,  Y_{n}  \right)  \sqrt{ \Delta  } \, \xi^{ k }_{n+1} 
\\
& \quad
+ \sum_{k, \ell =1}^{m} \mathcal{L}_k  \sigma_{\ell} \left(  \tau_n ,  Y_{n}  \right)   \Delta \,  \xi^{ k , \ell}_{n+1} 
 + \frac{1}{2}  \sum_{k=1}^{m} \left( 
 \mathcal{L}_k b \left(  \tau_n ,  Y_{n}  \right)  +  \mathcal{L}_0  \sigma_k \left(  \tau_n ,  Y_{n}  \right)   \right)
 \Delta ^{3/2}  \xi^{ k }_{n+1}  ,
\end{aligned}
\end{equation}
where $\Delta$ is a $\mathfrak{F}_{\tau_n}$-measurable positive random variable,
\begin{equation}
\label{eq:Def_L0_LK}
 \mathcal{L}_0
=
\frac{\partial }{\partial t}
+
\sum_{j = 1}^{d} b^j \frac{\partial}{\partial x^j}  
+
\frac{1}{2} \sum_{i, j =1}^{d}  \left( \sum_{k =1}^{m} 
\sigma_k^i  \sigma_k^j \right) \frac{\partial^2}{\partial x^i \partial x^j} ,
\qquad
\mathcal{L}_k
=
\sum_{j=1}^{d} \sigma_k^j  \frac{\partial }{ \partial x^j } ,
\end{equation}
and
$\xi^{ k}_{n+1}$, $\zeta^{ k }_{n+1}$
are  independent random variables such that 
$\xi^{ k }_{n+1} $ is normally distributed with mean $0$ and variance $1$,
and
\[
\xi^{ k , \ell}_{n+1} 
= 
\frac{1}{2}
\begin{cases}
\xi^{ k }_{n+1}  \xi^{\ell}_{n+1} +  \zeta^{ k }_{n+1} \zeta^{\ell}_{n+1} 
& 
\text{if } k < \ell
\\
\left( \xi^{ k }_{n+1} \right)^2 - 1  
& 
\text{if } k = \ell
\\
\xi^{ k }_{n+1}  \xi^{\ell}_{n+1} -  \zeta^{ k }_{n+1} \zeta^{\ell}_{n+1} 
& 
\text{if } k > \ell
\end{cases} 
\]
with $ \mathbb{P} \left( \zeta^{ k }_{n+1}  = \pm 1 \right) = 1/2 $ (see, e.g., \cite{Milstein2004}).
We consider the local discrepancy function \eqref{eq:DefLossFunction},
where 
$e_{i , n} \left(  \Delta  \right)$ and $e_{i , j, n} \left(  \Delta  \right)$
are estimators of \eqref{eq:3.11} and \eqref{eq:3.12} with  
the embedded pair  $Y_{n} \left( \tau_n + \Delta \right)$ and $\hat{Y}_{n} \left( \tau_n + \Delta \right)$
given by the Euler aproximation \eqref{eq:Euler-MaruyamaDelta} and
the second order  weak It\^o-Taylor  approximation  \eqref{eq:SecondOrderSchemeN}.

The rate of convergence of the  first two conditional moments of $ \hat{Y}_{n} \left( \tau_n + \Delta \right)  -  Y_{n} $
is greater than the one of $ Y_{n} \left( \tau_n + \Delta \right)  -  Y_{n}$.
In fact,  the one-step approximations $Y_{n} \left( \tau_n + \Delta \right)$ and $\hat{Y}_{n} \left( \tau_n + \Delta \right)$
to the solution  of \eqref{eq:2.1} satisfy  \eqref{eq:LocalWeakConvOrderA} and \eqref{eq:LocalWeakConvOrderS}
with $ \hat{\mu} = 3$ and $ \mu = 2$. 
Hence,
 the embedded approximation pair \eqref{eq:Euler-MaruyamaDelta} and \eqref{eq:SecondOrderSchemeN} brings about 
 the asymptotic equivalence of \eqref{eq:LossFLocal} with the fundamental local discrepancy function \eqref{eq:LossFLocalAp}.
 Therefore,
 we can expect that the largest  $\Delta_* > 0$ satisfying  \eqref{eq:GeneralAdaptive} 
 is  a good candidate for the step-size $\tau_{n+1} - \tau_{n} $.
 
Next,
we show that we can compute 
$
 \mathbb{E} \left(  \left(  \hat{Y}_{n} \left( \tau_n + \Delta \right)    -   Y_{n} \right)^{\alpha} \diagup \mathfrak{F}_{\tau_n} \right) 
-
\mathbb{E} \left(  \left(   Y_{n} \left( \tau_n + \Delta \right)  -   Y_{n} \right)^{\alpha} \diagup \mathfrak{F}_{\tau_n} \right)
$,
with $\left\vert \alpha \right\vert \leq 2 $,
by evaluating the partial derivatives, up to the second order,  of $b$ and $ \sigma_k$ at $ \left( \tau_n ,  Y_{n} \right)$.
Thus,
we do not need to sample the random variables $ Y_{n} \left( \tau_n + \Delta \right) $ and $  \hat{Y}_{n} \left( \tau_n + \Delta \right)  $.

\begin{lemma}
 \label{le:Momentos}
 Let  $Y_{n} \left( \tau_n + \Delta \right)  $ and  $\hat{Y}_{n} \left( \tau_n + \Delta \right)$ be described by 
 \eqref{eq:Euler-MaruyamaDelta} and \eqref{eq:SecondOrderSchemeN}.
 Then,
$
 \mathbb{E} \left(  \hat{Y}_{n} \left( \tau_n + \Delta \right)   \diagup \mathfrak{F}_{\tau_n} \right)
 -
\mathbb{E} \left(  Y_{n} \left( \tau_n + \Delta \right)   \diagup \mathfrak{F}_{\tau_n} \right)
=
\frac{1}{2} \mathcal{L}_0   b \, \left( \tau_n ,  Y_{n} \right)  \Delta^2
$,
and 
\begin{align*}
& 
 \mathbb{E} \left(  \prod_{k=i,j}  
 \left(   \hat{Y}^k_{n}\left( \tau_n + \Delta \right)  -   Y^k_{n} \right)  \diagup \mathfrak{F}_{\tau_n} \right)
 -
  \mathbb{E} \left( \prod_{k=i,j}   \left(   Y^k_{n} \left( \tau_n + \Delta \right)  -   Y^k_{n} \right) \diagup \mathfrak{F}_{\tau_n} \right)
\\
& =
T_2^{i, j} \left(   \tau_n ,  Y_{n} \right)    \Delta^2
+
T_3^{i, j} \left(   \tau_n ,  Y_{n} \right)    \Delta^3
+
T_4^{i, j}\left(  \tau_n ,  Y_{n}  \right)    \Delta^4,
\end{align*}
where  
$i, j = 1,\ldots, d$, 
$
T_4^{i, j}  
 =
  \frac{1}{4}   \mathcal{L}_0 b^{ i} \cdot \mathcal{L}_0 b^{ j} 
$,
\[
 T_2^{i, j}  
 =
 \frac{1}{2}  \sum_{k=1}^m  \sigma_k^{i} \left( \mathcal{L}_0   \sigma_k^{ j }  + \mathcal{L}_k   b^{ j }\right)
 +
 \frac{1}{2}  \sum_{ k =1}^m  \sigma_k^{j} \left( \mathcal{L}_0   \sigma_k ^{i}  + \mathcal{L}_k   b^{i }\right)
 +
 \frac{1}{2}  \sum_{k, \ell =1}^m   \mathcal{L}_k \sigma_{\ell }^{ i }  \cdot \mathcal{L}_k  \sigma_{\ell }^{ j } 
,
\]
\[
T_3^{i, j}
=
 \frac{1}{2}   b^{ i }  \mathcal{L}_0 b^{ j }
 +
 \frac{1}{2}   b^{ j }  \mathcal{L}_0 b^{ i}
 +
 \frac{1}{4}  \sum_{k=1}^m \mathcal{L}_0  \sigma_k^{ i }  \left( \mathcal{L}_k  b^{  j } +  \mathcal{L}_0  \sigma_k^{ j } \right)
 +
 \frac{1}{4}  \sum_{k=1}^m  \mathcal{L}_k  b^{  i } \left(  \mathcal{L}_k  b^{  j } + \mathcal{L}_0  \sigma_k^{ j } \right)
.
\]
\end{lemma}

\begin{proof}
 Deferred to Section \ref{subsec:Proofth:le:Momentos}.
\end{proof}

We recall that $e_{i , n} \left(  \Delta  \right)$ and $e_{i , j, n} \left(  \Delta  \right)$ are estimators of 
\[
 \mathbb{E} \left(    \hat{Y}^i_{n}\left( \tau_n + \Delta \right)  -   Y^i_{n}    \diagup \mathfrak{F}_{\tau_n} \right)
 -
  \mathbb{E} \left(    Y^i_{n} \left( \tau_n + \Delta \right)  -   Y^i_{n}   \diagup \mathfrak{F}_{\tau_n} \right)
\]
and 
\[
 \mathbb{E} \left( \prod_{k=i,j}  
 \left(   \hat{Y}^k_{n}\left( \tau_n + \Delta \right)  -   Y^k_{n} \right)  \diagup \mathfrak{F}_{\tau_n} \right)
 -
  \mathbb{E} \left(   \prod_{k=i,j} 
  \left(   Y^k_{n} \left( \tau_n + \Delta \right)  -   Y^k_{n} \right) \diagup \mathfrak{F}_{\tau_n} \right)
,
\]
respectively.
Lemma \ref{le:Momentos} gives
$
e_{i , n} \left(  \Delta  \right)
= 
\frac{1}{2} \mathcal{L}_0   b^i  \left( \tau_n ,  Y_{n} \right)   \Delta ^2 
$
for all $ i = 1, \ldots, d $.
Hence,
\[
\left\Vert \left(  \frac{ \mathfrak{p}_{i , n} }{d_{i , n}  \left( \Delta \right) }  e_{i , n} \left(  \Delta  \right)  \right)_{i}  \right\Vert_{\mathbb{R}^d}
=
\Delta^2 
 \left\Vert \left(  \frac{\mathfrak{p}_{i , n}}{2 d_{i , n}  \left( \Delta \right)} \mathcal{L}_0   b^i \left( \tau_n ,  Y_{n} \right)    \right)_{ i } \right\Vert_{\mathbb{R}^d} .
\]
By  Lemma \ref{le:Momentos},
pursuing computational efficiency we take
$
e_{i,j , n} \left(  \Delta  \right) = T_2^{i, j} \left(   \tau_n ,  Y_{n} \right)    \Delta^2 
$,
and so
\[
 \left\Vert \left(  \frac{\mathfrak{p}_{i , j , n} }{d_{i , j, n}  \left( \Delta \right) }  e_{i , j,  n} \left(  \Delta  \right) \right)_{i , j}  \right\Vert_{\mathbb{R}^{d \times d}}
 =
 \Delta^2 
 \left\Vert \left(  \frac{\mathfrak{p}_{i , j , n} }{d_{i , j, n}  \left( \Delta \right) }  T_2^{i, j} \left(   \tau_n ,  Y_{n} \right) \right)_{i , j}  \right\Vert_{\mathbb{R}^{d \times d}} 
,
\]
where  $\mathfrak{p}_{ \cdot}$ and $ d_{\cdot} \left( \Delta \right) $ denote the weights and thresholds, respectively.
This gives the local discrepancy function $\emph{L}_n \left( \Delta \right)$ defined by
\begin{equation}
 \label{eq:MainLocalDF}
 \Delta^2 
\left\Vert \left( 
 \left\Vert \left(  \frac{\mathfrak{p}_{i , n}}{2 d_{i , n}  \left( \Delta \right)} \mathcal{L}_0   b^i \left( \tau_n ,  Y_{n} \right)    \right)_{ i } \right\Vert_{\mathbb{R}^d} 
  , 
 \left\Vert \left(  \frac{\mathfrak{p}_{i , j , n} }{d_{i , j, n}  \left( \Delta \right) }  T_2^{i, j} \left(   \tau_n ,  Y_{n} \right) \right)_{i , j}  \right\Vert_{\mathbb{R}^{d \times d}}
\right)
\right\Vert_{\mathbb{R}^2}
\end{equation}
We recall that we wish to determine
a large enough $\mathfrak{F}_{\tau_n}$-measurable positive random variable $\Delta_*$
with the property $ \emph{L}_n \left( \Delta_* \right)   \leq  1 $.

\subsection{Automatic selection of the starting step-size}
\label{sec:Starting}

The choice \eqref{eq:3.3} provides  information to  $\tau_{n+1} - \tau_{n}$  on the previous step-size  $\tau_{n} - \tau_{n-1}$ whenever $n \geq 1$.
Indeed, 
the positive constant $\mathfrak{fac}_{max}$ prevents a bad selection of $\tau_{n+1} - \tau_{n}$,  with $n \geq 1$,  by avoiding sudden increases of the new step-sizes.
Since this precaution can not be taken in the computation of  $\tau_{1}$,
we next introduce a new module for the calculation of the starting  step-size  $\tau_{1} - \tau_{0}$.
Alternatively, the user has to specify an initial step-size,
which, for instance, could be a hard task for casual users.

Inspired by selection algorithms of the initial step-size for ODE solvers 
(see, e.g., \cite{GladwellShampineBrankin1987,HairerNorsettWanner1993}),
we limit the increment $Y_1 - Y_0$ to be within a given tolerance.
This leads to the step-size $\Delta_{s}$ given in Definition \ref{def:Delta_initial} below,
which we get by applying Lemma \ref{le:incrementos} below with 
$
d_{i,0} \left( \Delta \right) = Atol^s_{i} + Rtol^s_{i} \, \left\vert  Y_{0}^{i} \right\vert 
$
and
$
\widetilde{d}_{i,0} \left( \Delta \right)  = \sqrt{ Atol^s_{i} }+ \sqrt{ Rtol^s_{i} } \, \left\vert  Y_{0}^{i} \right\vert 
$,
where $Atol^s_{i}, Rtol^s_{i}  \in \left[ 0 , 1 \right]$ are the absolute and relative tolerance parameters
for the starting step-size.

\begin{lemma}
\label{le:incrementos}
Let $Y_{n} \left( \tau_n + \Delta \right)$ be defined by \eqref{eq:Euler-MaruyamaDelta}.
Let $d_{i,n} \left( \Delta \right) $ and $\widetilde{ d }_{i,n} \left( \Delta \right) $
be positive $ \mathfrak{F}_{\tau_n} $-measurable random variables  such that
$
\widetilde{ d }_{i,n} \left( \Delta \right)  \geq d_{i,n} \left( \Delta \right) > 0
$
for all $i =1, \ldots, d $.
Then
\begin{align*}
&  \left\Vert
\left(
\frac{
\mathbb{E} \left(  \left(  Y_{n}^i \left( \tau_n + \Delta \right)  -   Y^i_{n} \right)  \left( Y^j_{n} \left( \tau_n + \Delta \right)  -   Y^{j}_{n} \right) \diagup \mathfrak{F}_{\tau_n} \right) }
{
2 \, \widetilde{d}_{i,n} \left( \Delta \right) \widetilde{d}_{j,n} \left( \Delta \right)
} \right)_{i,j}
\right\Vert_{\mathbb{R}^{d \times d}}
\\
& 
\leq
\frac{\Delta^2}{2}
\left\Vert
\left(
\frac{ b^i \left( \tau_n ,  Y_{n} \right)  }
{ d_{i,n} \left( \Delta \right) 
} \right)_{i}
\right\Vert_{ \mathbb{R}^d }^2
 +
\frac{\Delta}{2}  
\sum_{k=1}^{m}
 \left\Vert
\left(
\frac{ \sigma^i_k \left(  \tau_n ,  Y_{n}  \right)  }
{ \widetilde{d}_{i,n} \left( \Delta \right)  } \right)_{i}
\right\Vert_{ \mathbb{R}^d } ^2 ,
\end{align*}
where the pair $\left(   \left\Vert \cdot \right\Vert_{\mathbb{R}^{d}} ,  \left\Vert \cdot \right\Vert_{\mathbb{R}^{d \times d}} \right)$ 
satisfies \eqref{eq:4.4}.
\end{lemma}

\begin{proof}
 Deferred to Section \ref{subsec:Proofth:le:incrementos}.
\end{proof}

  \begin{definition}
 \label{def:Delta_initial}
 Let $\Delta_{max}$ denote the maximum  step-size.
 In case 
\[
\max  \left\{ 
 \left\Vert 
\left( \frac{ b^i \left( \tau_0 ,  Y_{0} \right) }{ Atol^s_{i}+ Rtol^s_{i} \, \left\vert  Y_{0}^{i} \right\vert } \right)_{i }
 \right\Vert_{ \mathbb{R}^d }  , 
 \sum_{k=1}^{m}
 \left\Vert
\left(
\frac{ \sigma^i_k \left(  \tau_0 ,  Y_{0}  \right)  }
{ \left( \sqrt{Atol^s_{i}}+ \sqrt{Rtol^s_{i}} \, \left\vert  Y_{0}^{i} \right\vert \right) } \right)_{i}
\right\Vert_{ \mathbb{R}^d } ^2 
 \right\}
\]
is greater than $ 1/ \Delta_{max}$,
we define $\Delta_{s}$ to be 
\[
1/ 
\max \left\{
 \left\Vert 
\left( \frac{ b^i \left( \tau_0 ,  Y_{0} \right) }{ Atol^s_{i}+ Rtol^s_{i} \, \left\vert  Y_{0}^{i} \right\vert } \right)_{i }
 \right\Vert_{ \mathbb{R}^d }  , 
 \sum_{k=1}^{m}
 \left\Vert
\left(
\frac{ \sigma^i_k \left(  \tau_0 ,  Y_{0}  \right)  }
{ \left( \sqrt{Atol^s_{i}}+ \sqrt{Rtol^s_{i}} \, \left\vert  Y_{0}^{i} \right\vert \right) } \right)_{i}
\right\Vert_{ \mathbb{R}^d } ^2 
 \right\}.
 \]
Otherwise, 
we take $\Delta_{s} = \Delta_{max}$.
\end{definition}

Now, 
we compare $\Delta_s$ with the step-size $\Delta_*$ proposed in Section \ref{sec:GeneralStrategy} to compute $Y_{1}$.
In particular, following Section \ref{sec:BasicLossFunction} we find a large enough $\Delta_* > 0$ such that $ \emph{L}_0 \left( \Delta_* \right)   \leq  1 $,
where $ \emph{L}_0 \left( \Delta \right)$ is defined by \eqref{eq:MainLocalDF}.
Then,
we set $\Delta_{0}=\min \left\{ \Delta_*,  \Delta_s \right\}$,
and so we take the starting step-size
\[
\tau_{1} 
=
\tau_{0} 
+
\max \left\{  \Delta_{min} , 
\min \left\{ \Delta_{max}, \Delta_*,  \Delta_s \right\}
\right\} 
.
\]

\subsection{A basic variable step-size Euler scheme}
\label{sec:adaptiveEuler-Maruyama}

In this subsection we consider the local discrepancy function \eqref{eq:MainLocalDF}
with  $ \mathfrak{p}_{i , n}  = 1 $ and $ \mathfrak{p}_{i , j, n}  = 1/2 $,
i.e., 
the weights $\mathfrak{p}_{i , n}$ and $ \mathfrak{p}_{i , j, n}$ are provided by \eqref{eq:LossFunction1}.
In order to reduce the computational complexity
we choose the thresholds $d_{i , j, n}  \left( \Delta \right)$ to be 
$\widetilde{d}_{i , n}  \left( \Delta \right)  \widetilde{d}_{j, n}  \left( \Delta \right) $.
Therefore,  \eqref{eq:MainLocalDF} becomes 
\begin{equation}
\label{eq:DF1}
 \emph{L}_n \left( \Delta \right) 
 =
\frac{ \Delta^2 }{2}
\left\Vert \left( 
 \left\Vert \left(  \frac{ \mathcal{L}_0   b^i \left( \tau_n ,  Y_{n} \right)  }{ d_{i , n}  \left(  \Delta \right)}   \right)_{ i } \right\Vert_{\mathbb{R}^d} 
  , 
\left\Vert 
\left(  \frac{ T_2^{i, j} \left(   \tau_n ,  Y_{n} \right) }
{ \widetilde{d}_{i , n}  \left( \Delta \right)  \widetilde{d}_{j, n}  \left( \Delta \right)  }   \right)_{i , j}  
\right\Vert_{\mathbb{R}^{d \times d}} 
\right)
\right\Vert_{\mathbb{R}^2} ,
\end{equation}
where 
the  matrix-norm $ \left\Vert \cdot \right\Vert_{\mathbb{R}^{d \times d}}$ 
satisfies \eqref{eq:4.4}, 
and
the vector norm $  \left\Vert  \cdot  \right\Vert_{\mathbb{R}^d}  $ 
may be different from that of Section \ref{sec:Starting}.
We recall that $\mathcal{L}_0$ and $\mathcal{L}_k$ are defined by \eqref{eq:Def_L0_LK}. 
Since
\begin{equation*}
\left( T_2^{i, j}  \right)_{i,j}
= 
\frac{1}{2}  \sum_{k=1}^m  \sigma_k \left( \mathcal{L}_0   \sigma_k  + \mathcal{L}_k   b \right)^{\top}
+
\frac{1}{2}  \sum_{k=1}^m  \left( \mathcal{L}_0   \sigma_k  + \mathcal{L}_k   b \right) \sigma_k^{\top}
+
 \frac{1}{2}  \sum_{k, \ell =1}^m   \mathcal{L}_k \sigma_{\ell }  \cdot \mathcal{L}_k  \sigma_{\ell }^{ \top } ,
\end{equation*}
combining the triangle inequality with \eqref{eq:4.4} we obtain
$
\left\Vert \left(  \frac{   T_2^{i, j} \left(   \tau_n ,  Y_{n} \right)}{ \widetilde{d}_{i , n}  \left( \Delta \right)  \widetilde{d}_{j, n}  \left( \Delta \right) }  \right)_{i , j}  \right\Vert_{\mathbb{R}^{d \times d}} 
\leq
\mathfrak{t}_{n}  \left(  \Delta  \right)  
$,
where
\begin{equation}
 \label{eq:4.13}
  \mathfrak{t}_{n}  \left(  \Delta  \right) 
  =
 \sum_{k=1}^m 
  \left\Vert    \left(  \frac{    \sigma_k^i  \left(   \tau_n ,  Y_{n} \right)  }{  \widetilde{d}_{i , n}   \left( \Delta \right) } \right)_{i }   \right\Vert_{\mathbb{R}^d} 
  \left\Vert   \left(  \frac{  \mathcal{L}_0   \sigma_k^i \left(   \tau_n ,  Y_{n} \right)   + \mathcal{L}_k   b^i \left(   \tau_n ,  Y_{n} \right) }
  {  \widetilde{d}_{i , n}   \left( \Delta \right) } \right)_{i }  \right\Vert_{\mathbb{R}^d} 
  +   \frac{1}{2}  \sum_{k, \ell =1}^m  
  \left\Vert \left(  \frac{   \mathcal{L}_k \sigma_{\ell }^i \left( \tau_n ,  Y_{n} \right)  }{  \widetilde{d}_{i , n}   \left( \Delta \right) } \right)_{i } \right\Vert_{\mathbb{R}^d}^2  .
\end{equation}
Therefore, for all $\Delta > 0$,
\begin{equation}
\label{eq:4.6}
 \emph{L}_n \left( \Delta \right) 
\leq 
\widetilde{\emph{L}}_n \left( \Delta \right)
:=
\frac{ \Delta^2 }{2}
\left\Vert \left( 
 \left\Vert \left(   
 \frac{ \mathcal{L}_0   b^i \left( \tau_n ,  Y_{n} \right)  }{ d_{i , n}  \left( \Delta \right)}   \right)_{ i } \right\Vert_{\mathbb{R}^d} 
  , 
\mathfrak{t}_{n}  \left(  \Delta  \right) 
\right)
\right\Vert_{\mathbb{R}^2} .
 \end{equation}

Looking for computational efficiency we consider 
the ancillary local discrepancy function $\widetilde{\emph{L}}_n \left( \Delta \right)$ given by \eqref{eq:4.6},
and we search for the largest  $\Delta_* > 0$ such that $ \widetilde{\emph{L}}_n \left( \Delta_* \right) \leq 1$.
Moreover,
a way to simplify the computation of  $ \Delta_{*}$  is to take
$\widetilde{d}_{i , n}   \left( \Delta \right)$  and $d_{i , n} \left( \Delta \right)$ 
independent of $\Delta$, for instance, as in \eqref{eq:3.5}. 
In this case, 
$\mathfrak{t}_{n}  \left(  \Delta  \right) $ does not depend on $\Delta$,
and 
the largest $\mathfrak{F}_{\tau_n}$-measurable positive random variable $\Delta_{*,n}$ 
satisfying $ \widetilde{\emph{L}}_n \left( \Delta_* \right) \leq 1$ is 
 $
 \Delta_{*,n}
 =
 \sqrt{ 
 2/  
 \left\Vert \left( 
 \left\Vert \left(  
 \frac{ \mathcal{L}_0   b^i \left( \tau_n ,  Y_{n} \right)  }{ d_{i , n} }   \right)_{ i } \right\Vert_{\mathbb{R}^d} 
  , 
\mathfrak{t}_{n}  
\right)
\right\Vert_{\mathbb{R}^2}
 }
 $
whenever 
$\mathfrak{t}_{n}   \neq 0$ or
$ \mathcal{L}_0   b^i \left( \tau_n ,  Y_{n} \right)   \neq 0$ for some $i \in \left\{ 1,\ldots, n \right\}$.
Here $\mathfrak{t}_{n}  \left(  \Delta  \right) \equiv \mathfrak{t}_{n}$  
and $\widetilde{d}_{i , n}   \left( \Delta \right)  \equiv \widetilde{d}_{i , n}$. 
Selecting $d_{i , n} \left( \Delta \right)$ and $\widetilde{d}_{i , n}   \left( \Delta \right)$ as in \eqref{eq:3.5} we obtain the following
suboptimal  selection of the step-size $\Delta_*$.

\begin{definition}[Step-size $\Delta_{*,n}$  corresponding to  the ancillary local discrepancy function \eqref{eq:4.6}]
\label{def:Delta_*}
Let 
\begin{align*}
 \mathfrak{t}_{n} 
& =
 \sum_{k=1}^m 
  \left\Vert    \left(  \frac{    \sigma_k^i  \left(   \tau_n ,  Y_{n} \right)  }
  {   \sqrt{ Atol_{i}}+ \sqrt{Rtol_{i}} \, \left\vert  Y_{n}^{i} \right\vert  } \right)_{i }   \right\Vert_{\ell^{p}}  
  \left\Vert   \left(  \frac{  \mathcal{L}_0   \sigma_k^i \left(   \tau_n ,  Y_{n} \right)   + \mathcal{L}_k   b^i  \left(   \tau_n ,  Y_{n} \right) }
  {   \sqrt{ Atol_{i}}+ \sqrt{Rtol_{i}} \, \left\vert  Y_{n}^{i} \right\vert  } \right)_{i }  \right\Vert_{\ell^{p}} 
 \\
 & \quad 
 +  \frac{1}{2}  \sum_{k, \ell =1}^m  
  \left\Vert \left(  \frac{   \mathcal{L}_k \sigma_{\ell }^i \left( \tau_n ,  Y_{n} \right)  }
  {   \sqrt{ Atol_{i}}+ \sqrt{Rtol_{i}} \, \left\vert  Y_{n}^{i} \right\vert  } \right)_{i } \right\Vert_{\ell^{p}}^2   ,
\end{align*}
where $Atol_{i}, Rtol_{i}  \in \left[ 0 , 1 \right]$,
and $\mathcal{L}_0$, $\mathcal{L}_k$ are given by \eqref{eq:Def_L0_LK}. 
If
$\mathfrak{t}_{n}   \neq 0$
or 
$\mathcal{L}_0   b^i \left( \tau_n ,  Y_{n} \right)  \neq 0$ for some $i \in \left\{ 1,\ldots, n \right\}$,
then we set
\begin{equation}
 \label{eq:4.8}
  \Delta_{*,n}
 =
 \sqrt{
 2/ 
 \left\Vert \left( 
 \left\Vert \left(  
 \frac{ \mathcal{L}_0   b^i \left( \tau_n ,  Y_{n} \right)  }{  Atol_{i}+ Rtol_{i} \, \left\vert  Y_{n}^{i} \right\vert  }   \right)_{ i } \right\Vert_{\mathbb{R}^d} 
  , 
\mathfrak{t}_{n}  
\right)
\right\Vert_{\mathbb{R}^2}
 }.
\end{equation}
Otherwise, $  \Delta_{*,n} = \Delta_{max}$,
where $\Delta_{max}$ stands for the maximum  step-size.
\end{definition}

In Adaptive scheme \ref{AdaptiveEM1} below
we provide a weak approximation to $X_t$ at the times $0 < T_1 < \cdots < T_{M_*} \leq T$.
To this end,
we adjust the step-size $\tau_{n+1} - \tau_{n}$ to compute $Y_{n+1}$ from $Y_n$ by means of Definition \ref{def:Delta_*}.

\begin{strategy}
\label{AdaptiveEM1}
Consider the real numbers $0 < T_1 < \cdots < T_{M_*} \leq T$, given by the user.
Then:
 
\begin{enumerate}

\item Simulate $Y_0$. Take $k=1$.

\item Compute 
$
\tau_{1} 
=
\min \left\{ T_1 , 
\tau_{0} 
+
\max \left\{  \Delta_{min} , 
\min \left\{ \Delta_{max}, \Delta_{*,0},  \Delta_{s}  \right\}
\right\}  \right\} 
$,
where
$\Delta_{s}$ and $\Delta_{*,0}$ are given by Definition \ref{def:Delta_initial} and Definition \ref{def:Delta_*} with $n=0$,
respectively.

\item Compute $Y_1$ according to \eqref{eq:Euler-Maruyama} with $\tau_{0} = 0$.
Set $n=1$.

\item 
If $\tau_{n}  < T_k$, then go to Step 5.
In case $\tau_{n}  = T_k  < T_{M_*}$, go to to Step 5 with $k+1$ as the new value of $k$.
Stop when  $\tau_{n}  = T_k  = T_{M_*}$.

\item Compute $\Delta_{*,n}$ as in Definition \ref{def:Delta_*}. 
Then, 
take $\tau_{n+1} $ equal to the minimum between $T_k$ and 
$
\tau_{n} 
+
\max \left\{  \Delta_{min} , 
\min \left\{ \Delta_{max}, \Delta_{*,n},  \mathfrak{fac}_{max} \cdot  \left( \tau_{n} - \tau_{n-1} \right) \right\}   \right\} 
$.

\item Compute $Y_{n+1}$ according \eqref{eq:Euler-Maruyama}.
Return to Step 4 with  $n$ updated to $n+1$.

\end{enumerate}
\end{strategy}

\subsection{Variable step-size scheme of second order}
\label{subsec:Extrapolation} 

In this subsection we generalize to the SDEs  
the local extrapolation procedure for variable step-size schemes for ODEs based on embedded formulas 
(see, e.g., \cite{HairerNorsettWanner1993,Shampine1994,Shampine2003}).
Namely,
by interchanging the roles of  $Y_{n} \left( \tau_n + \Delta \right)$  and $\hat{Y}_{n} \left( \tau_n + \Delta \right)$ 
in Subsections \ref{sec:BasicLossFunction} and \ref{sec:adaptiveEuler-Maruyama} we obtain
the following weak second order variable step-size scheme.

\begin{strategy}
\label{AdaptiveO2_Extrap}
Proceed as in Adaptive scheme \ref{AdaptiveEM1}
except  in Steps 3 and 6 we define $Y_1$ and $Y_{n+1}$ to be the right-hand side of \eqref{eq:SecondOrderSchemeN}.
\end{strategy}

\begin{remark}
We can reduce the computational budget required for simulating  \eqref{eq:Euler-Maruyama} and  \eqref{eq:SecondOrderSchemeN}  by replacing the normally distributed random variables,
which model $W^k_{\tau_{n+1}} - W^k_{\tau_n}$,
by uniform or discrete random variables (see, e.g., Section 2.6 of \cite{Milstein2004}).
For example, 
in the right-hand side of \eqref{eq:SecondOrderSchemeN} we can choose $ \xi^{ k }_{n+1}$ given by
$ \mathbb{P} \left( \xi^{ k }_{n+1}  = \pm \sqrt{3} \right) = 1/6 $ and $ \mathbb{P} \left( \xi^{ k }_{n+1}  = 0 \right) = 2/3 $
(see, e.g., \cite{Kloeden1992,Milstein2004}).
\end{remark}

We are in the context of Subsections \ref{sec:BasicLossFunction} and \ref{sec:adaptiveEuler-Maruyama} 
but with $Y_{n} \left( \tau_n + \Delta \right)$  and $\hat{Y}_{n} \left( \tau_n + \Delta \right)$  swapped.
That is,
the additional approximation is defined by \eqref{eq:Euler-MaruyamaDelta}
and we  compute $Y_{n+1}$ by the second order  weak It\^o-Taylor  scheme obtained by iterating \eqref{eq:SecondOrderSchemeN}.
Since the difference between   \eqref{eq:Euler-MaruyamaDelta} and  \eqref{eq:SecondOrderSchemeN}
is measured by the ancillary local discrepancy function $\widetilde{\emph{L}}_n \left( \Delta \right)$ 
defined  by \eqref{eq:4.6},
we move forward with the step-sizes $\tau_{n+1} - \tau_{n}$ 
provided by Definition \ref{def:Delta_*} as in Subsection  \ref{sec:adaptiveEuler-Maruyama}.


\section{Adaptive adjustment of the  sample-size}
\label{Sec:SampleSize}

In order to estimate  $\mathbb{E} \varphi \left(X_{T_k}\right)$,
where $0 < T_1 < \cdots < T_M = T$ are deterministic times given by the user,
this paper combines the new variable step-size weak schemes
with 
a version of a classical method for determining the final number of simulations of the Monte-Carlo sampling.
Alternatively,
we can use a Multilevel Monte Carlo method (see, e.g.  \cite{GilesLesterWhittle2016}),
but this makes it hard to evaluate the performance of the new adaptive algorithms.

Consider a numerical scheme $ \left( Y_n \right)_{n}$ that approximates the solution of \eqref{eq:SDE} 
at nodes $\left( \tau_n \right)_n $ satisfying 
$ 
 n\left( k \right) : = \inf \left\{n : \tau_n = T_k \right\} < + \infty  
$ 
for all $k= 1, \ldots, M$.
We take $ \left( Y_n \right)_{n}$ to be 
 the Adaptive schemes \ref{AdaptiveEM1} and \ref{AdaptiveO2_Extrap}.
Then,
for every $k$ 
we simulate independent and identically distributed random variables
$Y^{\left\{ k ,1 \right\}}, \ldots, Y^{\left\{ k , S_k \right\}}$ 
distributed according to the law of $Y_{n\left( k \right)}$.
Thus,
 $
 \mathbb{E} \varphi \left(X_{T_k}\right) 
 \approx
 \mathbb{E} \varphi \left( Y_{n\left( k \right)} \right) 
 \approx
 \frac{1}{S_k}
 \sum_{s=1}^{S_k} \varphi \left( Y^{\left\{ k ,s \right\}} \right) 
 $.
 We would like to find the number of simulations $S_k$ necessary for 
 \begin{equation}
 \label{eq:3.8}
 \mathbb{P} \left( 
\left\vert 
 \mathbb{E} \varphi \left( Y_{n\left( k \right)} \right) 
 -
 \frac{1}{S_k}
 \sum_{s=1}^{S_k} \varphi \left( Y^{\left\{ k ,s \right\}} \right) 
 \right\vert
 <
 AS_{tol} + RS_{tol} \left\vert \mathbb{E} \varphi \left( Y_{n\left( k \right)} \right)  \right\vert
\right) \geq 1 - \delta  ,
\end{equation}
where 
$AS_{tol}$ (resp. $RS_{tol}$) is the absolute (resp. relative) tolerance parameter
and 
$\delta \in \left] 0, 1 \right[$ provides the confidence level.
Following Section 3.4.1 of  \cite{GrahamTalay2013} 
we apply the Bikelis theorem (see, e.g., \cite{Petrov1995}), together with Komatsu's inequality, to 
deduce that \eqref{eq:3.8} holds under the condition 
\begin{align*}
\frac{\delta}{2} 
&  \geq
  \sqrt{\frac{2}{\pi}}
\frac{ \varsigma_k}{
\varepsilon_k \sqrt{S_k} + \sqrt{ 2 \varsigma_k^2 + \varepsilon_k^2 S_k}
}
e^{- \frac{\varepsilon_k^2 S_k}{2 \varsigma_k^2}} 
\\
& \quad
+
\frac{
\mathbb{E} \left( \left\vert \varphi \left( Y_{n\left( k \right)} \right) \right\vert^3  \right) 
+
\left\vert \mathbb{E} \varphi \left( Y_{n\left( k \right)} \right) \right\vert  
\mathbb{E} \left( \varphi \left( Y_{n\left( k \right)} \right)^2  \right)
-
2  \mathbb{E} \varphi \left( Y_{n\left( k \right)} \right) 
\mathbb{E} \left( \varphi \left( Y_{n\left( k \right)} \right) \left\vert \varphi \left( Y_{n\left( k \right)} \right) \right\vert  \right) 
}{ \sqrt{S_k} \left( \varsigma_k + \varepsilon_k \sqrt{S_k}  \right)^3}
,
\end{align*} 
where
$
\varepsilon_k = AS_{tol} + RS_{tol} \left\vert \mathbb{E} \varphi \left( Y_{n\left( k \right)} \right)  \right\vert
$
and
$
 \varsigma_k^2
 =
 \mathbb{E} \left(  \varphi \left( Y_{n\left( k \right)} \right) ^2 \right)
 - 
 \left(   \mathbb{E} \varphi \left( Y_{n\left( k \right)} \right)  \right) ^2
 $
 (see also, e.g., \cite{Szepessy2001,Mordecki2008,Bayer2014,Gobet2016,Hickernell2013}).
Similar to \cite{GrahamTalay2013},
estimating from a sample  the expected values we arrive at the following adaptive strategy.

\begin{sampling}
\label{sampling:1}
Consider the safety factor $ \mathfrak{sfac}_{max} > 1$, 
the absolute and relative tolerance sample  parameters $AS_{tol}, RS_{tol} > 0$,
the minimum and maximum sample sizes $S_{min}$, $S_{max}$,
and the  confidence level $ 1- \delta \in \left] 0, 1 \right[$.
Then:

\begin{enumerate}
 
 \item Set $M_* = M$, $S_k^{old} = 0$ and $S_k = S_{min}$ for all $k=1, \ldots, M$. 
 
 \item 
 For any $k=1, \ldots, M_*$,
 simulate a realization $y^{\left\{ k ,s \right\}}$ of $Y^{\left\{ k ,s \right\}}$ 
 for all $s = S_k^{old}+1, \ldots, S_k$,
and keep the old realizations  $y^{\left\{ k ,s \right\}}$ with $s = 1,  \ldots, S_k^{old}$.

 \item
 For any $k=1, \ldots, M_*$,
 compute
 $
 \bar{f}_{j,k} = \frac{1}{S_k} \sum_{s = 1}^{S_k} f_j \left( \varphi \left(  y^{\left\{ k ,s \right\}}  \right)  \right) 
 $
 for all $j = 1, \ldots, 4$,
 where
 \[
 f_1 \left( x \right) = x, \quad
 f_2 \left( x \right) = x^2, \quad
 f_3 \left( x \right) = x \left\vert x \right\vert, 
 \text{  and }
 f_4 \left( x \right) = \left\vert x \right\vert ^3. 
 \]
 Then,
 take
 $
 \bar{\varepsilon}_k =  AS_{tol} + RS_{tol} \left\vert  \bar{f}_{1,k} \right\vert
 $
 and
 $
\left( \bar{\varsigma}_k\right)^2 =  \bar{f}_{2, k} -  \left( \bar{f}_{1 k} \right)^2
$.

\item For all $k=1, \ldots, M_*$,
find $S_k^{new}$ such that
\begin{equation*}
 \sqrt{\frac{2}{\pi}}
\frac{ \bar{\varsigma}_k}{
\bar{\varepsilon}_k \sqrt{S_k^{new}} + \sqrt{ 2 \bar{\varsigma}_k^2 + \bar{\varepsilon}_k^2 S_k^{new}}
}
e^{- \frac{ \bar{\varepsilon}_k^2 S_k^{new}}{2 \bar{ \varsigma }_k^2}}
+
\frac{
 \bar{f}_{4, k} +  \left\vert  \bar{f}_{1, k} \right\vert \bar{f}_{2 , k} - 2  \bar{f}_1  \bar{f}_{3, k}
}{ \sqrt{S_k^{new}} \left( \bar{ \varsigma }_k + \bar{ \varepsilon }_k \sqrt{S_k^{new}}  \right)^3}
\leq
\frac{\delta}{2} .
\end{equation*}

\item
In case 
$\min \hspace{-1pt}\left\{ S_k^{new}, S_{max} \right\} \hspace{-1pt} \leq S_k$ for all $k= 1, \ldots, M_*$
we stop, 
and so $\mathbb{E} \varphi \left(X_{T_k}\right) $ is approximated by $\bar{f}_{1,k}$
for all $k=1,\ldots,M$.
Else, 
return to Step 2 after updating the values of $M_*$, $S_k^{old} $ and $S_k$ as follows:
for any $k=1, \ldots, M_*$ we set $S_k^{old} = S_k$ 
and define 
\begin{equation}
\label{eq:DeltaSk}
\Delta S_k 
=
\max \left\{ 0, 
\max_{j = k, \ldots, M_*} \left\{ 
\min \left\{ S_j^{new}, \mathfrak{sfac}_{max} \, S_j , S_{max} \right\} - S_j
\right\} 
\right\} .
\end{equation}
Then, the new value of $M_*$ is 
\[
M_* 
=
\begin{cases}
 M_*
 &
 \text{ if } \Delta S_{M_*} > 0 
\\
\min \left\{ k = 1, \ldots, M_* : \Delta S_k = 0 \right\} 
&
 \text{ if } \Delta S_{M_*} = 0 
\end{cases} ,
\]
and
$S_k$ is update to $S_k + \Delta S_k$ for all $k$  less than or equal to the new $M_*$. 
\end{enumerate}
\end{sampling}

\begin{remark}
In Sampling Method \ref{sampling:1},
from \eqref{eq:DeltaSk} we have $S_{k} \geq S_{k+1}$ for all $k = 1, \ldots, M_*-1$.
Some iterations of certain problems yield
$M_* < M$ like in Sections \ref{subsec:LinearScalar} and \ref{sec:SmallAdditiveNoise},
and so we have to increment the sample size of $Y_n$ only for $n \leq n\left( M_* \right) < N$,
which leads to a decrease in the computational cost of the algorithm. 
\end{remark}

\begin{remark}
\label{re:sampling:1}
In Step 5 of Sampling Method \ref{sampling:1},
we can alternatively take $M_* = M$ and update $S_k$ to 
$
S_k = \min  \left\{  \max \left\{ S_1^{new}, \ldots, S_M^{new} \right\},  \mathfrak{sfac}_{max} \, S_k, S_{max} \right\} 
$.
Hence,
$S_k$ does not depend on $k$,
and 
we have to simulate realizations of $Y_n$ for all $n= 0, \ldots, N$.
\end{remark}

\begin{remark}
Returning to the step 4 of Sampling Method \ref{sampling:1},
for all $x > 0$ we define 
\[
 \phi \left( x \right)
 =
\sqrt{\frac{2}{\pi}}
\frac{ \bar{\varsigma}_k}{
\bar{\varepsilon}_k \, x + \sqrt{ 2 \bar{\varsigma}_k^2 + \bar{\varepsilon}_k^2 \, x^2}
}
e^{- \frac{ \bar{\varepsilon}_k^2 x^2}{2 \bar{ \varsigma }_k^2}}
+
\frac{
 \bar{f}_{4, k} +  \left\vert  \bar{f}_{1, k} \right\vert \bar{f}_{2 , k} - 2  \bar{f}_1  \bar{f}_{3, k}
}{ x \left( \bar{ \varsigma }_k + \bar{ \varepsilon }_k \, x  \right)^3}
-
\frac{\delta}{2} .
\]
Hence,  $ \phi $ is strictly decreasing and smooth. 
If $ \phi \left( \sqrt{S_k} \right) < 0 $, then we take $S_k^{new} = S_k$.
Else, 
we can compute the root $x_0$ of $\phi$ by using a bisection-like method
in case $ \phi \left( \sqrt{S_{max}} \right) < 0 $,
and so we select $S_k^{new} = \left[ x_0^2 \right] +1$.
For this purpose, we here use  the function {\it fzero } of MATLAB.
We choose $S_k^{new} = S_{max}$ provided that $ \phi \left( \sqrt{S_{max}} \right) > 0 $.  
\end{remark}

\section{Numerical experiments}

\label{sec:NumericalExperiment}

\subsection{Basic linear scalar SDE}
\label{subsec:LinearScalar}

We compute $\mathbb{E} \log \left( 1 + \left( X_t \right)^2 \right)$,
where 
\begin{equation}
 \label{eq:Test1}
 X_{t}
=
X_{0}  + \int_{0}^{t} \sigma X_{s} \,  dW^1_{s} 
\hspace{1cm}
\text{for all }t \geq 0   
\end{equation}
with $ \sigma  =  10$ and $X_0 = 5$.
The scalar SDE \eqref{eq:Test1} is a model problem
for the numerical solution of SDEs with multiplicative noise whose diffusion coefficients can take large values (see, e.g., \cite{Higham2000,Milstein1998}).
Contrary to the fact that $\mathbb{E} \log \left( 1 + \left( X_t \right)^2 \right)$ converges to $0$ with exponential rate
as  $ t \rightarrow + \infty $,
the trajectories of the Euler scheme, applied to  \eqref{eq:Test1} with constant step-size,
grow excessively when the step-size is not small enough.
Since the coefficients  $b$ and  $\sigma_k$ of \eqref{eq:Test1} have bounded derivatives of any order,
 \eqref{eq:Test1}  satisfies Hypotheses \ref{hyp:SDE} and \ref{hyp:EDP} 
(see, e.g., \cite{GrahamTalay2013,Kloeden1992,Krylov1999,Milstein2004}),
and  Adaptive schemes  \ref{AdaptiveEM1} and \ref{AdaptiveO2_Extrap} fullfil  Hypothesis \ref{hyp:WeakOrden1}.
By construction, Adaptive schemes  \ref{AdaptiveEM1} and \ref{AdaptiveO2_Extrap}  satisfy Hypotheses \ref{hyp:AcotN}.

The solution of \eqref{eq:Test1} is
$
X_t = e^{ - \frac{1}{2} \sigma^2 \, t +  \sigma W^1_t } X_0
$.
We got the reference values for $\mathbb{E} \log \left( 1 + \left( X_{T_k} \right)^2 \right)$
given in Table \ref{Table:E1.1},
where $T_1 = 0.5$, $T_2 = 1$ and $T_3 = 40$,
by running Sampling Method \ref{sampling:1}.
To this end,
we choose $Y^{\left\{ k ,s \right\}}$  distributed according to the law of $X_{T_k}$,
and we select
$\delta = 0.01$,
the tolerances 
$ AS_{tol} = RS_{tol} = 10^{-6}$,
and the sample-size parameters 
$S_{min} = 10^8$,  $ S_{max} = 10^{14}$ and 
$\mathfrak{sfac}_{max} = 120$.
We estimate the error
\begin{equation}
\label{eq:abserror1}
 \bar{\epsilon}_1 \left( Y \right)
=
\max_{k=1,2,3}
\left\vert
\mathbb{E} \log \left( 1 + \left( Y_{n\left( k \right)} \right)^2 \right) - \mathbb{E} \log \left( 1 + \left( X_{T_k} \right)^2 \right)
\right\vert ,
\end{equation}
where 
\begin{equation}
 \label{eq:def_nk}
 n\left( k \right)  = \inf \left\{n : \tau_n = T_k \right\} < + \infty  ,
\end{equation}
and $Y$ stands for a numerical scheme 
such that the law of $ Y_{n\left( k \right)}$ approximates the distribution of $ X_{T_k}$.
Better approximations are associated with lower values of the tolerance parameters.

\renewcommand{\arraystretch}{1}
\renewcommand\tabcolsep{3pt}
\begin{table}[tb]
\caption{References values for the solution of \eqref{eq:Test1}.}
\begin{center}
\footnotesize
\def\arraystretch{1.5}

\begin{tabular}{|c||c|c|c|}
\hline
$T_k$ & $0.5$ & $1$ & $40$ 
\\
\hline 
$\mathbb{E} \log \left(1 + \left( X_{T_k} \right)^2 \right)$ &
$1.95279045\cdot 10^{-3}$ & $3.0047531\cdot 10^{-6}$ & $ 0 $ 
\\
\hline
\end{tabular}

\end{center}

\label{Table:E1.1}
\end{table}

\begin{figure}[tb]
\centering
 \resizebox{15cm}{!}{\includegraphics{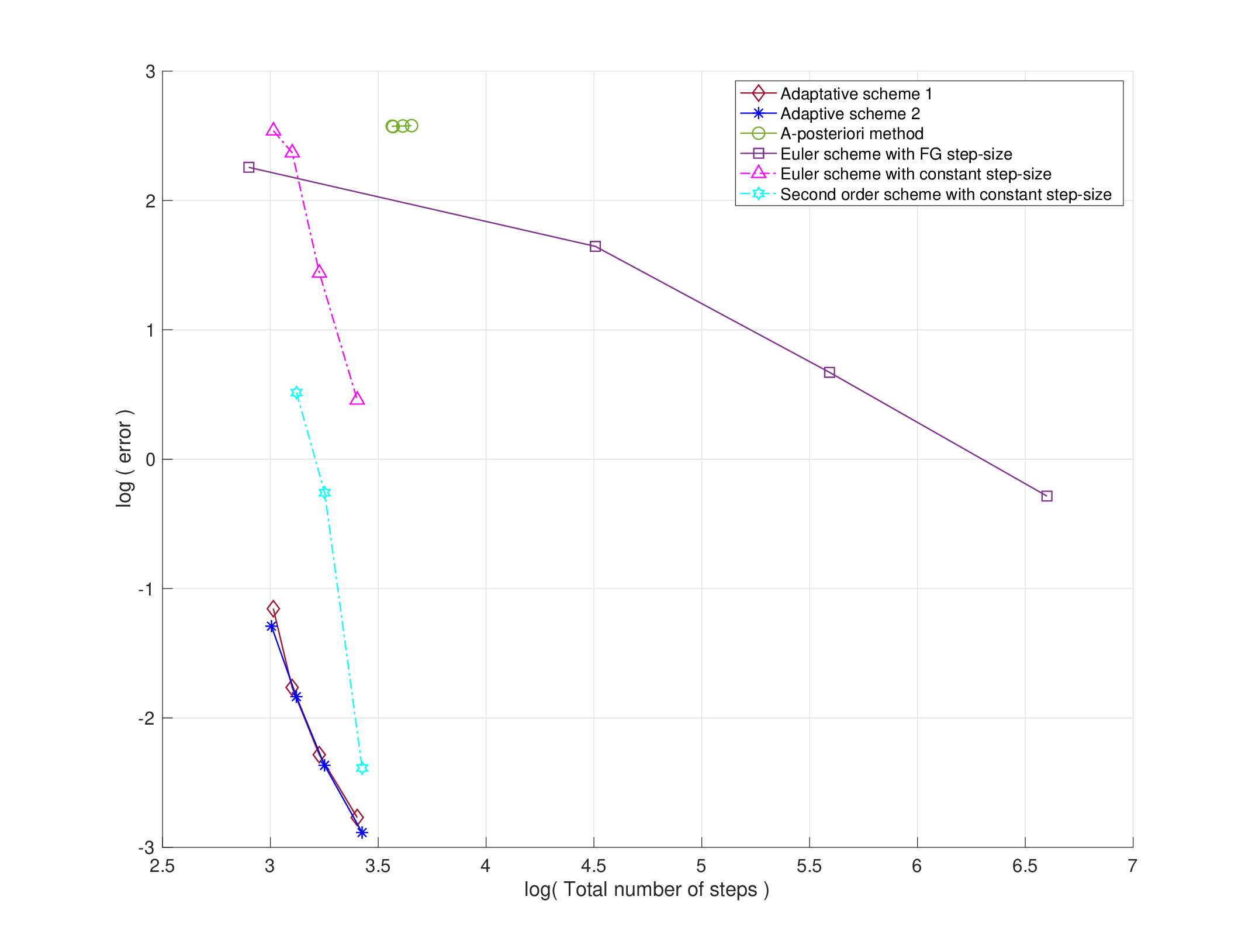}}
\caption{
The base $10$ logarithm of the error \eqref{eq:abserror1} as a function of  
the base $10$ logarithm of  the expected value of the total number of steps 
in the numerical solution of  \eqref{eq:Test1} with $ \sigma  =  10$ and $X_0 = 5$.
Better approximations are associated with lower values of the tolerance parameters
$10^{-2}, 10^{-3}, 10^{-4}$ and $10^{-5}$. }
\label{fig:Errores_Bilineal}
\end{figure}

Figure \ref{fig:Errores_Bilineal} plots $\bar{\epsilon}_1 \left( Y \right)$ 
as a function of the mean value of the total number of steps used by each scheme $Y$.
First,
Figure \ref{fig:Errores_Bilineal} presents estimations of the error \eqref{eq:abserror1}
with $Y$ being Adaptive scheme  \ref{AdaptiveEM1} (diamond) and  Adaptive scheme  \ref{AdaptiveO2_Extrap} (asterisk)
as a function of 
the mean value of the number of steps used by 
Adaptive schemes  \ref{AdaptiveEM1} and \ref{AdaptiveO2_Extrap}   to arrive at $T = 40$
with tolerances $Atol = Rtol =  10^{-2}, 10^{-3}, 10^{-4}, 10^{-5} $.
 For simplicity we take $Atol = Rtol$.
We choose $\mathfrak{fac}_{max} = 20$,
$\Delta_{max} = 1 / 2$, 
and
$\Delta_{min} = 2 \, eps$,
where $eps$ is the distance from $1$ to the next larger double precision number
($2^{-52}$ in MATLAB).
We sample Adaptive schemes  \ref{AdaptiveEM1} and \ref{AdaptiveO2_Extrap} by using Sampling Method \ref{sampling:1} 
with parameters 
$\delta = 0.01$, $AS_{tol} = RS_{tol} = 10^{-4}$, $S_{min} = 10^5$,  $ S_{max} = 10^9$ 
and $\mathfrak{sfac}_{max} = 120$. 
According to Figure \ref{fig:Errores_Bilineal},
there is steady decrease in the error $\bar{\epsilon}_1 $ 
that is guided by the  the tolerance parameters.
At the same time,
the mean value of the number of steps only increases slightly.
Moreover, we estimate the CPU time of 
$10^5$ realizations of Adaptive schemes  \ref{AdaptiveEM1} and \ref{AdaptiveO2_Extrap} in a 3,3 GHz Intel Core i5.
We obtain that the ratio between the CPU time of  Adaptive scheme \ref{AdaptiveO2_Extrap} and the  CPU time of Adaptive scheme  \ref{AdaptiveEM1}  is equal to $0.9914$, $1.0920$,  $1.1214$, $1.0896$ 
when $Atol = Rtol =  10^{-2}, 10^{-3}, 10^{-4}, 10^{-5} $, respectively.
Then, Adaptive schemes  \ref{AdaptiveEM1} and \ref{AdaptiveO2_Extrap}  show  similar performance in this example.


Second,
we compare the application of our selection mechanisms of the step-sizes 
with the use of constant step-sizes.
Following the notation of Subsection \ref{subsec:Extrapolation} we define the recursive scheme:
\begin{equation}
\label{eq:SecondOrderScheme}
\begin{aligned}
 Y_{n+1} 
& =
Y_{n} 
+  b \left( \tau_n ,  Y_{n} \right) \left( \tau_{n+1}  - \tau_{n} \right) 
+ \frac{1}{2} \mathcal{L}_0   b \, \left( \tau_n ,  Y_{n} \right) \left( \tau_{n+1}  - \tau_{n} \right)^2 
+ \sum_{k=1}^{m} \sigma_k \left(  \tau_n , Y_{n}  \right)  \sqrt{\tau_{n+1}  - \tau_{n}} \, \xi^{ k }_{n+1} 
 \\
 & \quad 
 + \frac{1}{2}  \sum_{k=1}^{m} \left( 
 \mathcal{L}_k b \left(  \tau_n ,  Y_{n}  \right)  +  \mathcal{L}_0  \sigma_k \left(  \tau_n ,  Y_{n}  \right)   \right)
 \left( \tau_{n+1}  - \tau_{n} \right)^{3/2}  \xi^{ k }_{n+1}  
 + \sum_{k, \ell =1}^{m} \mathcal{L}_k  \sigma_{\ell} \left(  \tau_n ,  Y _{n}  \right)  
 \left( \tau_{n+1}  - \tau_{n} \right) \xi^{ k , \ell}_{n+1} .
\end{aligned}
\end{equation}
Figure \ref{fig:Errores_Bilineal}  provides estimations of the error \eqref{eq:abserror1}
appearing in solving \eqref{eq:Test1} by the Euler scheme  \eqref{eq:Euler-Maruyama} (triangle) and
the second weak  order scheme \eqref{eq:SecondOrderScheme} (hexagram)
both with constant step-sizes.
As previously proceeded,
we apply Sampling Method \ref{sampling:1} with parameters 
$\delta = 0.01$, $AS_{tol} = RS_{tol} = 10^{-4}$, $S_{min} = 10^5$,  $ S_{max} = 10^9$ 
and $\mathfrak{sfac}_{max} = 120$.
The error corresponding to the scheme \eqref{eq:SecondOrderScheme} with step-size $0.0395$ ($1013$ integration steps) grow towards $+ \infty$, and so it has not been drawn in Figure \ref{fig:Errores_Bilineal}.
Figure \ref{fig:Errores_Bilineal}  shows that, for a similar number of total recursive steps,
Adaptive schemes  \ref{AdaptiveEM1} and \ref{AdaptiveO2_Extrap} greatly improve the accuracy of
their constant step-size versions,
i.e., 
the Euler scheme  \eqref{eq:Euler-Maruyama} and the second weak  order scheme \eqref{eq:SecondOrderScheme}
both with constant step-size.

We observe that Adaptive schemes  \ref{AdaptiveEM1} and \ref{AdaptiveO2_Extrap} 
show an almost sure asymptotically stable behavior, 
which contrasts with the unstable behavior of  the underlying schemes with constant step-size.
This is part of the reason why 
Adaptive schemes  \ref{AdaptiveEM1} and \ref{AdaptiveO2_Extrap}  achieve good accuracy 
with a small number steps. 
From the proof of Theorem 4.2 of \cite{Higham2007} we deduce that 
the Euler-Maruyama scheme \eqref{eq:Euler-Maruyama} applied to  \eqref{eq:Test1} 
with constant step-size $\tau_{n+1}  - \tau_{n} = \Delta $
is almost sure exponentially stable if
$
 5 \, \sigma^4  \Delta^2 + 10.5 \, \sigma^2  \Delta - 1 < 0 ,
$
which implies 
$
\sigma^2  \Delta < \left( -5.25 + \sqrt{130.25} \right) /10 \approx 0.61627
$.
This stability analysis suggests us to use the Euler-Maruyama scheme with a constant step-size less than 0.0061627,
with yields a uniform partition with at least $6491$ steps.
Since $\log_{10} \left( 6491 \right) \approx 3.8112$, 
Figure \ref{fig:Errores_Bilineal} shows that Adaptive scheme  \ref{AdaptiveEM1} 
needs to take fewer steps to achieve a good accuracy.

Third, since there is no well-established automatic selection mechanism 
of the step-sizes of weak schemes based on local error control,
we consider the adaptive Euler scheme introduced by  \cite{Szepessy2001}.
We apply the stochastic time stepping algorithm introduced by  \cite{Szepessy2001}
to solve \eqref{eq:Test1} by the Euler scheme 
in each time interval  $\left[ T_{k-1}, T_k \right] $, 
where $T_0 = 0$ and $k=1,2,3$.
That is, for each Brownian path we generate a discretization of $\left[ 0, T_1 \right] $ for \eqref{eq:Euler-Maruyama}
by means of the stochastic time stepping algorithm of \cite{Szepessy2001},
then we apply the same method to obtain a discretization of  $\left[ T_1, T_2 \right] $, and so on.
Plotting circles,
 Figure \ref{fig:Errores_Bilineal} presents the error  $\bar{\epsilon}_1 \left( Y \right)$
 for the Euler scheme \eqref{eq:Euler-Maruyama} with $\tau_{n+1}  -  \tau_{n}$
 generated as we just described.
 In order to provide more details, 
 Table \ref{Table:MaxErrorApost} gives estimations of the mean value of $n\left( 3 \right)$
(i.e., the number of nodes of the final time discretization),
and
the expected value of the total number of steps taken by the Euler-Maruyama scheme 
for each Brownian motion trajectory
(i.e., the sum of the number of nodes on all partitions generated by 
the algorithm for each realization of the Brownian motion).
Following   \cite{Szepessy2001} we include 
\begin{equation}
 \label{eq:Es}
 \mathrm{E}_S = \max_{k=1,2,3} \frac{1.65}{\sqrt{S_k}} \left(
\frac{1}{S_k}\sum_{\ell=1}^{S_k}  \varphi \left( Y^{\left\{k,\ell \right\}} \right)^2 
-
\left( \frac{1}{S_k}\sum_{\ell=1}^{S_k} \varphi \left( Y^{\left\{k,\ell \right\} } \right) \right)^2
\right)^{1/2} ,
\end{equation}
where
$\varphi \left( x \right)  = \log \left( 1 + x^2 \right) $
and
we use the notation of Section \ref{Sec:SampleSize}.
We restrict the sample size $S_k$ to $2 \times 10^6$ in order to reduce the total runtime to some days,
and we take the minimal step-size equal to $10^{-7}$.
Figure \ref{fig:Errores_Bilineal}  and Table \ref{Table:MaxErrorApost}  indicate that
the random a-posteriori strategy given by \cite{Szepessy2001}
has a poor accuracy in this example.
In case we compute $\mathbb{E}  X_t $ instead of  $\mathbb{E} \log \left( 1 + \left( X_t \right)^2 \right)$
we can check theoretically that the stochastic time stepping algorithm described in  \cite{Szepessy2001}
does not refine the initial partition of each time interval $\left[ T_{k-1}, T_k \right] $.
This follows by substituting $ b \left( t, x \right) = 0$ and $ \frac{d^2}{dx^2} \, \varphi \left( x \right) = 0$
into the formulation of the stochastic time stepping method given by  \cite{Szepessy2001}.

\renewcommand{\arraystretch}{1}
\renewcommand\tabcolsep{3pt}
\begin{table}[tb]
\caption{Error \eqref{eq:abserror1} arising from the solution to  \eqref{eq:Test1} with $ \sigma  =  10$ and $X_0 = 5$ by using the stochastic time stepping algorithm designed by  \cite{Szepessy2001} with initial uniform step-size $0.1$.}
\begin{center}
\footnotesize
\def\arraystretch{1.5}

\begin{tabular}{|c||c|c|c|c|}
\hline  
$Tol $ &                                $10^{-2}$      & $ 10^{-3}$    & $10^{-4}$   &  $10^{-5}$  
\\
\hline 
$ \bar{\epsilon}_1$ &    $378.9119$  &  $376.6461$ &   $374.8767$     & $372.3506$
\\
\hline
Mean total number of steps &    $4514.6$  &  $4110.4$     &  $3673.9$   & $3707.8$
\\
\hline
$\mathbb{E} \, n\left( 3 \right)$ & $1030.1$ &   $1001.6$       & $971.9138$   &  $980.7327$
\\
\hline
$\mathrm{E}_S$ &                         $0.1423$   &   $0.1435$     &  $0.1443$   &  $0.1462$
\\
\hline
\end{tabular}	
\end{center}
	
\label{Table:MaxErrorApost}
\end{table}

Finally,
in \cite{FangGiles2020} is proposed to solve scalar SDEs
and certain class of multidimensional SDEs by using the Euler-Maruyama scheme \eqref{eq:Euler-Maruyama} with  
the step-size
\begin{equation}
\label{eq:PasoFG} 
\tau^{FG}_{n+1} 
= 
\tau^{FG}_{n}  +
Tol * \max \left\{ 1, \left\Vert  Y_{n}   \right\Vert    \right\}  / \max \left\{ 1, \left\Vert  b \left( Y_{n} \right)  \right\Vert  \right\}  
\end{equation}
(see also \cite{KellyLord2016} for related strategies).
Applying this adaptive strategy to \eqref{eq:Test1} we get 
\[
\tau^{FG}_{n+1} 
= 
\tau^{FG}_{n}  + Tol * \max \left\{ 1, \left\Vert  Y_{n}   \right\Vert    \right\}  
,
\]
which does not depend on $\sigma$.
According to  Figure \ref{fig:Errores_Bilineal}  
the Euler scheme with nodes $\tau^{FG}_{n}$, 
which is represented by squares,
achieves poor performance for solving \eqref{eq:Test1} 
with tolerances $Tol = 10^{-2}, 10^{-3}, 10^{-4}, 10^{-5}$. 
Due to the high number of steps per realization,
we restrict the sample-size corresponding to  
the tolerances $Tol =10^{-4} $ and  $Tol =10^{-5}$ to $1.2 \times 10^7$ and $10^5$, respectively.

\subsection{SDE with small additive noise}
\label{sec:SmallAdditiveNoise}

From \cite{Biscay1996} we take the linear non-autonomous SDE with additive noise
\begin{equation}
\label{JCarlos}
X_{t}
=
X_0
- 
\int_{0}^{t} s^2 X_{s}  \, ds + \sigma  \int_{0}^{t}  \frac{\exp \left( - s^3 / 3 \right) }{s+1} dW^{1}_{s}
\hspace{1cm}
\text{for all }t \geq 0  ,
\end{equation}
where 
$ \sigma = 1$ and $X_0 = 1$.
Then
\begin{equation}
\label{JCarlos_solution}
X_{t}
=
\exp \left( - \frac{t^3}{ 3} \right) \left( X_0 + \sigma \int_0^t  \frac{1}{s+1}  dW^{1}_{s} \right) ,
\end{equation}
and so $X_{t}$ converges exponentially fast to $0$ as $t \rightarrow + \infty$.
In case $s$ is not small, 
the drift coefficient $b\left( s, x \right) = - s^2 x$ can take large values.
Hence,
the Euler scheme with uniform step-size has numerical instabilities solving \eqref{JCarlos}
(see, e.g.,  \cite{Biscay1996,Carbonell2006}).
We compute $\mathbb{E} \left(  \left( X_{T_k} \right)^3 \right)$,
with 
$k=1,2,3$, 
$T_1 = 0.5$, $T_2 = 1$ and $T_3 = 20$.
From \eqref{JCarlos_solution} we deduce that
\[
\mathbb{E} \left(  \left( X_t \right)^3 \right)
=
\exp \left( - t^3 \right) \left( 
\mathbb{E} \left(  \left( X_0 \right)^3 \right) + 3 \sigma^2 \mathbb{E} \left(  X_0 \right) \left( 1 -  \frac{1}{t+1}  \right)
\right) .
\]
Similar to Section \ref{subsec:LinearScalar}
the coefficients  $b$ and  $\sigma_k$ of \eqref{JCarlos} have bounded derivatives of any order
in $\left[ 0 , T_3 \right] \times \mathbb{R}$,
and so  \eqref{JCarlos}  satisfies Hypotheses \ref{hyp:SDE} and \ref{hyp:EDP},
and  Adaptive schemes  \ref{AdaptiveEM1} and \ref{AdaptiveO2_Extrap} fullfil  Hypothesis \ref{hyp:WeakOrden1}.
By construction, Adaptive schemes  \ref{AdaptiveEM1} and \ref{AdaptiveO2_Extrap}  satisfy Hypotheses \ref{hyp:AcotN}.

\begin{figure}[tb]
\centering
 \resizebox{12cm}{!}{\includegraphics{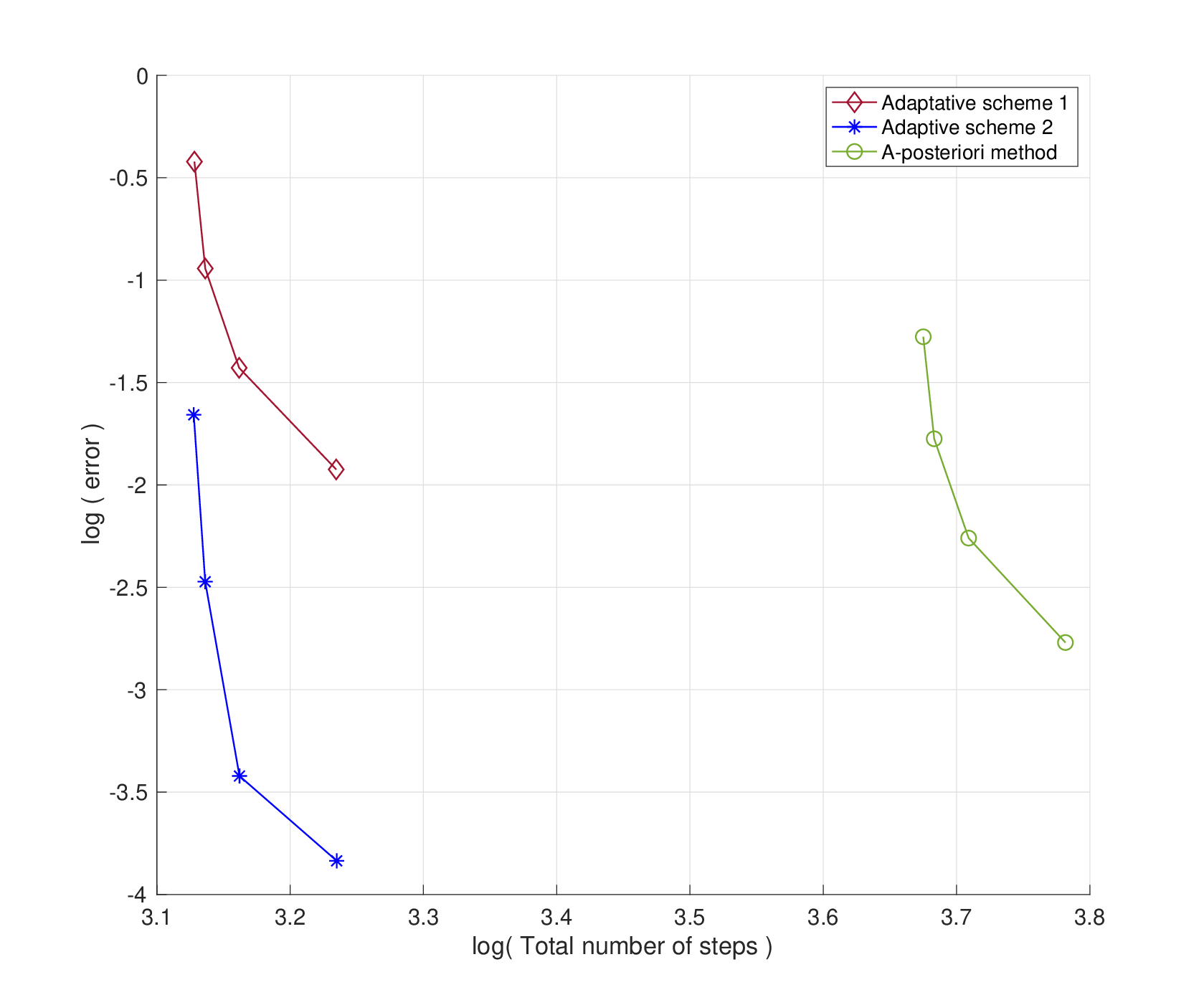}}
\caption{
The base $10$ logarithm of the weak error $ \bar{\epsilon}_2$ as a function of  
the base $10$ logarithm of  the expected value of the total number of steps 
in the numerical solution of the SDE with small additive noise  \eqref{JCarlos}
by  Adaptive schemes  \ref{AdaptiveEM1} and \ref{AdaptiveO2_Extrap},
and the stochastic time stepping algorithm given by  \cite{Szepessy2001} with initial uniform step-size $0.5$.
Better approximations are associated with lower values of the tolerance parameters.
}
\label{fig:Errores1_BJC}
\end{figure}

\begin{figure}[tb]
\centering
 \resizebox{12cm}{!}{\includegraphics{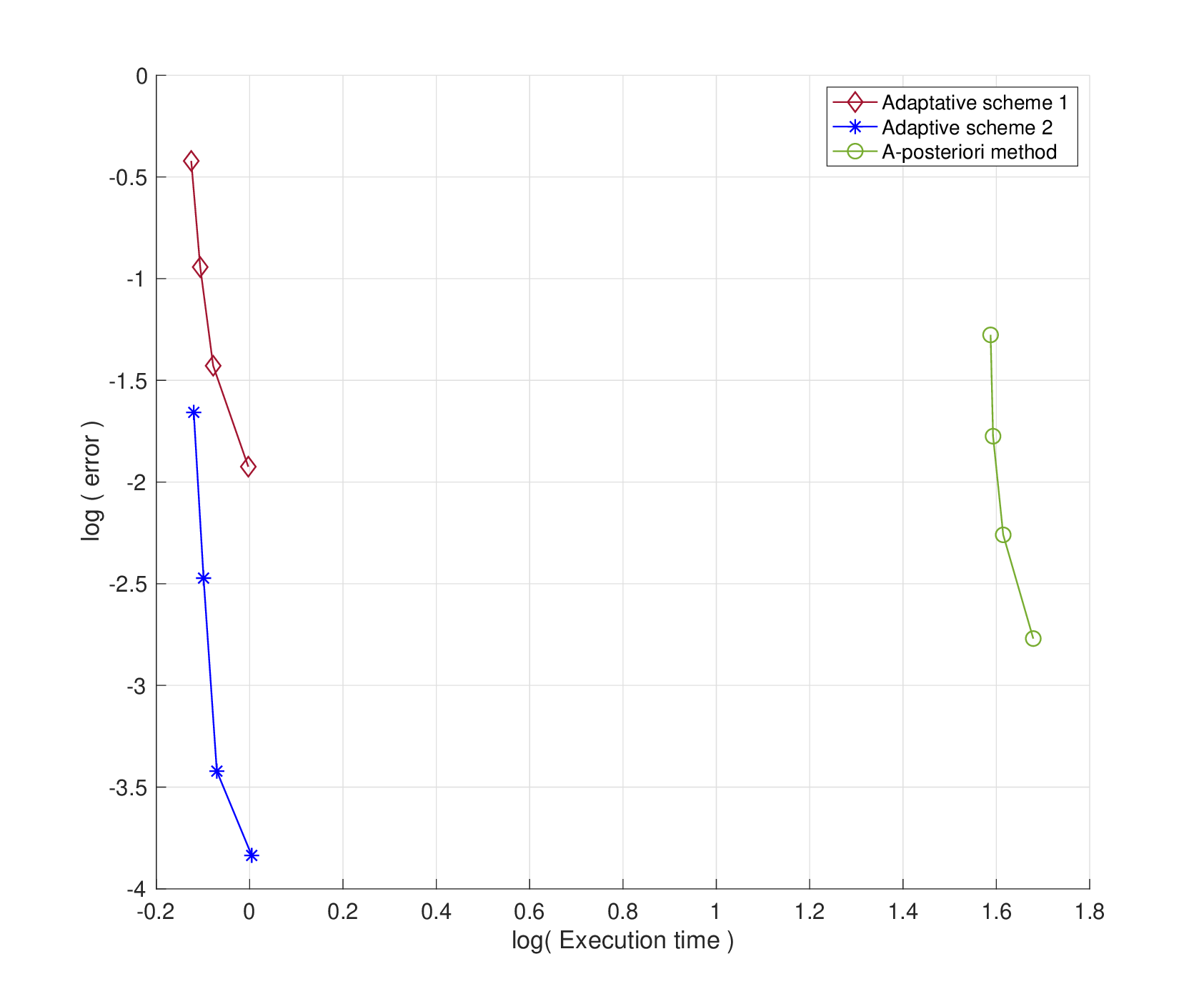}}
\caption{
The base $10$ logarithm of the weak error $ \bar{\epsilon}_2$ as a function of  
the base $10$ logarithm of  the execution time (in minutes)
in the numerical solution of the SDE with small additive noise  \eqref{JCarlos}
by  Adaptive schemes  \ref{AdaptiveEM1} and \ref{AdaptiveO2_Extrap},
and the stochastic time stepping algorithm given by  \cite{Szepessy2001} with initial uniform step-size $0.5$.
Better approximations are associated with lower values of the tolerance parameters.
}
\label{fig:Errores2_BJC}
\end{figure}

Figure \ref{fig:Errores1_BJC} displays estimations of the error 
\begin{equation}
\label{eq:abserror2}
 \bar{\epsilon}_2 \left( Y \right)
=
\max_{k=1,2,3}
\left\vert
\mathbb{E} \left( \left( Y_{n\left( k \right)} \right)^3 \right)  - \mathbb{E} \left( \left( X_{T_k} \right)^3 \right)
\right\vert 
\end{equation}
as a function of the mean value of the total number of steps used by the scheme $Y$,
where  $n\left( k \right)$ is defined by \eqref{eq:def_nk}.
We estimate the error \eqref{eq:abserror2} in case that $Y$ is equal to  
Adaptive schemes \ref{AdaptiveEM1} and \ref{AdaptiveO2_Extrap}
with tolerances  $Atol = Rtol  =10^{-2}, 10^{-3}, 10^{-4},10^{-5}$,
and parameters  
$\Delta_{min} = 2 \, eps$, $\Delta_{max} = 1 / 2$,  $\mathfrak{fac}_{max} = 20$.
To this end, we apply  Sampling Method \ref{sampling:1}
with  parameters 
$\delta = 0.01$, $AS_{tol} = RS_{tol} = 10^{-4}$, $S_{min} = 10^5$,  $ S_{max} = 10^9$ 
and $\mathfrak{sfac}_{max} = 120$.
Figure \ref{fig:Errores1_BJC}  
shows that 
Adaptive schemes  \ref{AdaptiveEM1} and \ref{AdaptiveO2_Extrap} solve properly \eqref{JCarlos},
improving their accuracy as the tolerance parameters $Atol$ and $Rtol $ decrease.
Moreover, 
the number of steps taken by
Adaptive schemes   \ref{AdaptiveEM1} and \ref{AdaptiveO2_Extrap}  to arrive at $T = 20$
have not increased  significantly as $Atol = Rtol $ decreases. 
This leads to a good computational efficiency.
Adaptive scheme \ref{AdaptiveO2_Extrap} is significantly more accurate than Adaptive scheme  \ref{AdaptiveEM1}
in this example.
Namely, the error of Adaptive scheme \ref{AdaptiveO2_Extrap}  with $ Atol = Rtol  = 10^{-5}$ 
is in the range of the sampling error.

Figure \ref{fig:Errores1_BJC} compares 
Adaptive schemes  \ref{AdaptiveEM1} and \ref{AdaptiveO2_Extrap} 
with the stochastic time stepping algorithm developed by  \cite{Szepessy2001}.
Indeed,  Figure \ref{fig:Errores1_BJC} presents  the error  $\bar{\epsilon}_2 \left( Y \right)$
for the Euler scheme \eqref{eq:Euler-Maruyama} with $\tau_{n+1}  -  \tau_{n}$
obtained by the stochastic time stepping algorithm given by  \cite{Szepessy2001},
which is applied in each time interval $\left[ T_{k-1}, T_k \right]$ as in Section \ref{subsec:LinearScalar}.
With circles we plot  $\bar{\epsilon}_2 \left( Y \right)$ as a function of the mean value of the total number of steps taken by the Euler-Maruyama scheme for each Brownian motion trajectory.
We choose the tolerance parameter $Tol = 10^{-2}, 10^{-3}, 10^{-4}, 10^{-5}$.
According to Figure \ref{fig:Errores1_BJC},
the accuracy of the  Euler scheme with the stochastic adaptive strategy introduced by \cite{Szepessy2001}
is very good, but it requires to compute many steps.
It is worth pointing out that 
the tolerance parameters $Atol$ and $Rtol$ control 
local errors of Adaptive schemes  \ref{AdaptiveEM1} and \ref{AdaptiveO2_Extrap}
via a local discrepancy function.
In contrast,
the parameter $Tol$
controls the global error of the stochastic time stepping method of \cite{Szepessy2001},
and so  we should not compare directly 
the errors $ \bar{\epsilon}_2$ for the same value of the parameters $Atol=Rtol$ and $Tol$.

Similar to Section \ref{subsec:LinearScalar},
we tailor the MATLAB code of the  stochastic adaptive strategy given by \cite{Szepessy2001}
to fit the characteristics of \eqref{JCarlos},
though it  may be not optimum.
Figure \ref{fig:Errores1_BJC} presents $\bar{\epsilon}_2 \left( Y \right)$
as a function of 
the running times spent by $10^6$ trajectories of our 
MATLAB codes of Adaptive schemes \ref{AdaptiveEM1}, \ref{AdaptiveO2_Extrap}
and of the adaptive strategy introduced by \cite{Szepessy2001}. 
Figure \ref{fig:Errores1_BJC}  shows a big difference between the execution times of
Adaptive schemes \ref{AdaptiveEM1}, \ref{AdaptiveO2_Extrap} and the strategy designed by \cite{Szepessy2001}.
The complexity of the new adaptive schemes --based on a priori discrepancy functions--
is lower than that of the a-posteriori strategy designed by \cite{Szepessy2001}.

In order to study the impact of the new adaptive time-stepping strategies,
we solve \eqref{JCarlos} by the Euler scheme  \eqref{eq:Euler-Maruyama} 
(resp. the second weak  order scheme \eqref{eq:SecondOrderScheme}) with the constant step-sizes
$0.01489$, $0.01378$, and $0.01165$
(resp. $0.014906$, $0.01377$, and $0.01164$),
which arise from dividing $T_3$ by 
the mean value of the number of steps taken by Adaptive scheme \ref{AdaptiveEM1}
(resp.  Adaptive scheme \ref{AdaptiveO2_Extrap})
with tolerance parameters $Atol = Rtol  = 10^{-2}, 10^{-4}, 10^{-5}$.
We apply the Sampling Method \ref{sampling:1} as specified in Table 2, i.e., 
with  parameters 
$\delta = 0.01$, $AS_{tol} = RS_{tol} = 10^{-4}$, $S_{min} = 10^5$,  $ S_{max} = 10^9$ 
and $\mathfrak{sfac}_{max} = 120$.
The  error $\bar{\epsilon}_2$ corresponding to the schemes \eqref{eq:Euler-Maruyama} and  \eqref{eq:SecondOrderScheme} with constant step-size grow towards $+ \infty$ in all cases except for the Euler scheme  \eqref{eq:Euler-Maruyama} with the step size $0.01165$ that is equal to $0.01942$.
Then,
Adaptive schemes  \ref{AdaptiveEM1} and \ref{AdaptiveO2_Extrap} avoid the numerical instabilities 
of the schemes \eqref{eq:Euler-Maruyama} and  \eqref{eq:SecondOrderScheme} 
by automatically adjusting their step-sizes.

\subsection{Stochastic Landau equation}
\label{sub:Stochastic_Landau}

\renewcommand{\arraystretch}{1.5}
\renewcommand\tabcolsep{3pt}
\begin{table}[tb]
\caption{ References values for $\mathbb{E} \log \left( 1 + \left( X_{T_k} \right)^2 \right)$,
where $X$ solves the stochastic Landau equation \eqref{eq:Test2}. }
\begin{center}
\footnotesize
\def\arraystretch{1.5}
\begin{tabular}{|c||c|c|c|}
\hline
$T_k$ & $0.5$ & $1$ & $20$ 
\\ \hline
Ex1  &
 $0.660792203$  &  $0.43373205$  & $0.068848214$ 
 \\\hline
Ex2  &
 $0.107344374$  &  $0.170826812$  & $0.112452505$ 
\\\hline
\end{tabular}
\end{center}

\label{Table:E2.1}
\end{table}

\begin{figure}[tb]
\centering
 \resizebox{12cm}{!}{\includegraphics{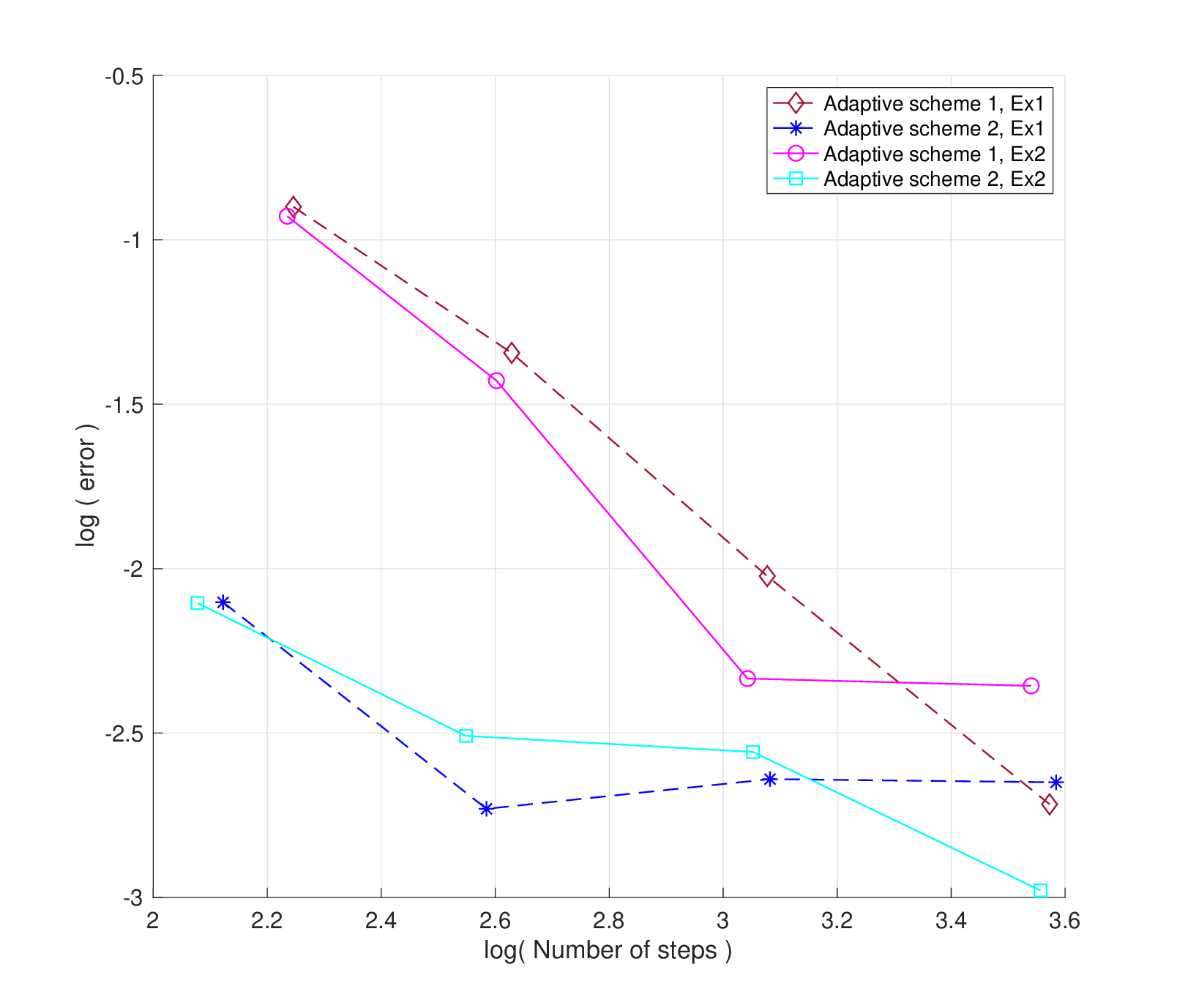}}
\caption{
The base $10$ logarithm of the weak error $ \bar{\epsilon}_1$ as a function of  
the base $10$ logarithm of  the expected value of the number of steps 
in the numerical solution of the stochastic Landau equation  \eqref{eq:Test2}.
Larger number of steps are associated with lower values of the tolerance parameters
$Atol = Rtol =  10^{-2}, 10^{-3}, 10^{-4}, 10^{-5} $.
}
\label{fig:Errores_GL}
\end{figure}

By considering time fluctuations in the bifurcation parameter of the Landau-Stuart ordinary differential equation 
we obtain 
\begin{equation}
\label{eq:Test2}
X_{t}
=
X_0
+
\int_{0}^{t} \left(  \left( a + \sigma^2/2 \right) X_s -   \left( X_s \right) ^3  \right)  ds
 + 
 \int_{0}^{t} \sigma \, X_{s} \, dW^{1}_{s}
 \hspace{1cm}
\text{for all }t \geq 0   ,
\end{equation}
where $X_{t}$ takes values in $ \mathbb{R}$, $a \in \mathbb{R}$ and $\sigma > 0$
(see, e.g., \cite{Arnold1998,Pavliotis2014}).
In addition to a large diffusion coefficient (see, e.g., Section \ref{subsec:LinearScalar}),
we face the difficulties arising from a saturating cubic term.
The stochastic Landau equation \eqref{eq:Test2} has been used to study stochastic bifurcations 
(see, e.g., \cite{Arnold1998,Pavliotis2014})
and to test numerical schemes (see, e.g., \cite{Higham2007,Hutzenthaler2015,Mora2017,Moro2007}).
In this example, $\sigma_1$ has bounded derivatives of any order
 and 
 \[
 x \, b \left( x \right) =  \left( a + \sigma^2/2 \right) x^2 -  x^4
 \qquad \qquad \text{for all } x \in \mathbb{R}.
 \]
 Therefore, the monotone condition \eqref{eq:MonotoneCondition} holds, 
 and so the SDE \eqref{eq:Test2} satisfies  Hypothesis \ref{hyp:SDE}.
 We leave open the problem if   \eqref{eq:Test2} satisfies Hypothesis \ref{hyp:EDP}.
 In this research direction, note that 
 the drift coefficient $b$ is not globally Lipschitz continuous, \eqref{eq:Test2}  is not  hypoelliptic,
 and 
  \eqref{eq:Test2} does not belong to the class of SDEs treated by  \cite{Bossy2021} 
since $\sigma_1 \left( x \right) =  \sigma \, x$.

For any $k = 1, 2, 3$ we calculate $\mathbb{E} \log \left( 1 + \left( X_{T_k} \right)^2 \right)$,
where $T_1 = 0.5$, $T_2 = 1$ and $T_3 = 20$, in the following situations:
\begin{description}

\item[\textbf{Ex1:}]  $ a = -0.1 $, $ \sigma = 2$ and $X_0 = 5$.

\item[\textbf{Ex2:}]  $ a = 0.1 $, $ \sigma = 2$ and $X_0 = 0.1$.

\end{description}
In Example \textbf{Ex1} we have that $a < 0$,
and so 
$ \mathbb{E} \left(  \left\Vert X_t \right\Vert ^p \right) = 0 $
converges exponentially fast to $0$ for some $p \in \left] 0, 1 \right[$
(see, e.g., \cite{Higham2007}).
Hence, 
\[ 
\lim_{t \rightarrow + \infty} \mathbb{E}  \log \left( 1 + \left(X_t \right)^2 \right) = 0
\]
due to
$ \log \left( 1 + x^2 \right) \leq 2 x^p / p$ for any $x \geq 0$.
On the other hand,
in  Example \textbf{Ex2}
the SDE \eqref{eq:Test2} has three invariant forward Markov measures since $a > 0$ 
(see, e.g., \cite{Arnold1998}).
Table \ref{Table:E2.1} presents the reference values for  $\mathbb{E} \log \left( 1 + \left( X_t \right)^2 \right)$,
which have been computed by  the Euler scheme \eqref{eq:Euler-MaruyamaDelta} with $\Delta = 0.0005$,
together with the Sampling Method \ref{sampling:1} with parameters 
$S_{min} = 10^6$,  $ S_{max} = 10^{14}$, $\mathfrak{sfac}_{max} = 120$, $\delta = 0.01$
and $AS_{tol} = RS_{tol} = 0.5 \cdot 10^{-4}$.

Figure \ref{fig:Errores_GL} plots  
the error \eqref{eq:abserror1} as a function of the mean value of the number of steps taken by Adaptive schemes  \ref{AdaptiveEM1} and  \ref{AdaptiveO2_Extrap} to arrive at $T = 20$  in the numerical solution of the examples \textbf{Ex1} and \textbf{Ex2}.
As in previous numerical experiments we take  
$Atol = Rtol =  10^{-2}, 10^{-3}, 10^{-4}, 10^{-5} $,
and
$\Delta_{min} = 2 \, eps$, $\Delta_{max} = 1 / 2$,  $\mathfrak{fac}_{max} = 20$.
We also apply  Sampling Method \ref{sampling:1}
with  
$\delta = 0.01$, $AS_{tol} = RS_{tol} = 10^{-4}$, $S_{min} = 10^5$,  $ S_{max} = 10^9$ 
and $\mathfrak{sfac}_{max} = 120$.
Figure \ref{fig:Errores_GL}  shows the good performance of 
Adaptive schemes  \ref{AdaptiveEM1} and  \ref{AdaptiveO2_Extrap}
that accurately solve  \eqref{eq:Test2}, particularly Adaptive scheme \ref{AdaptiveO2_Extrap}.
The Euler scheme  \eqref{eq:Euler-Maruyama} and the second weak  order scheme \eqref{eq:SecondOrderScheme}
with constant step-sizes present an  unstable behavior (see, e.g., \cite{Hutzenthaler2011} for a theoretical study).
In the examples \textbf{Ex1} and \textbf{Ex2}, 
the  Euler scheme  \eqref{eq:Euler-Maruyama} (resp. the second weak  order scheme \eqref{eq:SecondOrderScheme})
grow to $+ \infty$ when its step-size is equal to $20$ divided by 
the mean value of the number of steps taken by Adaptive scheme  \ref{AdaptiveEM1}
(resp.  Adaptive scheme \ref{AdaptiveO2_Extrap}) with tolerances 
$Atol = Rtol =  10^{-2}, 10^{-3}, 10^{-4}$.
This strongly suggests that the new adaptive strategy improves the dynamical properties of the underlying numerical schemes when a similar number of total recursive steps are used.

Consider the limiting case of $\Delta_{min} = 0$ and $\mathfrak{fac}_{max} = + \infty$.
Then,
in the step 5 of Adaptive scheme \ref{AdaptiveEM1} we have   
$
\tau_{n+1} - \tau_{n} =  \min \left\{ \Delta_{max}, \Delta_{*,n}   \right\} 
$,
where $\Delta_{*,n}$ is given by Definition \ref{def:Delta_*}.
Hence,
$
x \, b \left( x \right)  + \frac{1}{2} \left(  \tau_{n+1} - \tau_{n} \right)  \left\vert  b \left( x \right)    \right\vert^2
$
is asymptotically equivalent to 
$\left( -1 + \sqrt{ RTol/6} \right) x^4 $ as $ \left\vert x \right\vert  \rightarrow + \infty$.
This gives \eqref{eq:CondHyp4} if $RTol < 6$,
and so Hypothesis  \ref{hyp:WeakOrden1} holds in this case.

\subsection{Stochastic Duffing-van der Pol equation}
\label{sub:StochasticDP}

\begin{figure}[tb]
\caption{
Computation of $\mathbb{E} \, Q_{t}$,
where $ t = 0, 1, \ldots,  80 $ and $ Q_{t}$ is given by
\eqref{eq:StochVdpol} with 
$ - \delta  = \beta = 10$, $\alpha = -1$, $\gamma = -0.1$
and 
$\sigma_1 =  \sigma_2 =  \sigma_3 = 0.5$. 
The black solid line represents the reference values,
and the dashed lines describe the estimations of $ \mathbb{E} \, Q_{T_k}  $ obtained by  
Adaptive scheme  \ref{AdaptiveEM1} with tolerance
 $Atol_i = Rtol_i$ equal to $10^{-2} $ (red), $10^{-3}$ (blue), $10^{-4}$ (green).
}

\centering
 \resizebox{12cm}{!}{\includegraphics{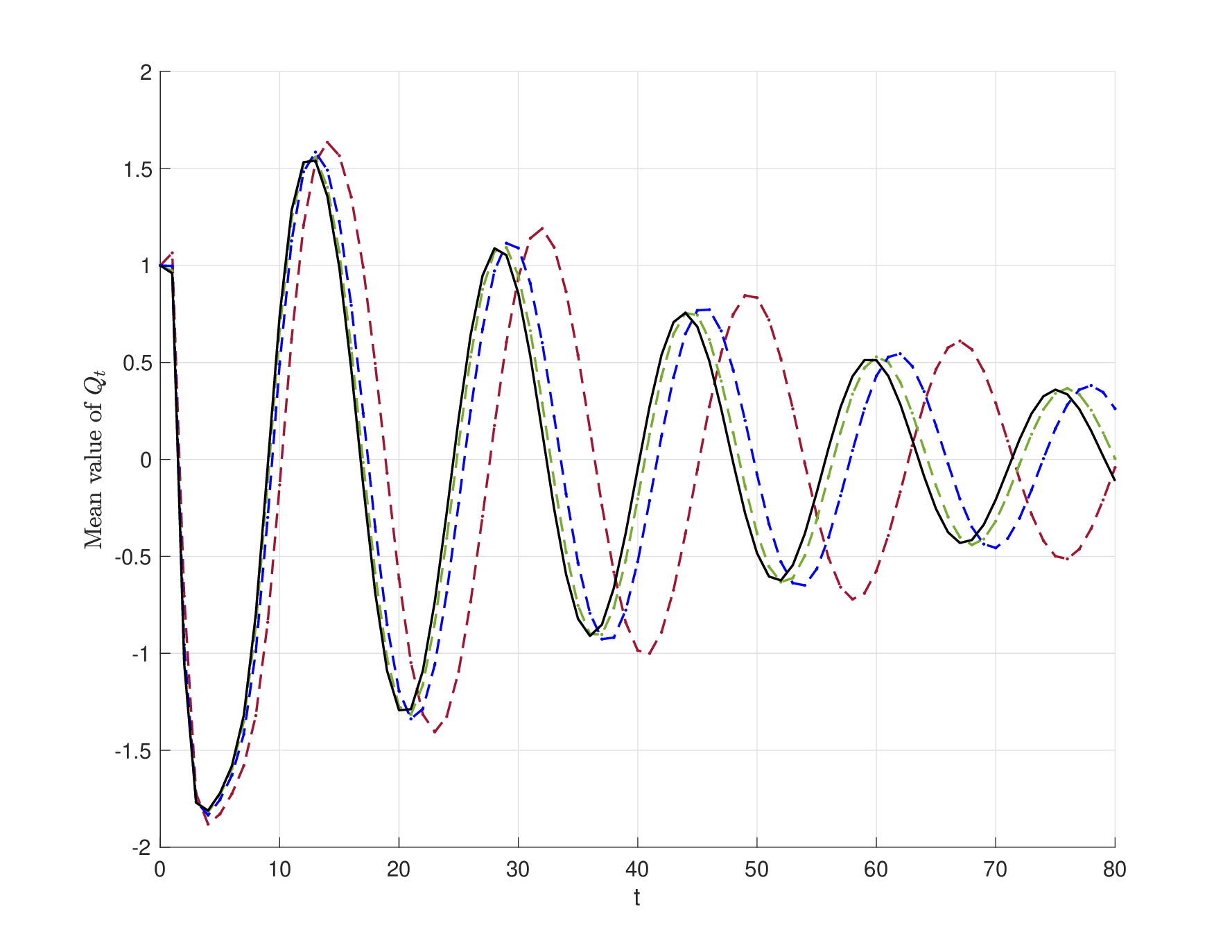}}
\label{fig1:DVDP_Q}
\end{figure}

\begin{figure}[tb]
\centering
 \resizebox{12cm}{!}{\includegraphics{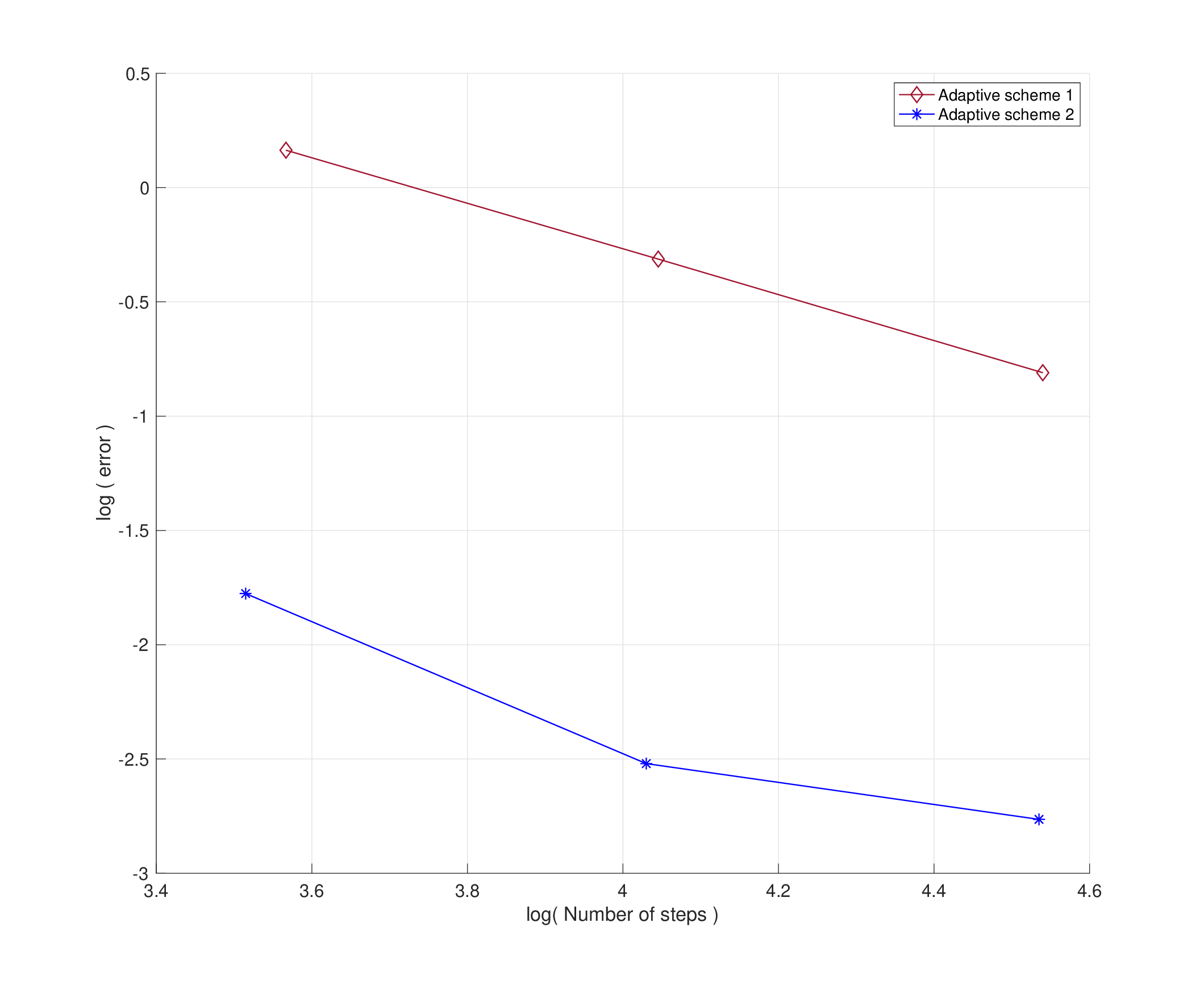}}
\caption{
The base 10 logarithm of the weak error $ \bar{\epsilon}_4$ as a function of  
the base 10 logarithm of  the expected value of the number of steps 
in the solution of the stochastic Duffing-van der Pol equation \eqref{eq:StochVdpol}
by Adaptive schemes  \ref{AdaptiveEM1} and  \ref{AdaptiveO2_Extrap}.
}
\label{fig:Errores_DVP}
\end{figure}

We deal with the following stochastic extension of the  Duffing-van der Pol equation
\begin{equation}
\label{eq:StochVdpol}
 \left\{
\begin{aligned}
 d Q_t 
 & =
 P _t \, dt
 \\
 d P _t
  & =
 \left( \alpha \, Q_t  + \left( \beta + \left( \sigma_2 \right)^2/2  \right) P _t  + \gamma \left( Q_t \right)^3  + \delta \left( Q_t \right)^2 P _t  \right) dt
 +
 \sigma_1 Q_t \, dW^{1}_{t}   +  \sigma_2 P_t \, dW^{2}_{t}  + \sigma_3 \, dW^{3}_{t} 
\end{aligned}
 \right.  ,
\end{equation}
where $t \geq 0$ and $ \alpha$, $ \beta $, $ \gamma $, $ \delta $,  $\sigma_1,  \sigma_2,  \sigma_3 \in \mathbb{R}$.
We take $Q_{0}  =  P_{0} = 1 $.
The non-linear Langevin type-equation \eqref{eq:StochVdpol}  has  already been used  for testing  SDE solvers 
(see, e.g., \cite{delaCruz2017,GilsingShardlow2007,Hutzenthaler2015,Kloeden1992,KupperLehnRossler2007,MardonesMora2019}).
We fix  $\beta = - \delta > 0$.
Thus, 
in case 
$
\gamma = \sigma_1 =  \sigma_2 =  \sigma_3 = 0
$
the SDE \eqref{eq:StochVdpol} reduces to the usual van der Pol oscillator,
a common model problem in the numerical solution of ODEs,
which becomes increasingly stiff as $\beta$ takes larger values.
We choose $ \beta = 10$, $\alpha = -1$, $\gamma = -0.1$
and 
$\sigma_1 =  \sigma_2 =  \sigma_3 = 0.5$.
It has been proven that \eqref{eq:StochVdpol} satisfies Hypothesis \ref{hyp:SDE}
by using the Lyapunov-type function 
$\left( Q, P \right) \mapsto 1 + Q^4 + 2 \, P^2$ 
(see, e.g., \cite{Hutzenthaler2015}).
For any  $x \in \mathbb{R}^2$ the linear span of 
$\sigma_3 \left( x \right) = 
 \sigma_3 
\begin{pmatrix}
 0
 \\
 1
\end{pmatrix}
$
and the Lie bracket 
$
\left[  b \left( x \right) , \sigma_3 \left( x \right) \right]
=
 \sigma_3
\begin{pmatrix}
 1
 \\
 \beta + \left( \sigma_2 \right)^2/2 +  \delta  \left( x_1 \right) ^2
\end{pmatrix}
$
is equal to $\mathbb{R}^2$,
and so \eqref{eq:StochVdpol} satisfies the H\"ormander condition. 
This implies that the Kolmogorov equation \eqref{eq:Kolmogorov} has a unique smooth classical solution
provided that $\varphi \in \mathcal{C}^{5}\left( \mathbb{R}^{d},\mathbb{R}\right)$ is bounded (see, e.g., \cite{Hairer2015,Hormander1967} for details).
We leave open the problem of checking the fulfillment of Hypothesis \ref{hyp:EDP} 
in case $\varphi $ is an unbounded smooth function.

We compute $ \mathbb{E} \, Q_{T_k}  $ for $T_k = k$ with $k = 0, 1, \ldots,  80 $.
In Figure \ref{fig1:DVDP_Q},
the solid line interpolates the reference values for $ \mathbb{E} \, Q_{T_k}  $
calculated by the Euler scheme with constant step-size
$ 10^{-5}$ and sample size $3.2 \cdot 10^{7}$.
The estimated value of $\mathrm{E}_S$ given by \eqref{eq:Es} with $\varphi(q,p) = q$,
but taking the maximum over $k = 1, \ldots,  80 $,
is $0.4735 \cdot 10^{-3}$.
Hence,
we would expect a precision of approximately $0.001$ in the computation of the reference values for $ \mathbb{E} \, Q_{T_k}  $.

We compute  $ \mathbb{E} \, Q_{T_k}  $ by  Adaptive schemes  \ref{AdaptiveEM1} and  \ref{AdaptiveO2_Extrap}
with tolerances $Atol_i = Rtol_i = 10^{-2}, 10^{-3}, 10^{-4}$.
Adaptive schemes \ref{AdaptiveEM1} and \ref{AdaptiveO2_Extrap} are sampled $10^7$ times;
we actually use Sampling Method \ref{sampling:1} with  $ S_{max} = 10^{7}$,
together with $\delta = 0.01$, $ AS_{tol} = RS_{tol} = 10^{-4}$,
$S_{min} = 10^6$,  $ S_{max} = 10^{7}$ and $\mathfrak{sfac}_{max} = 2$.
In all cases, the estimated value of $\mathrm{E}_S$ is around of $0.84  \cdot 10^{-3} $.
In Figure \ref{fig1:DVDP_Q},
the dashed lines  represent the estimations of $ \mathbb{E} \, Q_{T_k}  $ 
obtained by  Adaptive scheme  \ref{AdaptiveEM1}.
 Figure \ref{fig:Errores_DVP} presents estimations of the error 
\begin{equation}
\label{eq:errorDVP}
 \bar{\epsilon}_4 \left( Y \right)
=
\max_{k=1,2,\ldots,80}
\left\vert
\mathbb{E} \, Y^1_{n\left( k \right)}   - \mathbb{E} \, X^1_{T_k}  
\right\vert 
\end{equation}
as a function of  the mean values of the number of steps used in the computations
of $\mathbb{E} \, Q _{T_k}$,
 where $n\left( k \right)$ is given by \eqref{eq:def_nk},
and $Y^1$ stands for the first coordinate of  Adaptive schemes  \ref{AdaptiveEM1} and  \ref{AdaptiveO2_Extrap}.
In  Figures \ref{fig1:DVDP_Q} and \ref{fig:Errores_DVP},
better approximations are associated with lower values of $Atol_i = Rtol_i$.

 According to Figures \ref{fig1:DVDP_Q} and \ref{fig:Errores_DVP} 
we have that Adaptive scheme  \ref{AdaptiveEM1} 
gets closer and closer to 
the oscillations of $ \mathbb{E} \, Q_{T_k}  $ as the tolerance parameters $Atol_i=Rtol_i$ decrease.
In  Figure \ref{fig:Errores_DVP} highlights the very good accuracy of Adaptive scheme \ref{AdaptiveO2_Extrap}
with a low  number of integration steps.
For example,  
Adaptive scheme \ref{AdaptiveO2_Extrap} achieves the error of $0.0033$
with the same computational cost of  the second order method \eqref{eq:SecondOrderScheme}
with uniform step-size $0.0075 = 80/10715.7$  ($\log_{10} \left(10715.7 \right) =  3.5668$).
In contrast,
if we set, for instance, the step-size to the constant value $0.03$ --a bit larger than $80/3687.9$ ( $\log_{10} \left(3687.9 \right) =  4.03$)--
(resp. $0.04$), 
then the amplitude of the Monte-Carlo estimations of $ \mathbb{E} \, Q_{T_k}  $ given by 
Euler scheme (resp. the second order method \eqref{eq:SecondOrderScheme})
quickly take values close to  $+ \infty$, and provide Not-a-Number as output for the last $ \mathbb{E} \, Q_{T_k}  $. 
Thus, 
similar to Sections \ref{subsec:LinearScalar}-\ref{sub:Stochastic_Landau},
the new algorithms adjust appropriately the  step-size of the numerical scheme
to approximate the law of the solution of \eqref{eq:StochVdpol},
avoiding numerical  instabilities.


%

\subsection{Summary of the experimental results }
Numerical experiments show that the new adaptive schemes greatly overcome the accuracy of the Euler and second order Taylor schemes with fixed step-size in the integration of four test equations. 
The new adaptive strategy reduces appropriately the step-sizes of the schemes as the tolerances become smaller,
and  it improves the stability of the Euler and second weak order schemes
when a similar number of total recursive steps are used. 
In the examples where we consider the stochastic strategy given by  \cite{Szepessy2001},
the new  adaptive strategy achieves a much better performance than 
the above adaptive method based on global errors.

\section{Proofs}

\subsection{Proof of Theorem \ref{theo:LocalExpansion}} 
\label{subsec:Proofth:theo:LocalExpansion}

\begin{proof}
Take 
$
N = \sup \left\{ \mathcal{N} \left( \omega \right) :  \omega \in \Omega  \right\}
$.
Thus,
$\tau_{N} = T$ and $Y_{N} = Y_{\mathcal{N}}$.
Let  $u$ be the function described in Hypothesis \ref{hyp:EDP}.
Then,
$ \mathbb{E} \varphi \left(  Y_{N}\right) = \mathbb{E} u  \left(  T, Y_{N}\right) $,
and so
\[
\mathbb{E} \varphi \left(  Y_{N} \right)
=
\mathbb{E} u  \left(  0 , Y_{0}\right)
+
\sum_{n=0}^{N-1} \left( 
\mathbb{E}  u  \left(  \tau_{n+1}, Y_{n+1}\right)  -  \mathbb{E}  u  \left(  \tau_{n}, Y_{n}\right) 
\right) 
.
\]
Using
$\partial_t u  = -\mathcal{L}\left( u \right) $
and 
$
u \in  \mathcal{C}_{P}^{ 5 } \left( \left[ 0, T \right] \times \mathbb{R}^d,\mathbb{R}\right)
$
we get
$\partial_t u  \in \mathcal{C}_{P}^{3} \left( \left[ 0, T \right] \times \mathbb{R}^d,\mathbb{R}\right)$ 
and  
$\partial_{t t} u  \in   \mathcal{C}_{P}^{ 1 } \left( \left[ 0, T \right] \times \mathbb{R}^d,\mathbb{R}\right) $.
According to Taylor's theorem we have 
\begin{align*}
& u  \left(  \tau_{n+1}, Y_{n+1}\right)  
   =  
u  \left(  \tau_{n}, Y_{n}\right) 
+
 \partial_t u \left(  \tau_n ,  Y_{n}  \right)  \left( \tau_{n+1} - \tau_{n} \right)   
+
\sum_{ \left\vert \alpha \right\vert = 1 }  
 \partial_x^{ \alpha} u \left(  \tau_n ,  Y_{n}  \right) \left(   Y_{n+1} -   Y_{n} \right)^{\alpha}  
  \\
& \qquad
+
 \left( \tau_{n+1} - \tau_{n} \right)  ^2
 \int_{0}^{1} \int_{0}^{r} \partial_{t  t } u \circ \lambda \left(  s \right) ds \, dr 
 +
 2
 \left( \tau_{n+1} - \tau_{n} \right)  
 \sum_{ \left\vert \alpha \right\vert = 1 }   \left(   Y_{n+1} -   Y_{n} \right)^{\alpha} 
 \int_{0}^{1} \int_{0}^{r}   \partial_x^{ \alpha} \partial_{t } u \circ \lambda \left(  s \right)   ds \, dr 
 \\
& \qquad 
 + 
 \sum_{ \left\vert \alpha \right\vert = 2 }   \frac{ 2 }{ \alpha !} 
 \int_{0}^{1} \int_{0}^{r}   \partial_x^{ \alpha}  u \circ \lambda \left(  s \right)   ds \, dr  \left(   Y_{n+1} -   Y_{n} \right)^{\alpha}  ,
\end{align*}
with 
$
\lambda \left(  s \right)
=
\left(  \tau_n ,  Y_{n}  \right)  + s \left(  \tau_{n+1} - \tau_n ,  Y_{n+1}  - Y_{n}  \right) 
$.
Applying the fundamental theorem of calculus to
$\partial_x^{ \alpha} \partial_{t } u \circ \lambda \left(  s \right) $ and $\partial_x^{ \alpha}  u \circ \lambda \left(  s \right)$
we get
\begin{equation*}
 u  \left(  \tau_{n+1}, Y_{n+1}\right)  
=  
u  \left(  \tau_{n}, Y_{n}\right) 
+
 \partial_t u \left(  \tau_n ,  Y_{n}  \right)  \left( \tau_{n+1} - \tau_{n} \right)  
+
\sum_{ \left\vert \alpha \right\vert = 1, 2 }  \frac{1}{ \alpha !} \partial_x^{ \alpha} u \left(  \tau_n ,  Y_{n}  \right) \left(   Y_{n+1} -   Y_{n} \right)^{\alpha} 
+
R^Y_{n+1},
\end{equation*}
where 
\begin{align*}
 R^Y_{n+1} =
 & 
 \left( \tau_{n+1} - \tau_{n} \right)  ^2
 \int_{0}^{1} \int_{0}^{r} \partial_{t  t } u \circ \lambda \left(  s \right) ds \, dr 
 +
 \left( \tau_{n+1} - \tau_{n} \right)   
 \sum_{ \left\vert \alpha \right\vert = 1 }  \partial_x^{ \alpha} \partial_t u \left(  \tau_n ,  Y_{n}  \right) \left(   Y_{n+1} -   Y_{n} \right)^{\alpha} 
  \\
 & \quad
 +
  2 \left( \tau_{n+1} - \tau_{n} \right)  ^2
 \sum_{ \left\vert \alpha \right\vert = 1 }   \left(   Y_{n+1} -   Y_{n} \right)^{\alpha} 
 \int_{0}^{1} \int_{0}^{r}  \int_{0}^{u} \partial_x^{ \alpha} \partial_{t t} u \circ \lambda \left(  s \right)   ds \, du \, dr 
 \\
  & \quad + 
 \left( \tau_{n+1} - \tau_{n} \right)   
  \sum_{ \left\vert \alpha \right\vert = 2 } \frac{6}{ \alpha !}   \left(   Y_{n+1} -   Y_{n} \right)^{\alpha} 
  \int_{0}^{1} \int_{0}^{r}  \int_{0}^{u} \partial_x^{ \alpha} \partial_t u \circ \lambda \left(  s \right)   ds \, du \, dr  
   \\
 & \quad
 +
  \sum_{ \left\vert \alpha \right\vert = 3 }  \frac{6}{ \alpha !} \left(   Y_{n+1} -   Y_{n} \right)^{\alpha} 
  \int_{0}^{1} \int_{0}^{r}  \int_{0}^{u} \partial_x^{ \alpha}  u \circ \lambda \left(  s \right)   ds \, du \, dr  .
\end{align*}
Therefore,
\begin{equation}
 \label{eq:LD_SchemeN}
 \mathbb{E} \varphi \left(  Y_{N}\right) 
 =
 \mathbb{E} u \left(  0 ,  Y_{0}  \right)
 + 
\sum _{n=0}^{N-1} \mathbb{E}  \left(  
\mathcal{T}_n \left(  Y_{n+1} \right) + \mathbb{E} \left(  R^Y_{n+1}  \diagup  \mathfrak{F}_{\tau_n} \right)
\right) .
 \end{equation}

Combining  $\partial_t u  = -\mathcal{L}\left( u \right) $
with the application of the fundamental theorem of calculus to $\partial_x^{ \alpha}  u \circ \lambda $
we deduce 
\begin{align*}
\mathbb{E} \left( R^Y_{n+1}  \diagup \mathfrak{F}_{\tau_{n}} \right) =
&
-
 \left( \tau_{n+1} - \tau_{n} \right)  ^2
 \int_{0}^{1}  \int_{0}^{r} 
 \mathbb{E} \left( \partial_{t  } \mathcal{L}\left( u \right)  \circ \lambda \left(  s \right) 
 \diagup \mathfrak{F}_{\tau_{n}} \right) 
 ds \, dr 
 \\
 & 
  -
 \left( \tau_{n+1} - \tau_{n} \right)   
 \sum_{ \left\vert \alpha \right\vert = 1 }  \partial_x^{ \alpha} \mathcal{L}\left( u \right) \left(  \tau_n ,  Y_{n}  \right) 
  \mathbb{E} \left( 
 \left(   Y_{n+1} -   Y_{n} \right)^{\alpha} 
 \diagup \mathfrak{F}_{\tau_{n}} \right) 
 \\
 & 
 -
 \left( \tau_{n+1} - \tau_{n} \right)  ^2
 \sum_{ \left\vert \alpha \right\vert = 1 } \frac{2}{ \alpha !}  
 \int_{0}^{1} \! \int_{0}^{r} \! \int_{0}^{u} 
 \mathbb{E} \left( 
  \left(   Y_{n+1} -   Y_{n} \right)^{\alpha} 
 \partial_x^{ \alpha} \partial_{t } \mathcal{L}\left( u \right)  \circ \lambda \left(  s \right)   
 \diagup \mathfrak{F}_{\tau_{n}} \right) 
 ds \, du \, dr 
 \\
  & 
   -
 \left( \tau_{n+1} - \tau_{n} \right)   
  \sum_{ \left\vert \alpha \right\vert = 2 } \frac{6}{ \alpha !}  
  \int_{0}^{1} \int_{0}^{r}  \int_{0}^{u} 
   \mathbb{E} \left( 
   \left(   Y_{n+1} -   Y_{n} \right)^{\alpha} 
  \partial_x^{ \alpha} \mathcal{L}\left( u \right)  \circ \lambda \left(  s \right)  
  \diagup \mathfrak{F}_{\tau_{n}} \right) 
   ds \, du \, dr  
  \\
 & 
+
 \sum_{ \left\vert \alpha \right\vert = 3 }  \frac{1}{ \alpha !} \partial_x^{ \alpha} u \left(  \tau_n ,  Y_{n}  \right) 
 \mathbb{E} \left(  \left(   Y_{n+1} -   Y_{n} \right)^{\alpha} \diagup \mathfrak{F}_{\tau_{n}} \right)
  \\
 & 
 -
 \left( \tau_{n+1} - \tau_{n} \right)   
  \sum_{ \left\vert \alpha \right\vert = 3 } \frac{6}{ \alpha !}   
  \int_{0}^{1}  \int_{0}^{r}  \int_{0}^{u}  \int_{0}^{v} 
   \mathbb{E} \left( 
  \left(   Y_{n+1} -   Y_{n} \right)^{\alpha} 
  \partial_x^{ \alpha} \mathcal{L}\left( u \right) \circ \lambda \left(  s \right)   
  \diagup \mathfrak{F}_{\tau_{n}} \right) 
  ds \, dv \, du \, dr
  \\
& 
+
  \sum_{ \left\vert \alpha \right\vert = 4 } \frac{ 24}{ \alpha !}  
  \int_{0}^{1} \int_{0}^{r}  \int_{0}^{u} \int_{0}^{v} 
   \mathbb{E} \left( 
   \left(   Y_{n+1} -   Y_{n} \right)^{\alpha} 
  \partial_x^{ \alpha}  u \circ \lambda \left(  s \right)   
  \diagup \mathfrak{F}_{\tau_{n}} \right) 
  ds \, dv  \, du \, dr .
\end{align*}

Since
\begin{align*}
\left\vert
  \mathbb{E} \left(   Y^j_{n+1} -   Y^j_{n}  \diagup \mathfrak{F}_{\tau_{n}} \right) 
\right\vert
&
\leq
\left\vert
  \mathbb{E} \left(    Y^j_{n+1} -   Y^j_{n}  \diagup \mathfrak{F}_{\tau_{n}} \right) 
  -
  \mathbb{E} \left( Z^j _{n+1}\left(  Y_{n} \right) -   Y^j_{n}  \diagup \mathfrak{F}_{\tau_{n}} \right) 
\right\vert
+
\left\vert
 \mathbb{E} \left( \  Z^j _{n+1}\left(  Y_{n} \right) -   Y^j_{n}  \diagup \mathfrak{F}_{\tau_{n}} \right) 
\right\vert ,
\end{align*}
using Hypothesis \ref{hyp:WeakOrden1} (c) we obtain
\[
\left\vert
  \mathbb{E} \left(   Y^j_{n+1} -   Y^j_{n}  \diagup \mathfrak{F}_{\tau_{n}} \right) 
\right\vert
\leq
K \left( 1+\left\Vert  Y_n \right\Vert ^{q}\right)  \left(\tau_{n+1} - \tau_n \right)^2
+
\left\vert  b^j \left(  \tau_{n} ,  Y_n   \right) \right\vert \left(\tau_{n+1} - \tau_n \right)  .
\]
This gives
\[
\left\vert 
 \sum_{ \left\vert \alpha \right\vert = 1 }  \partial_x^{ \alpha} \mathcal{L}\left( u \right) \left(  \tau_n ,  Y_{n}  \right) 
  \mathbb{E} \left( 
 \left(   Y_{n+1} -   Y_{n} \right)^{\alpha} 
 \diagup \mathfrak{F}_{\tau_{n}} \right) 
\right\vert
\leq
K \left( 1+\left\Vert  Y_n \right\Vert ^{q}\right)  \left(\tau_{n+1} - \tau_n \right) .
\]
Similarly,
combining Hypothesis \ref{hyp:WeakOrden1} (c) with
\begin{align*}
&
 \mathbb{E} \left( 
 \left( Z^{j_1} _{n+1}\left(  Y_{n} \right) -   Y^{j_1}_{n} \right) 
 \left( Z^{j_2} _{n+1}\left(  Y_{n} \right) -   Y^{j_2}_{n} \right)
 \left( Z^{j_3} _{n+1}\left(  Y_{n} \right) -   Y^{j_3}_{n} \right) 
 \diagup \mathfrak{F}_{\tau_{n}} \right)
 \\
 & =
 b^{j_1}\left(  \tau_{n} , Y_{n}  \right) b^{j_2}\left(  \tau_{n} , Y_{n}  \right) b^{j_3}\left(  \tau_{n} , Y_{n}  \right)
  \left( \tau_{n+1} - \tau_{n} \right)^3
 +
 \left( \tau_{n+1} - \tau_{n} \right)^2 
 b^{j_1}\left(  \tau_{n} , Y_{n}  \right) \sum_{k=1}^m \sigma^{j_2}_k \left(  \tau_{n}, Y_{n} \right) \sigma^{j_3}_k \left(  \tau_{n}, Y_{n} \right)
\\
 & \quad 
 +
 \left( \tau_{n+1} - \tau_{n} \right)^2 
 \left(
 b^{j_2}\left(  \tau_{n} , Y_{n}  \right)  \sum_{k=1}^m \sigma^{j_1}_k \left(  \tau_{n}, Y_{n} \right) \sigma^{j_3}_k \left(  \tau_{n}, Y_{n} \right)
 +
 b^{j_3}\left(  \tau_{n} , Y_{n}  \right)  \sum_{k=1}^m \sigma^{j_1}_k \left(  \tau_{n}, Y_{n} \right) \sigma^{j_2}_k \left(  \tau_{n}, Y_{n} \right)
 \right)
 \end{align*}
yields
\[
\left\vert 
\sum_{ \left\vert \alpha \right\vert = 3 }  \frac{1}{ \alpha !} \partial_x^{ \alpha} u \left(  \tau_n ,  Y_{n}  \right) 
 \mathbb{E} \left(  \left(   Y_{n+1} -   Y_{n} \right)^{\alpha} \diagup \mathfrak{F}_{\tau_{n}} \right)
\right\vert
\leq
K \left( 1+\left\Vert  Y_n \right\Vert ^{q}\right)  \left(\tau_{n+1} - \tau_n \right)^2 .
\]
Using Hypothesis \ref{hyp:WeakOrden1} (b), together with the Cauchy-Bunyakovsky-Schwarz inequality,
we bound from above the absolute values of  the remaining terms of 
$\mathbb{E} \left( R^Y_{n+1}  \diagup \mathfrak{F}_{\tau_{n}} \right)$
to obtain 
\[
 \left\vert  \mathbb{E} \left( R^Y_{n+1}  \diagup \mathfrak{F}_{\tau_{n}} \right) \right\vert 
\leq 
 K  \left( 1 +  \left\Vert  Y_n \right\Vert^{q }  \right) \left( \tau_{n+1} - \tau_{n} \right)^2 
.
\]

Applying It\^o's formula we get
$
u  \left(  0 , X_{0} \right) 
= 
\mathbb{E}  \left(  u  \left( T ,  X_{T}  \right)  \diagup \mathfrak{F}_{\tau_{0}} \right)
$
and
\[
u  \left(  \tau_n , Y_{n} \right) 
= 
\mathbb{E}  \left(  u  \left(  \tau_{n+1} ,  X_{\tau_{n+1}} \left(  \tau_n , Y_{n}  \right) \right)   \diagup \mathfrak{F}_{\tau_{n}} \right)
\]
(see, e.g., proof of Theorem 5.7.6 of \cite{KaratzasShreve1998} 
or proof of Theorem 7.14 of  \cite{GrahamTalay2013}).
Therefore,
\begin{equation*}
\mathbb{E} \varphi \left(  X_T \right)
  =
\mathbb{E} u  \left(  0 , X_{0}\right)
 =
\mathbb{E} u  \left(  0 , X_{0}\right)
+
\sum_{n=0}^{N-1} \left( 
\mathbb{E}  u  \left(  \tau_{n+1} ,  X_{\tau_{n+1}} \left(  \tau_n , Y_{n}  \right) \right) 
 -  \mathbb{E}  u  \left(  \tau_{n}, Y_{n}\right) 
\right) .
\end{equation*}
As in the proof of \eqref{eq:LD_SchemeN},
using the fundamental theorem of calculus and Taylor's theorem
we obtain 
\begin{equation}
 \label{eq:LDN}
  \mathbb{E} \varphi \left( X_{T} \right) 
 =
 \mathbb{E} u \left(  0 , X_{0}  \right)
 + 
\sum_{n=0}^{N-1}  \mathbb{E}  \left(
\mathcal{T}_n \left(  X_{\tau_{n+1}} \left( \tau_{n} ,  Y_{n}  \right)  \right) +  \mathbb{E} \left(  R^X_{n+1}  \diagup  \mathfrak{F}_{\tau_n} \right) 
\right)
\end{equation}
with $ R^X_{n+1} $ equal to
\begin{align*}
&
-
 \left( \tau_{n+1} - \tau_{n} \right)  ^2
 \int_{0}^{1}  \int_{0}^{r} 
 \partial_{t  } \mathcal{L}\left( u \right)  \circ \lambda^X \left(  s \right) 
 ds \, dr 
  -
 \left( \tau_{n+1} - \tau_{n} \right)   
 \sum_{ \left\vert \alpha \right\vert = 1 }  \partial_x^{ \alpha} \mathcal{L}\left( u \right) \left(  \tau_n ,  Y_{n}  \right) 
 \left(  X_{\tau_{n+1}} \left(  \tau_n , Y_{n}  \right)  -   Y_{n} \right)^{\alpha}  
 \\
 & 
 -
 \left( \tau_{n+1} - \tau_{n} \right)  ^2
 \sum_{ \left\vert \alpha \right\vert = 1 } \frac{2}{ \alpha !}  
 \int_{0}^{1} \int_{0}^{r}  \int_{0}^{u} 
  \left(   X_{\tau_{n+1}} \left(  \tau_n , Y_{n}  \right)  -   Y_{n} \right)^{\alpha} 
 \partial_x^{ \alpha} \partial_{t } \mathcal{L}\left( u \right)  \circ \lambda^X \left(  s \right)   
 ds \, du \, dr 
 \\
  & 
   -
 \left( \tau_{n+1} - \tau_{n} \right)   
  \sum_{ \left\vert \alpha \right\vert = 2 } \frac{6}{ \alpha !}  
  \int_{0}^{1} \int_{0}^{r}  \int_{0}^{u} 
   \left(   X_{\tau_{n+1}} \left(  \tau_n , Y_{n}  \right)  -   Y_{n} \right)^{\alpha} 
  \partial_x^{ \alpha} \mathcal{L}\left( u \right)  \circ \lambda^X \left(  s \right)  
   ds \, du \, dr  
  \\
 & 
+
 \sum_{ \left\vert \alpha \right\vert = 3 }  \frac{1}{ \alpha !} \partial_x^{ \alpha} u \left(  \tau_n ,  Y_{n}  \right) 
 \left(   X_{\tau_{n+1}} \left(  \tau_n , Y_{n}  \right)  -   Y_{n} \right)^{\alpha}
  \\
 & 
 -
 \left( \tau_{n+1} - \tau_{n} \right)   
  \sum_{ \left\vert \alpha \right\vert = 3 } \frac{6}{ \alpha !}   
  \int_{0}^{1} \int_{0}^{r}   \int_{0}^{u}  \int_{0}^{v} \hspace{-2pt}
  \left(   X_{\tau_{n+1}} \left(  \tau_n , Y_{n}  \right)  -   Y_{n} \right)^{\alpha} 
  \partial_x^{ \alpha} \mathcal{L}\left( u \right) \hspace{-1pt} \circ \hspace{-1pt} \lambda^X \left(  s \right)   
  ds \, dv  \,  du   \, dr
  \\
& 
+
  \sum_{ \left\vert \alpha \right\vert = 4 } \frac{ 24}{ \alpha !}  
  \int_{0}^{1} \int_{0}^{r}  \int_{0}^{u} \int_{0}^{v} 
   \left(   X_{\tau_{n+1}} \left(  \tau_n , Y_{n}  \right)  -   Y_{n} \right)^{\alpha} 
  \partial_x^{ \alpha}  u \circ \lambda^X \left(  s \right)   
  ds \, dv  \, du \, dr ,
\end{align*}
where
$
\lambda^X \left(  s \right)
=
\left(  \tau_n ,  Y_{n}  \right)  + s \left(  \tau_{n+1} - \tau_n ,   X_{\tau_{n+1}} \left(  \tau_n , Y_{n}  \right)  - Y_{n}  \right) 
$.
Now,
applying basic properties of the solution of \eqref{eq:2.1}, together with \eqref{eq:0.1},  we get
$
 \left\vert  \mathbb{E} \left( R^X_{n+1}  \diagup \mathfrak{F}_{\tau_{n}} \right) \right\vert 
\leq 
 K  \left( 1 +  \left\Vert  Y_n \right\Vert^{q }  \right) \left( \tau_{n+1} - \tau_{n} \right)^2 
$.

Since $Y_{n+1}  \left( \omega \right) - Y_n  \left( \omega \right) = 0$ whenever $\tau_{n}  \left( \omega \right) = T$,
\[
\mathbb{E} \left(  \left(  Y_{n+1}  -  Y_{n} \right)^{\alpha}  \diagup \mathfrak{F}_{\tau_n} \right)
=
\mathbb{E} \left(  \left(  Y_{n+1}  -  Y_{n} \right)^{\alpha}  \diagup \mathfrak{F}_{\tau_n} \right)
\mathbf{1}_{\tau_n < T} .
\]
Moreover,
 $\tau_{n+1} \left( \omega \right) - \tau_{n}  \left( \omega \right) = 0$  
 in case $\tau_{n}  \left( \omega \right) = T$.
 Then,
 using  \eqref{eq:LD_SchemeN} and \eqref{eq:LDN} we obtain  \eqref{eq:LD_Scheme} and \eqref{eq:LD}.
\end{proof}

\subsection{Proof of Theorem \ref{theo:WeakConvergence}} 
\label{subsec:Proofth:theo:WeakConvergence}
\begin{proof}
Applying Theorem \ref{theo:LocalExpansion} gives
\begin{align*}
 \left\vert 
  \mathbb{E} \varphi \left( X_{T} \right) 
  -
  \mathbb{E} \varphi \left(  Y_{\mathcal{N}}\right)
 \right\vert
 &
 \leq
 \left\vert 
 \mathbb{E} u \left(  0 , X_{0}  \right)
 -
  \mathbb{E} u \left(  0 ,  Y_{0}  \right)
 \right\vert
 \\
 & \quad
 +
  K\left( T \right) \, \mathbb{E} \sum _{n=0}^{  \mathcal{N} -1}
  \left(
 \left\vert 
 \mathcal{T}_n \left(  X_{\tau_{n+1}} \left( \tau_{n} ,  Y_{n}  \right)  \right)
 -\mathcal{T}_n \left(  Y_{n+1} \right)
\right\vert
 +
  \left( 1 +  \left\Vert  Y_n \right\Vert^{q }  \right) \left( \tau_{n+1} - \tau_{n} \right)^2
  \right).
 \end{align*}
 Using the condition (c) of Hypothesis \ref{hyp:WeakOrden1}, together with 
 $u \in  \mathcal{C}_{P}^{2} \left( \left[ 0, T \right] \times \mathbb{R}^d,\mathbb{R}\right)$,
 we obtain
\[
 \left\vert 
  \mathbb{E} \varphi \left( X_{T} \right) 
  -
  \mathbb{E} \varphi \left(  Y_{\mathcal{N}}\right)
 \right\vert
  \leq
 \left\vert 
 \mathbb{E} u \left(  0 , X_{0}  \right)
 -
  \mathbb{E} u \left(  0 ,  Y_{0}  \right)
 \right\vert
 +
 K\left( T \right) \, \mathbb{E} \sum _{n=0}^{  \mathcal{N} -1}
  \left( 1 +  \left\Vert  Y_n \right\Vert^{q }  \right) \left( \tau_{n+1} - \tau_{n} \right)^2 ,
 \]
 and so
 \begin{align*}
 \left\vert 
  \mathbb{E} \varphi \left( X_{T} \right) 
  -
  \mathbb{E} \varphi \left(  Y_{\mathcal{N}}\right)
 \right\vert
 & \leq
 \left\vert 
 \mathbb{E} u \left(  0 , X_{0}  \right)
 -
  \mathbb{E} u \left(  0 ,  Y_{0}  \right)
 \right\vert
 \\
 & \quad
 +
 K\left( T \right) 
\sup_{k \geq 0, \,  \omega \in \Omega  } \left\{ \tau_{k+1}   \left( \omega \right) -   \tau_{k}   \left( \omega \right) \right\}
\, \mathbb{E} \int_{0}^T  \left( 1 +  \left\Vert  Y_{n\left( s \right)} \right\Vert^{q }  \right) ds ,
 \end{align*}
 where
 $
 n\left( s \right) 
 =
\max \left\{ n \in \mathbb{N}: \tau_{n} \leq s \right\} 
$.
Since  $u \in  \mathcal{C}_{P}^{4} \left( \left[ 0, T \right] \times \mathbb{R}^d,\mathbb{R}\right)$,
we get the theorem from condition (a) and (d) of Hypothesis \ref{hyp:WeakOrden1}.
\end{proof}

\subsection{Proof of Lemma \ref{le:Momentos}} 
\label{subsec:Proofth:le:Momentos}

\begin{proof}
From the exact values of the expectation of the random variables that approximate 
the products of iterated It\^o integrals (see, e.g., \cite{Kloeden1992,Milstein2004}) 
it follows that
\[
\mathbb{E} \left(  \hat{Y}_{n} \left( \tau_n + \Delta \right)  \diagup \mathfrak{F}_{\tau_n} \right) 
=
Y_{n} 
+  b \left( \tau_n ,  Y_{n} \right) \Delta
+  \frac{1}{2} \mathcal{L}_0   b \, \left( \tau_n ,  Y_{n} \right) \Delta^2
\]
and
\begin{align*}
& 
\mathbb{E} \left(  \left(  \hat{Y}_{n}^i \left( \tau_n + \Delta \right)  -   Y^i_{n} \right)  
\left( \hat{Y}^j_{n} \left( \tau_n + \Delta \right)  -   Y^{j}_{n} \right) \diagup \mathfrak{F}_{\tau_n} \right)
\\
& =
b^i  b^j \Delta^2 
+
\frac{1}{2} b^i   \mathcal{L}_0   b^j   \Delta^3
+
\frac{1}{2} b^j   \mathcal{L}_0   b^i   \Delta^3
+
\frac{1}{4} \mathcal{L}_0   b^i   \cdot \mathcal{L}_0   b^j \Delta^4
+
\frac{1}{2}  \Delta^2  \sum_{k, \ell =1}^{m} 
\mathcal{L}_k  \sigma_{\ell}^i  \cdot \mathcal{L}_k  \sigma_{\ell}^j 
\\
& \quad
+
\Delta \sum_{k =1}^{m}  \sigma_{k}^i \sigma_{k}^j
+
\frac{1}{2}  \Delta^2  \sum_{k =1}^{m}  \sigma_{k}^i  \mathcal{L}_k   b^j
+
\frac{1}{2}  \Delta^2  \sum_{k =1}^{m}  \sigma_{k}^i  \mathcal{L}_0   \sigma_{k}^j
\\
& \quad
+
\frac{1}{2}  \Delta^2  \sum_{k =1}^{m}  \mathcal{L}_k   b^i \cdot \sigma_{k}^j 
+
\frac{1}{4}  \Delta^3  \sum_{k =1}^{m}  \mathcal{L}_k   b^i \cdot \mathcal{L}_k   b^j
+
\frac{1}{4}  \Delta^3  \sum_{k =1}^{m}  \mathcal{L}_k   b^i \cdot \mathcal{L}_0   \sigma_{k}^j
\\
& \quad
+
\frac{1}{2}  \Delta^2  \sum_{k =1}^{m}   \mathcal{L}_0   \sigma_{k}^i  \cdot \sigma_{k}^j  
+
\frac{1}{4}  \Delta^3  \sum_{k =1}^{m}   \mathcal{L}_0   \sigma_{k}^i \cdot \mathcal{L}_k   b^j
+
\frac{1}{4}  \Delta^3  \sum_{k =1}^{m}   \mathcal{L}_0   \sigma_{k}^i \cdot  \mathcal{L}_0   \sigma_{k}^j ,
\end{align*}
where the right-hand side of the last equality is evaluated at $\left( \tau_n ,  Y_{n} \right)$.
Now, using  \eqref{eq:4.10} we obtain the assertion of the lemma.
\end{proof}

\subsection{Proof of Lemma \ref{le:incrementos}} 
\label{subsec:Proofth:le:incrementos}

\begin{proof}
Using the optional sampling theorem, 
together with the fact that $\tau_{n} + \Delta$ is a stopping time,
we obtain 
\begin{equation}
\label{eq:4.10}
\begin{aligned}
 & \mathbb{E} \left(  \left(  Y_{n}^i \left( \tau_n + \Delta \right)  -   Y^i_{n} \right)  \left( Y^j_{n} \left( \tau_n + \Delta \right)  -   Y^{j}_{n} \right) \diagup \mathfrak{F}_{\tau_n} \right)
 \\
 & \quad =
 b^i \left( \tau_n ,  Y_{n} \right)  b^j \left( \tau_n ,  Y_{n} \right)  \Delta^2
 +
 \sum_{k=1}^{m} \sigma^i_k \left(  \tau_n ,  Y_{n}  \right) \sigma^j_k \left( \tau_n ,  Y_{n} \right) \Delta 
\end{aligned}
\end{equation}
for all $i , j = 1,\ldots, d$.
Applying the the triangle inequality gives
\begin{align*}
&
\left\Vert \left(  \frac{ 
  b^i \left( \tau_n ,  Y_{n} \right)  b^j \left( \tau_n ,  Y_{n} \right)  \Delta^2
 +
 \sum_{k=1}^{m} \sigma^i_k \left(  \tau_n ,  Y_{n}  \right) \sigma^j_k \left( \tau_n ,  Y_{n} \right) \Delta
}{ \widetilde{d}_{i,n} \left( \Delta \right) \widetilde{d}_{j,n} \left( \Delta \right) }  \right)_{i , j}  \right\Vert_{\mathbb{R}^{d \times d}} 
/ \Delta
\\
& \leq
\Delta
\left\Vert
\left(
\frac{ b^i \left( \tau_n ,  Y_{n} \right)  b^j \left( \tau_n ,  Y_{n} \right) }
{
\widetilde{d}_{i,n} \left( \Delta \right) \widetilde{d}_{j,n} \left( \Delta \right)
} \right)_{i,j}
\right\Vert_{\mathbb{R}^{d \times d}}
+
 \sum_{k=1}^{m}
 \left\Vert
\left(
\frac{ \sigma^i_k \left(  \tau_n ,  Y_{n}  \right) \sigma^j_k \left( \tau_n ,  Y_{n} \right)  }
{
 \widetilde{d}_{i,n} \left( \Delta \right) \widetilde{d}_{j,n} \left( \Delta \right)
} \right)_{i,j}
\right\Vert_{\mathbb{R}^{d \times d}} 
\\
& =
\Delta
\left\Vert
\left(
\frac{ b^i \left( \tau_n ,  Y_{n} \right)  }
{
\widetilde{ d }_{i,n} \left( \Delta \right) 
} \right)_{i}
\left(
\frac{ b^i \left( \tau_n ,  Y_{n} \right)  }
{
\widetilde{ d }_{i,n} \left( \Delta \right) 
} \right)_{i}^{\top}
\right\Vert_{\mathbb{R}^{d \times d}}
+
 \sum_{k=1}^{m}
 \left\Vert
\left(
\frac{ \sigma^i_k \left(  \tau_n ,  Y_{n}  \right)  }
{
 \widetilde{ d }_{i,n} \left( \Delta \right) 
} \right)_{i}
\left(
\frac{ \sigma^i_k \left(  \tau_n ,  Y_{n}  \right)  }
{
 \widetilde{ d }_{i,n} \left( \Delta \right) 
} \right)_{i}^{\top}
\right\Vert_{\mathbb{R}^{d \times d}}  .
\end{align*}
Therefore, using  \eqref{eq:4.4} yields
\begin{align*}
&
\left\Vert \left(  \frac{ 
  b^i \left( \tau_n ,  Y_{n} \right)  b^j \left( \tau_n ,  Y_{n} \right)  \Delta^2
 +
 \sum_{k=1}^{m} \sigma^i_k \left(  \tau_n ,  Y_{n}  \right) \sigma^j_k \left( \tau_n ,  Y_{n} \right) \Delta
}{ \widetilde{d}_{i,n} \left( \Delta \right) \widetilde{d}_{j,n} \left( \Delta \right) }  \right)_{i , j}  \right\Vert_{\mathbb{R}^{d \times d}} 
\\
& \leq
\Delta^2
\left\Vert
\left(
\frac{ b^i \left( \tau_n ,  Y_{n} \right)  }
{
 \widetilde{ d }_{i,n} \left( \Delta \right) 
} \right)_{i}
\right\Vert_{ \mathbb{R}^d }^2
 +
\Delta  \sum_{k=1}^{m}
 \left\Vert
\left(
\frac{ \sigma^i_k \left(  \tau_n ,  Y_{n}  \right)  }
{  \widetilde{ d }_{i,n} \left( \Delta \right)  } \right)_{i}
\right\Vert_{ \mathbb{R}^d } ^2
\\
& \leq
\Delta^2
\left\Vert
\left(
\frac{ b^i \left( \tau_n ,  Y_{n} \right)  }
{ d_{i,n} \left( \Delta \right) 
} \right)_{i}
\right\Vert_{ \mathbb{R}^d }^2
 +
\Delta  \sum_{k=1}^{m}
 \left\Vert
\left(
\frac{ \sigma^i_k \left(  \tau_n ,  Y_{n}  \right)  }
{ \widetilde{d}_{i,n} \left( \Delta \right)  } \right)_{i}
\right\Vert_{ \mathbb{R}^d } ^2,
\end{align*}
since
$
\widetilde{ d }_{i,n} \left( \Delta \right)  \geq d_{i,n} \left( \Delta \right) > 0
$. 
Hence, the lemma follows from \eqref{eq:4.10}.
\end{proof}

\section{Conclusions} 
We introduce a new general methodology to choose automatically step-sizes of weak numerical schemes for SDEs,
with two main innovative components: i) the matching between the first conditional moments of embedded pairs of weak approximations is controlled by appropriate local discrepancy functions; and ii) the step-size selection process does not involve sampling random variables. 
Guided by the new methodology, two variable step-size weak schemes were derived with orders 1 and 2. 
Numerical experiments illustrate the effectiveness of these adaptive schemes and their capability to overcome instability issues of the conventional weak schemes with fixed step-size. 
Similar to ODEs, 
these experiments for SDEs reveal that the adaptive time-stepping strategies based on local error control perform better than those based on global error.

\section*{References}


\end{document}